\documentclass[12pt]{amsart}
\usepackage[T1]{fontenc}
\usepackage[utf8]{inputenc}
\usepackage{amsmath,amsfonts,amssymb,amsthm}
\usepackage{mathtools}
\usepackage{tensor}
\usepackage{pgf,tikz}
\usepackage{csquotes}
\usetikzlibrary{matrix,arrows,calc}
\usepackage{tikz-cd}
\usepackage{shuffle}
\usepackage{empheq}

\usepackage[
	backend=bibtex,
	style=alphabetic,
	minalphanames=3,
	maxnames=99,
	maxalphanames=4,
	giveninits=true,
	isbn=false,
	doi=true,
]{biblatex}
\usepackage{xurl}

\AtEveryBibitem{
	\clearfield{eprintclass}
	\clearlist{language}%
}

\usepackage{enumitem}

\usepackage{xurl}
\usepackage[svgnames]{xcolor}

\usepackage[
	colorlinks, 
	citecolor=DarkGreen,
	linkcolor=DarkBlue,
	bookmarksdepth=4
]{hyperref}

\usepackage{cleveref}

\usepackage[colorinlistoftodos]{todonotes}

\numberwithin{equation}{section}

\newtheorem{thm}{Theorem}
\newtheorem{prop}[thm]{Proposition}
\newtheorem{lemma}[thm]{Lemma}
\newtheorem{cor}[thm]{Corollary}
\newtheorem{conj}[thm]{Conjecture}

\theoremstyle{definition}
\newtheorem*{prop*}{Proposition}

\numberwithin{thm}{section}

\theoremstyle{remark}
\newtheorem{rem}[thm]{Remark}
\newtheorem{rmk}[thm]{Remark}

\newtheorem{example}[thm]{Example}

\theoremstyle{definition}
\newtheorem{defn}[thm]{Definition}

\theoremstyle{plain}

\theoremstyle{remark}

\theoremstyle{definition}

\AddToHook{env/defn/begin}{\crefalias{thm}{defn}}
\AddToHook{env/prop/begin}{\crefalias{thm}{prop}}
\AddToHook{env/lemma/begin}{\crefalias{thm}{lemma}}
\AddToHook{env/conjecture/begin}{\crefalias{thm}{conjecture}}
\AddToHook{env/cor/begin}{\crefalias{thm}{cor}}
\AddToHook{env/conj/begin}{\crefalias{thm}{conj}}
\AddToHook{env/rem/begin}{\crefalias{thm}{rem}}
\AddToHook{env/rmk/begin}{\crefalias{thm}{rmk}}
\AddToHook{env/remarks/begin}{\crefalias{thm}{remarks}}
\AddToHook{env/example/begin}{\crefalias{thm}{example}}

\newcommand{\bA}{\mathbb A}
\renewcommand{\AA}{\mathbb A}

\newcommand{\Bb}{\mathcal{B}}
\newcommand{\CC}{\mathbb C}

\newcommand{\rd}{\mathrm d}
\newcommand{\rD}{\mathrm D}
\newcommand{\dR}{\mathrm{dR}}
\newcommand{\cE}{\mathcal E}
\newcommand{\Fib}{\mathrm{Fib}}
\newcommand{\FF}{\mathbb F}
\newcommand{\bF}{\mathbb F}
\newcommand{\cF}{\mathcal{F}}
\newcommand{\cG}{\mathcal{G}}

\newcommand{\bG}{\mathbb G}
\newcommand{\Hh}{\mathcal{H}}

\newcommand{\gp}{{\mathrm{gp}}}

\newcommand{\Isom}{\mathrm{Isom}}
\newcommand{\cH}{\mathcal{H}}
\newcommand{\rH}{\mathrm H}
\newcommand{\Ind}{\mathrm{Ind}}

\newcommand{\fn}{\mathfrak n}
\newcommand{\MT}{\mathbf{MT}}
\newcommand{\MHS}{\mathbf{MHS}}
\newcommand{\rMT}{\mathrm{MT}}
\newcommand{\mot}{\mathrm{mot}}
\newcommand{\PL}{{\mathrm{PL}}}
\newcommand{\bP}{\mathbb P}
\newcommand{\PP}{\mathbb P}
\newcommand{\Pro}{\mathrm{Pro}}
\newcommand{\bQ}{\mathbb Q}
\newcommand{\bR}{\mathbb R}
\newcommand{\bZ}{\mathbb Z}
\newcommand{\cI}{\mathcal{I}}
\newcommand{\cO}{\mathcal{O}}

\newcommand{\Sh}{\mathrm{Sh}}

\newcommand{\fu}{\mathfrak{u}}
\newcommand{\ur}{\mathrm{ur}}
\newcommand{\Un}{\mathrm{Un}}
\newcommand{\univ}{\mathrm{univ}}
\newcommand{\cV}{\mathcal{V}}
\newcommand{\hwt}{\mathrm{hwt}}
\newcommand{\rZ}{\mathrm{Z}}
\newcommand{\fz}{\mathfrak z}

\newcommand{\llangle}{\langle\!\langle}
\newcommand{\rrangle}{\rangle\!\rangle}

\newcommand{\Bcris}{\mathrm{B_{cris}}}
\newcommand{\et}{\operatorname{\acute{e}t}}
\newcommand{\logBK}{\log_{\mathrm{BK}}}
\newcommand{\lto}{\longrightarrow}

\mathchardef\mhyphen="2D %

\DeclareMathOperator{\ab}{{\mathrm{ab}}}

\DeclareMathOperator{\ad}{ad}
\DeclareMathOperator{\Ad}{Ad}
\DeclareMathOperator{\Aut}{Aut}
\DeclareMathOperator{\can}{can}

\DeclareMathOperator{\cris}{cris}
\DeclareMathOperator{\Der}{Der}
\DeclareMathOperator{\Ext}{Ext}

\DeclareMathOperator{\gr}{gr}

\DeclareMathOperator{\red}{red}

\DeclareMathOperator{\loc}{loc}
\DeclareMathOperator{\Li}{Li}
\DeclareMathOperator{\Lie}{Lie}

\DeclareMathOperator{\Res}{Res}

\DeclareMathOperator{\Hom}{Hom}
\DeclareMathOperator{\per}{per}

\DeclareMathOperator{\Sel}{Sel}
\DeclareMathOperator{\Spec}{Spec}

\DeclareMathOperator{\Vect}{Vect}

\newcommand{\new}{\newcommand}
\new{\dmo}{\DeclareMathOperator}
\new{\bb}{\mathbb}
\new{\fk}{\mathfrak}
\new{\ZZ}{\bb{Z}}
\new{\QQ}{\bb{Q}}
\new{\ka}{\kappa}
\new{\m}{\mathrm}
\new{\ep}{\epsilon}
\new{\al}{\alpha}
\new{\be}{\beta}
\new{\ga}{\gamma}
\new{\de}{\delta}
\new{\ze}{\zeta}
\new{\la}{\lambda}
\new{\si}{\sigma}
\new{\up}{\upsilon}
\new{\om}{\omega}
\new{\NN}{\bb{N}}
\new{\cl}{\mathcal}
\new{\Oo}{\cl{O}}
\new{\Pp}{\cl{P}}
\new{\uU}{\fk{u}}
\new{\nN}{\fk{n}}
\new{\Uu}{\cl{U}}
\new{\Ee}{\cl{E}}

\new{\opnm}{\operatorname}
\new{\wt}{\opnm{wt}}
\new{\zeu}{\ze^\uU}
\new{\pp}{\rangle}
\new{\qq}{\langle}

\new{\note}[1]{
{#1}}

\new{\from}{\leftarrow}

\new{\Aa}{\mathcal{A}}
\new{\Si}{\Sigma}
\new{\Kk}{\mathcal{K}}
\new{\hand}{\quad \text{and} \quad}
\new{\surj}{\twoheadrightarrow}
\new{\jrus}{\twoheadleftarrow}
\new{\RR}{\mathbb{R}}

\newcommand{\logu}{\log^\uU}
\newcommand{\Liu}{\Li^\uU}

\newcommand{\myCol}{\operatorname{Col}}

\usepackage[all]{xy}

\newcommand{\UZ}{U_Z}

\newcommand{\Fb}{{\mathbb F}}

\newcommand{\Pb}{{\mathbb P}}

\newcommand{\Zb}{{\mathbb Z}}
\newcommand{\Qb}{{\mathbb Q}}

\newcommand{\Ic}{{\mathcal I}}

\newcommand{\Oc}{{\mathcal O}}

\newcommand{\Free}{\mathrm{Free}}

\newcommand{\nf}{\mathfrak{n}}

\newcommand{\Gm}{{\mathbb{G}_m}}
\newcommand{\Qp}{{\QQ_p}}
\newcommand{\Zp}{{\ZZ_p}}

\newcommand{\ev}{\mathrm{ev}}

\newcommand{\Poneminusthreepoints}{\PP^1\smallsetminus\{0,1,\infty\}}

\newcommand{\real}{\mathrm{real}}

\title{Mixed Tate Chabauty--Kim beyond the polylog quotient}
\author{David Corwin, Ishai Dan-Cohen, Martin Lüdtke}
\thanks{\textbf{Grant acknowledgements:} I.D. was supported by ISF grants 726/17 and 621/21. 
M.L. was supported by a Minerva Fellowship of the Minerva Stiftung Gesellschaft für die Forschung mbH and also acknowledges support through a guest postdoc fellowship at the Max Planck Institute for Mathematics in Bonn and through an NWO Grant, project number VI.Vidi.192.106, in the group of Steffen Müller at the University of Groningen.
\textbf{Political disclaimer:} I. Dan-Cohen asks not to be considered responsible for the actions of any government that does not fully embrace the principles of democracy. }

\begin{document}

\begin{abstract}
We develop an algorithmic approach to mixed Tate Chabauty--Kim theory for the full unipotent fundamental group of $\Poneminusthreepoints$, beyond the polylogarithmic quotient used in previous work. The key step is computing the cocycle evaluation map in the presence of a nontrivial motivic Galois action, using relations among motivic multiple zeta values. This yields explicit Kim functions involving multiple polylogarithms, as well as an algorithm for solving the $S$-unit equation whose halting is implied by certain natural conjectures. As an application, we prove new instances of Kim's Conjecture, including the first examples with $\#S = 2$.
\end{abstract}

\maketitle

\thispagestyle{empty}
\tableofcontents

\section{Introduction}

For $(X,x)$ a flat and finite type scheme over an open subscheme $Z$ of $\Spec \ZZ$ (perhaps assumed to have smooth generic fiber) the rational pro-unipotent completion $\pi_1^{\m{Betti}}(X,x)$ of the fundamental group of $X(\CC)$ is conjectured to come from a pro-unipotent group object $\pi_1^\mot(X,x)$ of the Tannakian category $\opnm{MM}(Z, \QQ)$ of mixed motives. This conjectural pro-unipotent group object of a conjectural $\QQ$-Tannakian category is referred to as the \emph{unipotent fundamental group}. While the correct construction of $\opnm{MM}(Z, \QQ)$ has yet to be understood, the various realizations of the unipotent fundamental group (at least in special cases), as well as the expected comparison isomorphisms between them, have been studied intensely for many years. Particularly when $X$ is a suitable model of a hyperbolic curve, and in the presence of a chosen closed point $p \in Z$, Chabauty--Kim theory uses these realizations, and the interplay between them, to construct obstructions for a $\Zp$-point of $X$ to come from a $Z$-point. In turn, these obstructions shed new light on the set $X(Z)$ of \emph{$Z$-integral points} of $X$ in various ways. In one direction, they provide a view of integral points which has a geometric or homotopical flavor, and help to connect their study with core aspects of arithmetic such as algebraic K-theory, $p$-adic Galois cohomology, $p$-adic Hodge theory, $p$-adic periods, $p$-adic heights, and syntomic regulators. In another direction, they go a certain distance towards an ideal of understanding which is \textit{explicit}, \textit{computational}, and even \textit{algorithmic}. 

In the literature on Chabauty--Kim theory, two separate strands (of rather unequal size) may be discerned. In one strand, to which the Quadratic Chabauty method belongs, one restricts attention to very small finite-type quotients of $\pi_1^\mot(X,x)$ \cite{BD18, BD-II, BD19, BDMTV19, BDMTV2}. In the other strand, to which the present article belongs, we instead restrict attention to very special $X$; most often, as will be the case starting now, and throughout the remainder of this article,
\[
X = \Poneminusthreepoints.
\]
This case has proved a fruitful testing ground for Kim's program \cite{CKone, MTMUE, MTMUEII, BrownIntegral, CDC:polylog1, PolGonII, BBKLMQSX, BKL:chabauty_kim_sc, luedtke:refined, GeometricStep, Luedtke:motivic_selmer_scheme, LL:PolylogNF, KLL:cyclotomic}. 
Nevertheless, it has been somewhat of an embarrassment that up until now we've had to restrict attention to the so-called \textit{polylogarithmic quotient}
\[
\pi_1^\mot(X) \surj \pi_1^\mot(X)_{\PL}.
\]
In the present setting, the unipotent fundamental group is not merely a conjectural object of a conjectural category, but is well defined as a pro-unipotent group object of the Tannakian category $\opnm{MT}(Z, \QQ)$ of mixed Tate motives.\footnote{This means, for instance, that its Lie algebra is endowed with the structure of a pro-object of $\opnm{MT}(Z, \QQ)$; see \cite[§5]{Deligne89} for generalities concerning the notion of unipotent group object in a Tannakian category.} Unlike $\pi_1^\mot(X)$ itself, whose Lie algebra is quite complicated as a motive, if we set our basepoint $x$ equal to a $\ZZ$-integral tangent vector, then $\Lie \pi_1^\mot(X)_{\PL}$ is merely a product of Tate motives, and is in particular semisimple. Previous works on this topic have depended on this semisimplicity quite heavily. 

Chabauty--Kim theory associates to every motivic quotient
\[
\pi_1^\mot(X) \surj \Pi
\]
of the unipotent fundamental group, a subset $X(\Zp)_{Z,\Pi} \subseteq X(\Zp)$ of the $p$-adic points, called \emph{Chabauty--Kim locus}, which is guaranteed to contain the set of global points $X(Z)$, thus forming a sandwich
\[
X(Z) \subseteq X(\Zp)_{Z,\Pi} \subseteq X(\Zp).
\]
These sandwiches are further arranged so that if
\[
\Pi \surj \Pi'
\]
is a quotient of $\Pi$, then
\[
X(\Zp)_{Z,\Pi} \subseteq X(\Zp)_{Z,\Pi'}.
\]
The set $X(\Zp)_{Z,\Pi}$ is constructed as the vanishing locus of an ideal
\[
\Ic_{Z,\Pi} \subseteq \opnm{Col}(X(\Zp))
\]
in the ring of Coleman functions \cite{Besser}. We refer to elements of $\Ic_{Z,\Pi}$ as \textit{Kim functions} associated to $Z$ and $\Pi$. In the present setting, general Kim functions are polynomials in $p$-adic multiple polylogarithms
\[
\Li_{k_1, \dots, k_m}(z)
\]
\cite{FurushoI, FurushoII} with $\Qp$-coefficients; Kim functions associated to quotients that factor through the polylogarithmic quotient are polynomials in the \textit{classical polylogarithms} $(m=1)$. Thus, by going beyond the polylog quotient we obtain the first concrete examples of Kim functions for the thrice-punctured line that involve terms with $m >1$. For instance, combining \Cref{4functions} and \Cref{periods-Z-one-half} we find:

\begin{thm}
Let $Z = \Spec \ZZ[1/2]$, let $\pi_1^\mot(X) \surj \Pi$ be the level-4 quotient, i.e., the quotient by the fifth step in the descending central series, let $p$ be any odd prime, let
\[
a = \frac87 \Li_4(2),
\quad
\text{and let}
\quad
b 
= \frac47 \Li_{3,1}(-1).
\]
Then the function 
\begin{equation*}
	\begin{split}
		&\log(2) \zeta(3) \left( \Li_{3,1}(z) + \tfrac18 \log(z)^2 \Li_1(z)^2\right) \\
		&\qquad - a \log(z) \Li_{2,1}(z) - b \Li_1(z) \Li_3(z) + b \zeta(3) \Li_1(z)
	\end{split}
\end{equation*}
belongs to the Kim ideal $\Ic_{Z,\Pi}$ which defines the locus $X(\Zp)_{Z,\Pi}$.
\end{thm}

Here and below, multiple polylogarithms and multiple zeta values all refer to the $p$-adic versions of Furusho (but note that our $\Li_{k_1,\ldots,k_m}$ is Furusho's $\Li_{k_m,\ldots,k_1}$, see footnote \ref{footnote:furusho}). We also find completely explicit level~4 Kim functions in the cases $Z = \Spec \bZ[1/3]$ (\Cref{4functions}+\Cref{periods-Z-one-third}) and $Z = \Spec \bZ[1/6]$ (\Cref{kim-function-S=2}+\Cref{periods-Z-one-sixth}), using computations in SageMath \cite{sagemath} in the latter case. This allows us to verify new instances of Kim's Conjecture \cite[§8.1]{nabsd}. The conjecture states that 
\[
X(Z) = X(\Zp)_{Z,\Pi} \quad
\text{for all sufficiently large $\Pi$},
\]
that is, the method eventually cuts out the exact set of $Z$-integral points.

\begin{thm}[= Theorems~\ref{kims-conj-Z-one-half} \& \ref{kims-conj-Z-one-third} \& \ref{kims-conj-Z-one-sixth}]
	\label{kim-conj-thm-intro}
	Let $\pi_1^{\mot}(X) \surj \Pi$ be the level-4 quotient. We have $X(Z) = X(\Zp)_{Z,\Pi}$ for all primes $p < 1000$, $p \in Z$, in each of the following cases:
	\begin{enumerate}[label=(\alph*)]
		\item $Z = \Spec \bZ[1/2]$;
		\item $Z = \Spec \bZ[1/3]$;
		\item $Z = \Spec \bZ[1/6]$. \label{item:Z-one-sixth}
	\end{enumerate}
\end{thm}

The Sage code used for verifying \Cref{kim-conj-thm-intro} is available at the GitHub repository \url{https://github.com/martinluedtke/BeyondPolylog}.

Theorem~\ref{kim-conj-thm-intro}\,\ref{item:Z-one-sixth} provides the first instances where Kim's Conjecture is verified for $Z = \Spec \bZ[1/S]$ with $S$ containing two primes. In comparison, being limited to using only the polylogarithmic quotient of the fundamental group, the best known result so far was the construction of a huge Kim function in level~6, which could moreover not be made completely explicit due to a missing \emph{arithmetic step} \cite{DCJ:M05}. Our Kim functions are completely explicit and of a much more manageable size.

In higher level, the advantage of working with the full fundamental group instead of the polylogarithmic quotient becomes even more significant. For each level $n \geq 1$, we can use the dimension formulas from \cite[§7]{Luedtke:motivic_selmer_scheme} to determine the maximal number $\# S$ of inverted primes $S = \Spec(\bZ) \smallsetminus Z$ for which the existence of a non-trivial Kim function (thus, finiteness of the Chabauty--Kim locus) can be guaranteed in level~$n$ through a dimension inequality between the global and local Selmer scheme. The results of this comparison between the full level-$n$ quotient and the polylogarithmic level-$n$ quotient are shown in the following table.

\begin{center}
	\begin{tabular}{|c||c|c|c|c|c|c|c|c|}
		\hline
		level & 1 & 2 & 4 & 6 & 8 & 10 & 12 & 14 \\
		\hline
		$\# X(\bZ_p)_{Z,n}$ finite for $\#S \leq $ & 0 & 1 & 2 & 7 & 22 & 71 & 239 & 819 \\
		$\# X(\bZ_p)_{Z,\PL,n}$ finite for $\#S \leq $ & 0 & 1 & 1 & 2 & 2 & 3 & 3 & 4 \\
		\hline
	\end{tabular}
\end{center}

(Odd levels $\geq 3$ are omitted because level $2k+1$ gives the same bound on $\#S$ as level $2k$.)

In other concrete examples, we turn our attention to the \textit{refined} Chabauty--Kim locus
\[
X(\Zp)_{Z,\Pi}^\m{min} \subseteq X(\Zp)_{Z,\Pi}
\]
of Betts--Dogra \cite{BD:refined}. Continuing with the level-4 quotient $\pi_1^\mot(X) \surj \Pi$, we're able to completely compute the case $Z = \Spec \ZZ[1/6]$ (see §\ref{sec:S=2-refined}--§\ref{sec:Z-one-sixth}). 
The refined locus breaks up into a union, parametrized (in the present case) by $\Si \in \{0,1,\infty\}^2$, and for another concrete example, we find that the functions
\[
		\log(z), \quad
		\Li_2(z), \quad
		\Li_4(z), \quad
		\text{and}
\]
\[
		(\Li_{2,1}(z) - \zeta(3)) \Li_{3,1}(z) - (\Li_3(z) - \zeta(3))\Li_{2,1,1}(z),
\]
vanish on the portion associated to $\Si = (1,1)$. For $\Sigma = (1,0)$, four explicit Kim functions are derived in \Cref{thm:kim_functions_10} (+ \Cref{periods-Z-one-sixth}), extending the list of two polylogarithmic Kim functions from \cite{luedtke:refined}.

Finally, we investigate the implications of being able to work with general fundamental group quotients on the construction of a Chabauty--Kim-theoretic algorithm for solving the $S$-unit equation. Previous works led, among other things, to the construction of an algorithm which takes $Z \subseteq \Spec \ZZ$ as input, and, upon halting, is proved to produce the finite set $X(Z)$ \cite{MTMUEII,PolGonII}. This algorithm, however, comes with a caveat: its halting depends on three conjectures. Two of them are fairly well established. The third, however, \emph{Kim's Conjecture}, comes with a further caveat.
Because we were limited to working with the polylogarithmic quotient, we were forced to rely on a somewhat ad hoc strengthening of this conjecture, in which we restrict attention to quotients that factor through the polylogarithmic quotient. However, recent work by X.\ Li and one of the authors (M.L.) in the number field setting has produced examples where this stronger conjecture fails \cite[§6.4]{LL:PolylogNF}.

One of our goals here is to get rid of this further caveat. In doing so, we find that we must assume yet another conjecture in order to prove that our algorithm halts. In fact, this newly added member of our list of dependencies is an intersection of several conjectures concerning relations between multiple zeta values; we refer to it as the ``syncretist's multiple zeta value conjecture.'' The resulting \emph{Equation solving theorem} goes as follows.

\begin{thm}
\label{30305A}
Let $Z$ be an open subscheme of $\Spec \ZZ$.
\begin{enumerate}
\item
Suppose the \textit{Equation solving algorithm} (\ref{30302B}) halts. Then the output is equal to the set $X(Z)$ of $Z$-integral solutions to the unit equation. 
\item
Suppose \textit{Kim's Conjecture}, \textit{Goncharov's conjecture refined with respect to ramification}, the \textit{$p$-adic period conjecture}, and the \textit{syncretist's motivic multiple zeta value conjecture} hold. Then the \textit{Equation solving algorithm} halts.
\end{enumerate}
\end{thm}

The value of this exchange of conjectures needs explaining. At the time of writing, Kim's Conjecture, which was formulated in \cite{nabsd}, is still quite new. While surely the conjecture is backed up by some philosophical motivation, evidence for it is far from decisive. All the more so if we must rely on a strengthened version, which, as currently formulated, applies only to the thrice-punctured line. To make things worse, as explained in the introduction to \cite{PolGonII}, this strengthened version, where we restrict attention to the polylogarithmic quotient, traveled down a bumpy road, from which it emerged with a bruise and a bandaid. In light of the counterexamples from \cite{LL:PolylogNF} in the number field setting, this strengthened version of Kim's Conjecture now appears to have been too optimistic.

By contrast, several of the conjectures incorporated into our syncretist's multiple zeta value conjecture are old and widely believed to be true. Moreover, if we apply them to \textit{motivic} multiple zeta values (instead of \textit{complex} or \textit{$p$-adic} multiple zeta values), then their content is purely combinatorial, which makes them somewhat less mysterious than Kim's Conjecture. Thus, by returning to Kim's \textit{actual} conjecture, even at the cost of adding in this additional conjecture, we place our algorithm, as well as our ensuing \textit{Equation solving theorem} concerning its halting, on more solid ground. In fact, this new dependence on relations between multiple zeta values, while weakening our theorem in one way, adds interest in a different way, since it helps to connect different areas of investigation, and gives new motivation for the study of multiple zeta values. 

The Tannakian fundamental group $G_Z$ of the category of mixed Tate motives at the de Rham fiber functor decomposes canonically as an extension of $\Gm$ by a free pro-unipotent group $U_Z$. The $p$-adic period map \cite{ChatUnv} provides a $\Qp$-point
\[
\eta_p: \Spec \Qp \to U_Z.
\]
Goncharov's theory of motivic polylogarithms gives rise to certain special functions on $U_Z$ whose pullbacks along $\eta_p$ are values of $p$-adic polylogarithms. At any moment in time, we're only interested in some finitely generated portion $U_Z'$ of $U_Z$. To the extent that the given portion is spanned by motivic polylogarithms (as is typically implied by slight refinements of Goncharov's conjectures), we're able to use these to algorithmically construct coordinates on $U_Z'$. In turn, these give rise to free generators for the Lie algebra, which finally lead to a set of coordinates of the affine space $\rZ^1(U_Z, \Pi^\dR)^\Gm$ of $\Gm$-equivariant cocycles. 

With these in hand, the central problem we face here is to compute the scheme-theoretic image of the evaluation map
\[
\rZ^1\big( U_Z, \Pi^\dR\big)^\Gm \times U_Z
\xrightarrow{\ev}
\Pi^\dR \times U_Z
\]
\[
\ev(c,\gamma) = (c\gamma, \gamma)
\]
in terms of the given coordinates. When $\Pi$ is semisimple, cocycles are just homomorphisms, and this computation amounts to a purely geometric problem (and is referred to as the \textit{geometric step} in the computation). With $\Pi$ no longer semisimple, the cocycle condition now involves the action of $U_Z$ on $\Pi^\dR$, and so arithmetic aspects intrude upon the geometric step. Our main discovery in the present paper is that these can be managed algorithmically and computationally using the calculus of motivic multiple zeta values.

\subsection*{AI statement}
This paper contains no AI-generated text and the main results were obtained before AI became useful in mathematics. ChatGPT-6 was used for proofreading.

\section{Background on Mixed Tate Chabauty--Kim}

We review the relevant technical background material. We assume familiarity with the basics of Kim's program \cite{KimP1,kim09}\footnote{See \cite{KimNonLinear}, \cite{KimFGDG}, \cite{KimGTDG}, and \cite{ChabautytoKim} for introductions to Kim's program.} and review the application of mixed Tate motives to the $S$-unit equation (as developed in \cite{HadianDuke,CKone,MTMUE,BrownIntegral,CDC:polylog1}).

Throughout this paper, $X=\Poneminusthreepoints$ over $\mathbb{Z}$.

\subsection{Free pro-unipotent groups}

We start by recalling some facts about free pro-unipotent groups and their associated Lie algebra and cocommutative Hopf algebra. The importance of these is twofold in that both the de Rham fundamental group of $\Poneminusthreepoints$ and the unipotent mixed Tate Galois group of a dense open subscheme $Z \subseteq \Spec(\bZ)$ are examples of free pro-unipotent groups.

\subsubsection{Profinite-dimensional vector spaces}
\label{sec:profinite-dimensional-vector-spaces}

Let $k$ be a field of characteristic zero. Let $\Vect_k$ be the category of all $k$-vector spaces (of arbitrary dimension) and let $\opnm{ProFinVect}_k$ be the category of \emph{profinite-dimensional} vector spaces. The latter is equivalently defined as either the pro-category of finite-dimensional vector spaces, or as the category of topological vector spaces which are isomorphic to a filtered limit of finite-dimensional discrete vector spaces, with continuous linear maps as morphisms. We usually adopt the topological vector space viewpoint.
There is an equivalence of categories 
\begin{equation}
	\label{eq:vector-space-duality}
	\Vect_k^\mathrm{op} \simeq \opnm{ProFinVect}_k
\end{equation}
given as follows. In one direction, we send a $k$-vector space $V$ to $V^\lor = \Hom_{\Vect_k}(V, k)$ with topology induced by the isomorphism
\[
V^\lor \cong \lim_{\underset{\text{fin. dim.}}{U\subseteq V}} U^\lor.
\]
In the other direction, we send a profinite-dimensional vector space $W$ to its continuous dual $W^\lor = \Hom_{\opnm{ProFinVect}_k}(W, k)$. The equivalence~\eqref{eq:vector-space-duality} is monoidal if $\opnm{ProFinVect}_k$ is endowed with the completed tensor product:
\[ (\varprojlim_i V_i) \hat\otimes (\varprojlim_j W_j) \coloneqq \varprojlim_{i,j} (V_i \otimes W_j). \]

\subsubsection{Free pro-nilpotent Lie algebras and the Lyndon basis}
\label{sec:lyndon-basis}
By a \emph{nilpotent} Lie algebra over~$k$ we always mean a \emph{finite-dimensional} nilpotent Lie algebra. A \emph{pro-nilpotent Lie algebra} is a pro-object of the category of nilpotent Lie algebras, or equivalently a topological Lie algebra which is an inverse limit of nilpotent Lie algebras.
Thus, there is an equivalence between unipotent groups and nilpotent Lie algebras, and between pro-unipotent groups and pro-nilpotent Lie algebras.

The forgetful functor from pro-nilpotent Lie algebras to $\opnm{ProFinVect}_k$ admits a left adjoint, so we can talk about the \emph{free pro-nilpotent Lie algebra} $\opnm{Free}_{\opnm{pro-nil}}(V)$ on a profinite-dimensional vector space~$V$. It can be constructed as the subspace of primitive elements in the completed tensor algebra $TV = \prod_{n=0}^{\infty} T^n V$ (with respect to the completed tensor product), or as the inverse limit
\[ \opnm{Free}_{\opnm{pro-nil}}(V) = \varprojlim_i \varprojlim_n \opnm{Free}(V_i)/\fz^n \opnm{Free}(V_i), \]
with $V_i$ running through the finite-dimensional quotients of~$V$, $\opnm{Free}(V_i)$ denoting the ordinary free Lie algebra on~$V_i$, and $\fz^{\bullet}$ denoting the descending central series.

\begin{defn}
	\label{def:free-pronilpotent-lie-algebra}
	Let $\Sigma$ be a set. The \emph{free pro-nilpotent Lie algebra} on~$\Sigma$ over~$k$ is defined as the free pro-nilpotent Lie algebra on the profinite-dimensional vector space $k^{\Sigma}$:
	\[ \Free_{\opnm{pro-nil}}(\Sigma) \coloneqq \Free_{\opnm{pro-nil}}(k^{\Sigma}). \]
	We also denote this by $\fn(\Sigma)$.
\end{defn}

Thus, $\fn(\Sigma)$ is universal among pro-nilpotent Lie algebras $\fn$ with a continuous linear map $k^{\Sigma} \to \fn$. Here, $k^{\Sigma}$ carries the product topology, which makes it the inverse limit of the finite-dimensional discrete vector spaces $k^{\Sigma'}$ with $\Sigma'$ running through the finite subsets of~$\Sigma$. We identify elements of $\Sigma$ with their image in $\fn(\Sigma)$. If $\Free(\Sigma)$ denotes the ordinary free Lie algebra on~$\Sigma$, then $\fn(\Sigma)$ is given by
\[ \fn(\Sigma) = \varprojlim_{\Sigma' \subseteq \Sigma \text{ finite}} \varprojlim_n  \Free(\Sigma')/\fz^n \Free(\Sigma'). \]

We describe a basis of $\fn(\Sigma)$, i.e., a linearly independent subset generating each of the finite-dimensional quotients in this presentation. A generating set is given by all nested Lie brackets in the generators but these satisfy many relations. The \emph{Lyndon basis} is a particular subset which is linearly independent. Its definition depends on a choice of total ordering on~$\Sigma$. 

\begin{defn}
	\label{def: Lyndon word}
	Let $(\Sigma, <)$ be a totally ordered set. A word over~$\Sigma$ is a \emph{Lyndon word} if it is nonempty and lexicographically strictly smaller than all its nontrivial cyclic rotations.
\end{defn}

\begin{example}
	\label{lyndon-words-of-length-le-4}
	If $\Sigma = \{e_0,e_1\}$ with $e_0 < e_1$, then the Lyndon words of length at most 4 are
	\begin{align*}
		e_0, \quad e_1, \quad e_0e_1, \quad e_0e_0e_1, \quad e_0e_1e_1, \quad e_0e_0e_0e_1, \quad e_0e_0e_1e_1, \quad e_0e_1e_1e_1.
	\end{align*}
\end{example}

To each Lyndon word~$w$ over $\Sigma$ we associate an element $\lambda(w)$ of the free pro-nilpotent Lie algebra~$\fn(\Sigma)$. If $w = \sigma$ is a single letter, then $\lambda(\sigma) = \sigma$. Otherwise, we recursively define $\lambda(w) = [\lambda(u), \lambda(v)]$, where $w = uv$ is the \emph{standard factorization} of~$w$, i.e., $v$ is the longest proper suffix of~$w$ which is a Lyndon word, in which case the prefix~$u$ is also a Lyndon word. For example, if again $\Sigma = \{e_0,e_1\}$:
\[ \lambda(e_0e_0 e_1 e_1) = [e_0, \lambda(e_0e_1e_1)] = [e_0, [\lambda(e_0e_1), e_1]] = [e_0,[[e_0,e_1],e_1]]. \]
We call $\lambda(w)$ the \emph{standard bracketing} of~$w$.

\begin{prop}
	Let $\Sigma$ be a totally ordered set. The elements $\lambda(w)$ with $w$ running through the Lyndon words over~$\Sigma$ form a basis of the free pro-nilpotent Lie algebra on~$\Sigma$.
\end{prop}

\begin{proof}
	See for example \cite[§5.1]{reutenauer93}.
\end{proof}

\subsubsection{Free unipotent groups}
The free pro-nilpotent Lie algebra~$\fn(\Sigma)$ on a set~$\Sigma$ is the Lie algebra of the free pro-unipotent group~$U(\Sigma)$ on the elements $\exp(\sigma) \in U(\Sigma)(k)$ for $\sigma \in \Sigma$. The universal property of~$U(\Sigma)$ is that for any unipotent group~$V$ and any map $f\colon \Sigma \to V(k)$ mapping all but finitely many elements to the identity, there exists a unique homomorphism of pro-algebraic groups $g\colon U(\Sigma) \to V$ such that $g(\exp(\sigma)) = f(\sigma)$ for $\sigma \in \Sigma$.
The exponential map $\exp\colon \fn(\Sigma) \to U(\Sigma)$ is a pro-algebraic map over~$k$ when $\fn(\Sigma)$ is viewed as a pro-affine space in the obvious way. The exponential map is an isomorphism with inverse the logarithm $\log\colon U(\Sigma) \to \fn(\Sigma)$.

\subsubsection{The cocommutative Hopf algebra of an affine group scheme}
\label{260826A}

If $G$ is an affine group scheme over~$k$, then $\mathcal{O}(G)$ is a commutative Hopf algebra, i.e.\ a commutative Hopf algebra object in $\Vect_k$; its dual $\mathcal{H}(G) \coloneqq \cO(G)^{\vee}$ is thus a cocommutative Hopf algebra object in $\opnm{ProFinVect}_k$ via~\eqref{eq:vector-space-duality}.

\begin{defn}
	\label{260826B}
	We define a \emph{complete Hopf algebra} to be a Hopf algebra object of $\opnm{ProFinVect}_k$. For $G$ an affine group scheme over~$k$, the \emph{(complete) cocommutative Hopf algebra associated to $G$} (also called the \emph{convolution algebra} in the literature) is defined by 
	\[ \Hh(G) := \Oo(G)^\lor. \]
\end{defn}

Note that the formation $G \mapsto \cH(G)$ commutes with inverse limits, so if $G = \varprojlim_i G_i$ with linear algebraic groups $G_i$, then $\cH(G) = \varprojlim_i \cH(G_i)$. 

\begin{rmk}
	If $G \to \Spec k$ is the finite constant group scheme associated to an abstract finite group, then $\Hh(G) = k[G]$ is just the group algebra. 
	By contrast, if $G = \varprojlim_i G_i$ is an inverse limit of unipotent groups, then $\Hh(G)$ is canonically isomorphic to the completed universal enveloping algebra
	\[
	\Hh(G) \cong
	\varprojlim_i \varprojlim_n \Uu \Lie(G_i)/I_i^n,
	\]
	where $I_i$ denotes the augmentation ideal of $\Uu \Lie G_i$ (see appendix A of Deligne--Goncharov \cite{DelGon05}). This isomorphism plays an important role elsewhere in Chabauty--Kim theory, but will not intervene here. 
\end{rmk}

\subsubsection{The cocommutative Hopf algebra associated to a free pro-unipotent group}
\label{sec:hopf-algebra}
Let~$\Sigma$ be a set and $k$ a field of characteristic zero. Let $\fn = \fn(\Sigma)$ be the free pro-nilpotent Lie algebra on~$\Sigma$, and $U = U(\Sigma)$ the corresponding free pro-unipotent group over~$k$. We can write $U = \varprojlim_i U_i$ as a limit of finite-dimensional unipotent groups by taking the limit over finite subsets $\Sigma' \subseteq \Sigma$ and taking quotients along the descending central series. We describe its associated complete Hopf algebra 
\[ \cH(U) = \varprojlim_i \varprojlim_{A \subseteq \cO(U_i) \text{ f.d.}} A^{\vee}. \]
This is canonically isomorphic to the non-commutative power series ring
\[ \cH(U) = k\llangle \Sigma \rrangle, \]
whose underlying vector space carries a profinite-dimensional topology by writing $k\llangle \Sigma \rrangle$ as a limit over finite subsets of $\Sigma$ and quotienting by powers of the augmentation ideal. Elements of $k\llangle \Sigma \rrangle$ have the form $\sum_w a_w \, w$ with $a_w \in k$ and $w$ running through the words over~$\Sigma$. This includes elements like $\sum_{\sigma \in \Sigma} \sigma$ if $\Sigma$ is infinite. 
The coproduct on $\cH(U)$ is uniquely determined by 
\[ \Delta(\sigma) = 1 \otimes \sigma + \sigma \otimes 1, \]
for $\sigma \in \Sigma$, the counit $\varepsilon\colon \cH(U) \to k$ sends a power series to its constant coefficient, and the antipode is given on monomials by
\[ S(\sigma_1\cdots \sigma_n) = (-1)^n \sigma_n \cdots \sigma_1. \]

When $R$ is a $k$-algebra and $R \, \hat \otimes \, -$ denotes the tensor product in the sense of pro-objects\footnote{Over a field, it is safe to identify the pro-category of finite-dimensional vector spaces with the category of profinite-dimensional topological vector spaces, but over more general rings, it is better to view $R\llangle \Sigma\rrangle$ as shorthand for the pro-object given by taking the limit over finite subsets of~$\Sigma$ and truncating by powers of the augmentation ideal.}, we have $R \, \hat\otimes \, \cH(U) = R\llangle \Sigma \rrangle$ and its elements are $k$-linear maps $\cO(U) \to R$. The $R$-valued points $U(R)$ correspond to $k$-linear maps which are in addition multiplicative, so $U(R)$ can be identified with the \emph{grouplike} elements in $R\llangle \Sigma\rrangle$: elements $u$ satisfying $\varepsilon(u) = 1$ and $\Delta u = u \otimes u$. We write this as
\[ U(R) = R\llangle \Sigma\rrangle^{\gp}. \]

Similarly, the Lie algebra $\fn \, \hat \otimes \, R$ is the vector space of derivations $D\colon \cO(U) \to R$, i.e., $k$-linear maps satisfying in addition the rule $D(1) = 0$ and $D(fg) = \varepsilon(f) D(g) + \varepsilon(g) D(f)$ for $f,g \in \cO(U)$, where $\varepsilon$ is the augmentation of~$\cO(U)$. Thus, $\fn\, \hat\otimes\, R$ can be identified with the set of \emph{Lie-like} (or \emph{primitive}) elements in $R\llangle\Sigma\rrangle$, i.e., those elements~$X$ satisfying $\varepsilon(X) = 0$ and $\Delta X = 1 \otimes X + X \otimes 1$:
\[ \fn \, \hat\otimes \, R = R\llangle \Sigma\rrangle^{\Lie}. \]
In this way, the cocommutative Hopf algebra $\cH(U)$ serves as a common superstructure of both the group $U$ and its Lie algebra $\fn$. The group operation on $U(R)$ corresponds to multiplication of grouplike elements in $\cH(U)$; the Lie bracket on $\fn \, \hat\otimes\, R$ corresponds to the commutator bracket $[a,b] = ab - ba$ on Lie-like elements in $\cH(U)$.

The ring of functions $\cO(U)$ can be recovered from $\cH(U)$ as the (topological) dual space
\begin{equation}
	\label{eq:functions-dual-to-hopf-algebra}
	\cO(U) = \cH(U)^\vee.
\end{equation}
Elements of $\cO(U)$ are continuous $k$-linear functionals $k\llangle\Sigma\rrangle \to k$, or equivalently finite linear combinations of coefficient extraction functionals.

\begin{defn}
	\label{def:coeff-extraction-functions}
	Given a word $w$ in~$\Sigma$, the \emph{coefficient extraction functional}
	\[ f_w\colon k\llangle \Sigma \rrangle \to k \]
	is defined as the linear map sending a non-commutative power series to its $w$-coefficient. We view $f_w$ as an element of~$\cO(U)$ via~\eqref{eq:functions-dual-to-hopf-algebra}.
\end{defn}

The product of two coefficient extraction functionals is given by the \emph{shuffle product}:
\begin{equation}
	\label{eq:shuffle-product}
	f_{w_1} f_{w_2} = \sum_{\pi \in \Sh(|w_1|,|w_2|)} f_{\pi(w_1 w_2)}
\end{equation}
for all words $w_1,w_2$ on~$\Sigma$, where $\Sh(n,m) \subseteq S_{n+m}$ denotes the set of all permutations satisfying $\pi^{-1}(1) < \ldots < \pi^{-1}(n)$ and $\pi^{-1}(n+1) < \ldots < \pi^{-1}(n+m)$, and $\pi(\sigma_1\cdots\sigma_{n+m}) \coloneqq \sigma_{\pi(1)}\cdots\sigma_{\pi(n+m)}$ for $\sigma_i \in \Sigma$.
It follows that $\cO(U)$ is the free $k$-algebra on the elements $f_w$ with $w$ running through the Lyndon words over~$\Sigma$, for any choice of total ordering on~$\Sigma$ \cite[Theorem~6.1]{reutenauer93}.

\subsection{The de Rham fundamental group of~$X$}
\label{sec: de Rham fundamental group}

\subsubsection{Basepoints}
By a basepoint on~$X = \bP^1\smallsetminus \{0,1,\infty\}$ we mean either a point in $X(\bQ)$ or a tangential basepoint in the sense of \cite[§15]{Deligne89}, i.e., a nonzero tangent vector at one of the three cusps. If $t$ denotes the standard coordinate on $\bA^1$, then for $\lambda \in \bQ^\times$ we denote by $\vec{\lambda}_0$ the tangential basepoint at~0 with $\mathrm{d}t(\vec{\lambda}_0) = \lambda$, and by $\vec{\lambda}_1$ the one at~1 with $\mathrm{d}t(\vec{\lambda}_1) = \lambda$. We will often use the particular tangential basepoints $\vec{1}_0$ and $(\vec{-1})_1$, which usually we write simply as~$0$ and~$1$.

\subsubsection{De Rham fundamental group and path spaces}
Let $\Un(X)$ be the $\bQ$-linear Tannakian category of unipotent vector bundles with flat connection on $X_{\bQ}$. Let $x$ and $y$ be two basepoints. They define fiber functors $\Fib_x, \Fib_y \colon \Un(X) \to \Vect_{\bQ}$ and the unipotent \emph{de Rham path space} $\pi_1^{\dR}(X; x,y)$ is defined as the scheme of tensor isomorphisms
\[ \pi_1^{\dR}(X; x,y) \coloneqq \underline{\Isom}^{\otimes}(\Fib_x, \Fib_y). \]
In the case $x = y$ it is denoted by $\pi_1^{\dR}(X,x)$ and called the unipotent \emph{de Rham fundamental group} of~$X$. It is a pro-unipotent group over $\bQ$. When we write $\pi_1^{\dR}(X)$ without specified basepoint, we are referring to the tangential basepoint $0 = \vec{1}_0$.

\subsubsection{Canonical de Rham paths}
\label{sec:canonical-paths}

For any two basepoints $x$, $y$ on~$X$ there is a canonical de Rham path 
\begin{equation}
	\label{eq:canonical-path}
	\gamma_{x,y}^{\dR} \in \pi_1^{\dR}(X; x,y)(\bQ).
\end{equation}
It can be constructed as follows. Let $\overline{X} \coloneqq \bP^1$ be the smooth completion of $X$, and let $D \coloneqq \{0,1,\infty\}$ be the boundary divisor. By Deligne's extension theorem, the category $\Un(X)$ is canonically equivalent to the category $\Un(\overline{X}, D)$ of unipotent vector bundles with flat log connection on $(\overline{X}, D)$. Since every unipotent vector bundle on $\bP^1$ is trivial \cite[§12.4]{Deligne89}, the global sections functor on $\Un(\overline{X},D)$ is compatible with the tensor product, hence one has a canonical fiber functor
\[ \Fib_{\can}\colon \Un(X) \to \Vect_{\bQ}, \; \cE \mapsto \rH^0(\overline{X}, \cE_{\can}), \]
where $\cE_{\can}$ denotes the canonical extension of~$\cE$ to $\overline{X}$. For any two basepoints $x$, $y$ on~$X$, the fiber functors $\Fib_x$ and $\Fib_y$ are canonically isomorphic to $\Fib_{\can}$, and the composition of the isomorphisms
\[ \Fib_x \cong \Fib_{\can} \cong \Fib_y \]
defines the canonical de Rham path \eqref{eq:canonical-path}.

The path $\gamma_{x,y}^{\dR}$ can also be described as the unique element of the zeroth Hodge filtration subscheme $F^0 \pi_1^{\dR}(X;x,y)$. This follows from the fact that $F^0 \pi_1^{\dR}(X;x,y)$ consists of those Tannakian paths on $\Un(X) \cong \Un(\overline{X},D)$ which come from the category of unipotent vector bundles on $\overline{X}$ without connection, see \cite[Lemma~4.2.1]{betts:effective}. As a consequence, the map 
\begin{equation}
	\label{eq:hodge-filtered-map}
	(\gamma_{x,y}^{\dR})^{\sharp} \colon \cO(\pi_1^{\dR}(X;x,y)) \to \bQ
\end{equation}
is Hodge-filtered, i.e., vanishes on $F^1 \cO(\pi_1^{\dR}(X;x,y))$.

\subsubsection{Structure of the fundamental group}
\label{sec:structure-of-fundamental-group}

The fundamental group $\pi_1^{\dR}(X)$ has the structure of a free pro-unipotent group on two generators $\gamma_0$, $\gamma_1$. They can be constructed as follows. Denote by $T^\times_0 \cong \bG_m$ the tangent space of $\bP^1$ at~$0$ without the origin. Deligne's theory of tangential basepoints \cite[§15]{Deligne89} yields a functor $\Un(X) \to \Un(T_0^{\times})$, which upon taking Tannaka groups yields the \emph{local monodromy homomorphism}
\begin{equation}
	\label{eq:monodromy-homomorphism}
	\mu_0\colon \pi_1^{\dR}(T_0^{\times}, 1) \to \pi_1^{\dR}(X,0).
\end{equation}
The de Rham fundamental group of $T_0^{\times} = \bG_m$ is the 1-dimensional vector group $\bG_a$ with Lie algebra $\rH_1^{\dR}(\bG_m)$, the dual space of $\rH^1_{\dR}(\bG_m) = \bQ [\rd t/t]$ \cite[§14]{Deligne89}. The image of the canonical generator of $\pi_1^{\dR}(T_0^{\times}, 1)$ under the local monodromy homomorphism~\eqref{eq:monodromy-homomorphism} defines an element $\gamma_0 \in \pi_1^{\dR}(X,0)(\bQ)$, which can be thought of as an infinitesimal loop around~$0$. The same construction for the tangential basepoint $(\vec{-1})_1$ near~1 yields an element of $\pi_1^{\dR}(X,1)(\bQ)$ and we define $\gamma_1'$ to be its inverse. Transporting it along the canonical path $\gamma_{0,1}^{\dR} \in \pi_1^{\dR}(X;0,1)(\bQ)$ yields the loop 
\begin{equation}
	\label{eq:gamma1-def}
	\gamma_1 \coloneqq (\gamma_{0,1}^{\dR})^{-1} \gamma_1' \gamma_{0,1}^{\dR} \in \pi_1^{\dR}(X,0)(\bQ).
\end{equation}
It follows from \cite[Proposition~12.10]{Deligne89} that $\pi_1^{\dR}(X)$ is a free pro-unipotent group on the generators~$\gamma_0$, $\gamma_1$. Equivalently, the Lie algebra $\fn(X) \coloneqq \Lie \pi_1^{\dR}(X)$ is free pro-nilpotent on the two generators $e_0 \coloneqq \log(\gamma_0)$ and $e_1 \coloneqq \log(\gamma_1)$.

\subsubsection{Matrix coefficients}
Let $x$ and $y$ be two basepoints. By the general Tannakian formalism, functions on the path space $\pi_1^{\dR}(X; x,y)$ are given by matrix coefficients. We thank Alex Betts for explaining this to us and refer to \cite[§2]{brown:motivic-periods} and \cite[§4.7]{deligne:categories-tannakiennes} for further details. A \emph{matrix coefficient} is a triple 
\[ [\cE, v, \phi] \] 
where $\cE$ is an object of~$\Un(X)$, $v \in \Fib_x(\cE)$, and $\phi\colon \Fib_y(\cE) \to \bQ$ is a $\bQ$-linear map. It represents the function on the path space sending $\gamma \in \pi_1^{\dR}(X; x,y)(R)$, valued in some $\bQ$-algebra, to $(\phi \otimes R)(\gamma_{\cE}(v)) \in R$, in other words to the image of $v \in \Fib_x(\cE)$ under the composition
\[ \Fib_x(\cE) \otimes R \overset{\gamma_{\cE}}{\lto} \Fib_y(\cE) \otimes R \overset{\phi \otimes R}{\lto} R. \]

We say that two matrix coefficients $[\cE,v,\phi]$ and $[\cE',v',\phi']$ are \emph{basic-equivalent} if there exists a morphism $f\colon \cE \to \cE'$ in $\Un(X)$ such that $v' = f_*(v)$ and $\phi = f^*(\phi')$. The collection of all matrix coefficients modulo the equivalence relation generated by basic equivalence has the structure of a $\bQ$-algebra, with addition and multiplication given by
\begin{align*}
	[\cE,v,\phi] + [\cE',v',\phi'] &= [\cE\oplus \cE', (v,v'), \phi+\phi'],\\
	[\cE,v,\phi] \cdot [\cE',v',\phi'] &= [\cE\otimes \cE', v \otimes v', \phi\otimes \phi'].
\end{align*}
The significance of this construction is that $\pi_1^{\dR}(X;x,y)$ is isomorphic as a $\bQ$-scheme to the spectrum of this algebra of matrix coefficients. (This construction works for arbitrary Tannakian categories; it is not specific to $\Un(X)$.) This can be concisely summarized in categorical language by saying that the ring of functions on $\pi_1^{\dR}(X;x,y)$ is given by the \emph{coend}
\[ \cO(\pi_1^{\dR}(X;x,y)) = \int^{\cE \in \Un(X)} \Fib_x(\cE) \otimes \Fib_y(\cE)^\vee. \]

The group(oid) operations on the de Rham fundamental group and path spaces can also be expressed in terms of matrix coefficients. Given three basepoints $x,y,z$ of~$X$, the path composition map $(\gamma_1,\gamma_2) \mapsto \gamma_1 \gamma_2$
\[ \pi_1^{\dR}(X;y,z) \times \pi_1^{\dR}(X;x,y) \to \pi_1^{\dR}(X;x,z) \]
corresponds in terms of matrix coefficients to the map
\begin{align*}
	\int^{\cE} \Fib_x(\cE) \otimes \Fib_z(\cE)^\vee &\to \Bigl(\int^{\cF} \Fib_y(\cF) \otimes \Fib_z(\cF)^\vee\Bigr) \otimes \Bigl(\int^{\cG} \Fib_x(\cG) \otimes \Fib_y(\cG)^\vee \Bigr),\\
	[\cE, v_x, \phi_z] &\mapsto \sum_{i=1}^n [\cE, b_i, \phi_z] \otimes [\cE, v_x, b_i^\vee],
\end{align*}
where $(b_i)_{i=1,\ldots,n}$ is a basis of $\Fib_y(\cE)$ and $(b_i^\vee)_{i=1,\ldots,n}$ is the dual basis. (It is easily checked that this is well-defined.)
The path inversion map $\pi_1^{\dR}(X;x,y) \to \pi_1^{\dR}(X;y,x)$, $\gamma \mapsto \gamma^{-1}$ corresponds in terms of matrix coefficients to
\begin{align*}
	\int^{\cE} \Fib_y(\cE) \otimes \Fib_x(\cE)^\vee &\to \int^{\cF} \Fib_x(\cF) \otimes \Fib_y(\cF)^\vee, \\
	[\cE, v_y, \phi_x] &\mapsto [\cE^\vee, \phi_x^{\vee}, v_y^{\vee}].
\end{align*}

\subsubsection{Tannakian iterated integrals}
\label{sec:tannakian-iterated-integrals}
Let $\underline{\omega} = \omega_1 \cdots \omega_n$ be a word in the differential forms $\mathrm{d}t/t$ and $\mathrm{d}t/(1-t)$. To $\underline{\omega}$ we associate an object $\cE_{\underline{\omega}}$ of $\Un(X)$, namely the trivial vector bundle $\cO_X^{n+1}$ with connection matrix
\begin{equation}
	\label{eq:connection_matrix}
	\mathrm{d} - \nabla = \begin{pmatrix}
		0 & \omega_1 & 0 & \ldots & 0 & 0 \\
		& 0 & \omega_2 & \ldots & 0 & 0 \\
		&   & 0 & \ldots  & 0 & 0 \\
		& & & \ddots & \vdots & \vdots \\
		& & & & 0 & \omega_n \\
		& & & & & 0
	\end{pmatrix}.
\end{equation}
For any basepoint~$x$ we have a canonical isomorphism $\Fib_x(\cE_{\underline{\omega}}) \cong \bQ^{n+1}$ since the underlying vector bundle of~$\cE_{\underline{\omega}}$ is trivial. (In the case where~$x$ is a tangential basepoint, this uses also the fact that the differential forms have logarithmic poles at the cusps, so Deligne's canonical extension of $\cE_{\underline{\omega}}$ to~$\bP^1$ also has trivial underlying vector bundle.) For any two basepoints~$x$ and~$y$ and any path $\gamma \in \pi_1^{\dR}(X; x,y)(R)$ valued in some $\bQ$-algebra~$R$, the Tannakian iterated integral
\[ \int_{\gamma} \omega_1 \cdots \omega_n \in R \]
is defined as the matrix coefficient $[\cE_{\underline{\omega}}, e_{n+1}, e_1^\vee]$ evaluated at~$\gamma$, i.e.\ the value of $e_{n+1} = (0,\ldots,0,1) \in R^{n+1}$ under the composition
\[ R^{n+1} = \Fib_x(\cE_{\underline{\omega}}) \otimes R  \xrightarrow{\gamma_{\cE_{\underline{\omega}}}} \Fib_y(\cE_{\underline{\omega}}) \otimes R = R^{n+1} \overset{e_1^\vee}{\lto} R. \]

Tannakian iterated integrals satisfy the following properties:
\begin{enumerate}
	\item \emph{Empty word.} For $\underline{\omega} = \emptyset$ the empty word and any path $\gamma \in \pi_1^{\dR}(X;x,y)$ we have
	\begin{equation*}
		\label{eq:integral_empty}
		\int_{\gamma} \emptyset = 1.
	\end{equation*}
	\item \emph{Constant paths.} For $1_x \in \pi_1^{\dR}(X,x)(\bQ)$ the constant path and $n \geq 1$ we have
	\begin{equation*}
		\label{eq:integral_constant}
		\int_{1_x} \omega_1 \cdots \omega_n = 0.
	\end{equation*}
	\item \emph{Path composition.} For any three basepoints $x,y,z$ and for any two paths $\gamma_1 \in \pi_1^{\dR}(X;y,z)$ and $\gamma_2 \in \pi_1^{\dR}(X;x,y)$ we have
	\begin{equation*}
		\label{eq:integral_composition}
		\int_{\gamma_1 \gamma_2} \omega_1 \cdots \omega_n = \sum_{i=0}^n \left(\int_{\gamma_1} \omega_1 \cdots \omega_i \right) \left(\int_{\gamma_2} \omega_{i+1}\cdots \omega_n \right).
	\end{equation*}
	\item \emph{Inversion.} For $\gamma \in \pi_1^{\dR}(X;x,y)$ with inverse $\gamma^{-1} \in \pi_1^{\dR}(X;y,x)$ we have
	\begin{equation*}
		\label{eq:integral_inversion}
		\int_{\gamma^{-1}} \omega_1 \cdots \omega_n = (-1)^n \int_{\gamma} \omega_n \cdots \omega_1.
	\end{equation*}
	\item \emph{Shuffle product formula.} For $\gamma \in \pi_1^{\dR}(X;x,y)$ and two sequences $\omega_1\cdots\omega_n$ and $\omega_{n+1}\cdots \omega_{n+m}$ of differential forms we have
	\begin{equation*}
		\left(\int_{\gamma} \omega_1\cdots\omega_n\right) \left(\int_{\gamma} \omega_{n+1}\cdots\omega_{n+m}\right) = \sum_{\pi \in \Sh(n,m)} \int_{\gamma} \omega_{\pi(1)}\cdots \omega_{\pi(n+m)}
	\end{equation*}
	where the sum is over shuffle permutations $\pi \in \Sh(n,m) \subseteq S_{n+m}$, i.e., permutations satisfying $\pi^{-1}(1) < \ldots < \pi^{-1}(n)$ and $\pi^{-1}(n+1) < \ldots < \pi^{-1}(n+m)$.
\end{enumerate}

Let us sketch just the proof of the path inversion formula, to give an idea of how these properties can be proved.

\begin{proof}[Proof of path inversion formula]
	The inversion map
	\[
	\pi_1^{\dR}(X;x,y) \to \pi_1^{\dR}(X;y,x)
	\] 
	corresponds to the antipode $S\colon \cO(\pi_1^{\dR}(X;y,x)) \to \cO(\pi_1^{\dR}(X;x,y))$ which is given by $[\cE,v,\phi] \mapsto [\cE^\vee, \phi^\vee, v^\vee]$ in terms of matrix coefficients. Given $\cE_{\underline{\omega}} = \cO_X^{n+1}$ with connection matrix~\eqref{eq:connection_matrix}, the dual object in $\Un(X)$ is $\cE_{\underline{\omega}}^\vee = \cO_X^{n+1}$ with connection matrix
	\[ 
	\begin{pmatrix}
		0 &  & & & & \\
		-\omega_1 & 0 & & & & \\
		0 & -\omega_2 & 0 & & & \\
		\vdots & \vdots & \vdots & \ddots & & \\
		0 & 0 & 0 & \cdots & 0 & \\
		0 & 0 & 0 & \cdots & -\omega_n & 0
	\end{pmatrix}.
	\]
	Reversing the order of coordinates transforms this into
	\[
	\begin{pmatrix}
		0 & -\omega_n & 0 & \ldots & 0 & 0 \\
		& 0 & -\omega_{n-1} & \ldots & 0 & 0 \\
		&   & 0 & \ldots  & 0 & 0 \\
		& & & \ddots & \vdots & \vdots \\
		& & & & 0 & -\omega_1 \\
		& & & & & 0
	\end{pmatrix},
	\]
	so we find that
	\[ \int_{\gamma^{-1}} \omega_1 \cdots \omega_n = \int_\gamma (-\omega_n)\cdots(-\omega_1) = (-1)^n \int_\gamma \omega_n \cdots \omega_1 \]
	as claimed.
\end{proof}

We note the following property of the canonical de Rham path $\gamma_{x,y}^{\dR} \in \pi_1^{\dR}(X;x,y)(\bQ)$ (§\ref{sec:canonical-paths}).

\begin{lemma}
	\label{integral-over-canonical-de-rham-path}
	If $\underline{\omega} = \omega_1\cdots\omega_n$ is not the empty word then
	\[ \int_{\gamma_{x,y}^{\dR}} \omega_1\cdots \omega_n = 0. \]
\end{lemma}

\begin{proof}
	Deligne's extension of $\cE_{\underline{\omega}}$ to a unipotent vector bundle $\cE_{\underline{\omega},\can}$ on $\overline{X} = \bP^1$ with flat logarithmic connection is still the trivial vector bundle $\cO_{\overline{X}}^{n+1}$ and has the same connection matrix~\eqref{eq:connection_matrix} since the differential forms already have logarithmic poles at the cusps. The canonical fiber functor $\Fib_{\can}$ applied to $\cE_{\underline{\omega}}$ is simply
	\[ \Fib_{\can}(\cE_{\underline{\omega}}) = \rH^0(\overline{X}, \cE_{\underline{\omega},\can}) = \rH^0(\overline{X}, \cO_{\overline{X}}^{n+1}) = \bQ^{n+1} \]
	and the composition
	\[ (\gamma_{x,y}^{\dR})_{\cE_{\underline{\omega}}} \colon \bQ^{n+1} = \Fib_x(\cE_{\underline{\omega}}) \cong \Fib_{\can}(\cE_{\underline{\omega}}) \cong \Fib_y(\cE_{\underline{\omega}}) = \bQ^{n+1} \]
	is the identity on~$\bQ^{n+1}$. In particular, since $n > 0$, the vector $(\gamma_{x,y}^{\dR})_{\cE_{\underline{\omega}}}(0,\ldots,0,1)$ has its first entry equal to zero. But this means precisely that the iterated integral over $\gamma_{x,y}^{\dR}$ vanishes.
\end{proof}

\subsubsection{Multiple polylogarithms}
\label{sec:multiple-polylogarithms}
The formalism of Tannakian iterated integrals outlined above applies to arbitrary smooth curves. In our particular case of the thrice-punctured line, iterated integrals are called \emph{multiple polylogarithms} and special notation is used. For any two basepoints~$x$ and~$y$ on~$X$ and any path $\gamma \in \pi_1^{\dR}(X;x,y)$, the integral over $\rd t/t$ is called the \emph{logarithm}:
\begin{align*}
	\log(\gamma) \coloneqq \int_\gamma \frac{\rd t}{t}.
\end{align*}
The integral over a word of the form $\frac{\rd t}{t} \cdots \frac{\rd t}{t} \frac{\rd t}{1-t}$ of length~$n$ is called the \emph{$n$-th polylogarithm}:
\[ \Li_n(\gamma) \coloneqq \int_\gamma \underbrace{\frac{\rd t}{t} \cdots \frac{\rd t}{t} \frac{\rd t}{1-t}}_{n}. \]
More generally, for a tuple $(n_1,\ldots,n_r)$ of positive integers, we have the \emph{multiple polylogarithm}
\[ \Li_{n_1,\ldots,n_r}(\gamma) \coloneqq \int_\gamma \underbrace{\frac{\rd t}{t} \cdots \frac{\rd t}{t} \frac{\rd t}{1-t}}_{n_1} \cdots \cdots \underbrace{\frac{\rd t}{t} \cdots \frac{\rd t}{t} \frac{\rd t}{1-t}}_{n_r}. \]

An alternative notation uses words in two symbols $e_0$ and $e_1$ as subscripts, corresponding to the two differential forms~$\rd t/t$ and $\rd t/(1-t)$. In this notation we have $\log = \Li_{e_0}$, $\Li_n = \Li_{e_0^{n-1}e_1}$, and $\Li_{n_1,\ldots,n_r} = \Li_{e_0^{n_1-1}e_1\cdots e_0^{n_r-1}e_1}$. 

\subsubsection{The universal unipotent vector bundle with flat connection}
\label{sec:universal-unipotent-vector-bundle}

Fix a basepoint $x$ of~$X$. The category $\Un(X)$ has a universal pointed pro-object $(\cE_{\univ}, 1_x)$ where $1_x$ is an element of $\cE_{\univ,x} \coloneqq \Fib_x(\cE_{\univ})$ \cite{kim09}. It is characterized by the fact that for any $\cV \in \Un(X)$ and $v_x \in \cV_x$ there exists a unique map $\phi\colon \cE_{\univ} \to \cV$ of pro-unipotent vector bundles with flat connection such that $1_x \mapsto v_x$; in other words, $(\cE_{\univ}, 1_x)$ represents the fiber functor~$\Fib_x$. There is a natural co-algebra structure on $\cE_{\univ}$: the coproduct $\Delta\colon \cE_{\univ} \to \cE_{\univ} \,\hat \otimes \, \cE_{\univ}$ is the unique map with $1_x \mapsto 1_x \otimes 1_x$ in the fiber over~$x$, and the counit is the unique map $\varepsilon\colon \cE_{\univ} \to \cO_X$ with $1_x \mapsto 1$. In particular, for any other basepoint~$y$ of~$X$, the fiber $\cE_{\univ,y}$ is a co-algebra over~$\bQ$, so its linear dual $\cE_{\univ,y}^{\vee}$ is a $\bQ$-algebra. There is a natural isomorphism
\begin{equation}
	\label{eq:path-space-via-universal-object}
	\pi_1^{\dR}(X;x,y) \cong \Spec(\cE_{\univ,y}^{\vee})
\end{equation}
which can be described as follows: given a de Rham path $\gamma \in \pi_1^{\dR}(X;x,y)(R)$ defined over some $\bQ$-algebra~$R$, we get an element $\gamma_{\cE_{\univ}}(1_x) \in \cE_{\univ,y} \, \hat\otimes\, R$, which defines a $\bQ$-linear map $f_{\gamma}\colon \cE_{\univ,y}^\vee \to R$. The fact that $\gamma$ is tensor-natural implies that $f_{\gamma}$ is an algebra homomorphism, hence defining an element of $\Spec(\cE_{\univ,y}^\vee)(R)$.

An explicit construction of the universal bundle is described in \cite{kim09}. Corresponding to the choice of basis $\omega_0 = \rd t/t$ and $\omega_1 = \rd t/(1-t)$ of $\rH^1_{\dR}(X)$ we have the pro-vector space of non-commutative power series $E \coloneqq \bQ\llangle e_0,e_1 \rrangle$, viewed as a pro-object by quotienting by powers of the augmentation ideal. The universal pro-bundle is given by $\cE_{\univ} \coloneqq E \, \hat \otimes\, \cO_X$ with connection determined by
\[ \nabla f = -e_0 f \omega_0 - e_1 f \omega_1 \]
for constant sections $f \in E$, and with distinguished element $1$ in the fiber over~$x$. The coproduct on this bundle satisfies $\Delta e_i = e_i \otimes 1 + 1 \otimes e_i$ for $i=0,1$ and is compatible with the multiplication on~$E$. Note that all fibers of $\cE_{\univ}$ are isomorphic to~$E$. Via the isomorphism~\eqref{eq:path-space-via-universal-object}, de Rham paths $\gamma \in \pi_1^{\dR}(X;x,y)(R)$ can be identified with ring homomorphisms $E^{\vee} \to R$, hence with grouplike elements of $E \,\hat\otimes\, R = R\llangle e_0,e_1 \rrangle$.

\begin{lemma}
	\label{paths-as-series}
	Under the isomorphism~\eqref{eq:path-space-via-universal-object}, a de Rham path $\gamma \in \pi_1^{\dR}(X;x,y)(R)$ corresponds to the grouplike element
	\begin{equation}
		\label{eq:path-as-series}
		\sum_w \left(\int_{\gamma} \underline{\omega}_w\right) w \in R\llangle e_0,e_1\rrangle,
	\end{equation}
	where $\underline{\omega}_w = \omega_{i_1} \cdots \omega_{i_n}$ denotes the sequence of differential forms corresponding to a word $w = e_{i_1}\cdots e_{i_n}$.
\end{lemma}

\begin{proof}
	The path $\gamma$ corresponds to the element $\gamma_{\cE_{\univ}}(1_x) \in \cE_{\univ,y} \, \hat\otimes\, R = E \,\hat\otimes \,R = R\llangle e_0,e_1\rrangle$. Write this as $\sum_w a_w w$. We are claiming that $a_w = \int_{\gamma} \underline{\omega}_w$. Fix $w = e_{i_1}\cdots e_{i_n}$, and let $\underline{\omega} = \omega_{i_1}\cdots\omega_{i_n}$ be the corresponding sequence of differential forms. Associated to $\underline{\omega}$ we have the unipotent vector bundle with flat connection $\cE_{\underline{\omega}}$ from §\ref{sec:tannakian-iterated-integrals}. By the universality of $\cE_{\univ}$ there is a unique map $\phi\colon \cE_{\univ} \to \cE_{\underline{\omega}}$ whose induced map on $x$-fibers $E \to \bQ^{n+1}$ sends $1 \mapsto (0,\ldots,0,1)$. The map given on constant sections by $\sum_{w'} b_{w'} w' \mapsto (b_{e_{i_k}\cdots e_{i_n}})_{k=1}^{n+1}$ is compatible with the connections on $\cE_{\univ}$ and $\cE_{\underline{\omega}}$ and maps $1 \mapsto (0,\ldots,0,1)$, so it is the map~$\phi$. By naturality of~$\gamma$, the iterated integral $\int_{\gamma} \underline{\omega}$ equals the first entry of the vector $\phi_y(\sum_{w'} a_{w'} w') \in \cE_{\underline{\omega},y} \otimes R = R^{n+1}$, but this is precisely $a_w$.
\end{proof}

In \Cref{paths-as-series}, the fact that $\sum_w \left(\int_{\gamma} \underline{\omega}_w\right) w$ is a grouplike element of $R\llangle e_0,e_1\rrangle$ is a reflection of the shuffle product formula for iterated integrals (§\ref{sec:tannakian-iterated-integrals}).

\begin{lemma}
	\label{lyndon-basis-of-coordinate-ring}
	The multiple polylogarithm functions $\Li_w$ with $w$ running through the Lyndon words on the letters $\{e_0 < e_1 \}$ form an algebra basis of $\cO(\pi_1^{\dR}(X;x,y))$.
\end{lemma}

\begin{proof}
	The ring $\cO(\pi_1^{\dR}(X;x,y)) = \bQ\llangle e_0,e_1 \rrangle^\vee$ is a shuffle algebra. An algebra basis is given by the coefficient extraction functionals $f_w$ with $w$ running through the Lyndon words \cite[Theorem~6.1]{reutenauer93}. By \Cref{paths-as-series}, $f_w$ is the multiple polylogarithm function $\Li_w$ associated to the word~$w$.
\end{proof}

\subsubsection{Frobenius-invariant paths}
\label{sec:frobenius-invariant-paths}

Let $p$ be a prime. For any two $\bZ_p$-integral basepoints $x$, $y$ of~$X$ there exists a certain canonical $\bQ_p$-valued de Rham path
\begin{equation}
	\label{eq:frobenius-invariant-path}
	\gamma_{x,y}^{\cris} \in \pi_1^{\dR}(X; x,y)(\bQ_p).
\end{equation}
Its construction relies on the comparison isomorphism
\begin{equation}
	\label{eq:deRham-crystalline-comparison}
	\pi_1^{\dR}(X;x,y) \otimes_{\bQ} \bQ_p \cong \pi_1^{\cris}(X;x,y)
\end{equation}
between the de Rham path space and the crystalline path space \cite[Proposition~2.4.1]{chiarelotto-le-stum}. The latter is defined as the scheme of tensor isomorphisms of fiber functors on the Tannakian category of unipotent overconvergent isocrystals on $X_{\FF_p}$. There is a canonical path $\gamma_{x,y}^{\cris} \in \pi_1^{\cris}(X; x,y)(\bQ_p)$ called \emph{analytic continuation along Frobenius} \cite[Definition~3.5]{Besser}, the unique path that is fixed by the Frobenius automorphism \cite[Corollary~3.3]{Besser}. Via the comparison isomorphism~\eqref{eq:deRham-crystalline-comparison}, we can view $\gamma_{x,y}^{\cris}$ as a $\bQ_p$-valued de Rham path.

\subsubsection{Iterated Coleman integrals and $p$-adic multiple zeta values}
\label{sec:padic-multiple-polylogarithms}

Let $p$ be a prime and $x,y$ two $\bZ_p$-integral basepoints. In Besser's interpretation of Coleman integration via the Tannakian formalism \cite{Besser}, the \emph{Coleman iterated integral} from $x$ to $y$ is defined as the iterated integral (§\ref{sec:tannakian-iterated-integrals}) along the Frobenius-invariant path $\gamma_{x,y}^{\cris} \in \pi_1^{\dR}(X;x,y)(\bQ_p)$ discussed in the previous paragraph: 
\[ \int_x^y \omega_1\cdots \omega_n \coloneqq \int_{\gamma_{x,y}^{\cris}} \omega_1 \cdots \omega_n \in \bQ_p. \]
In particular, when the first basepoint is fixed as the tangential basepoint $0 = \vec{1}_0$, we have the \emph{$p$-adic logarithm}
\[ \log^p(z) \coloneqq \int_0^z \frac{\rd t}{t}; \]
for $n \geq 1$ we have the \emph{$n$-th $p$-adic polylogarithm}
\[ \Li_n^p(z) \coloneqq \int_0^z \underbrace{\frac{\rd t}{t} \cdots \frac{\rd t}{t} \frac{\rd t}{1-t}}_{n}; \]
and more generally, for a tuple $(n_1,\ldots,n_r)$ of positive integers, we have the \emph{$p$-adic multiple polylogarithm}
\[ \Li_{n_1,\ldots,n_r}^p(z) \coloneqq \int_0^z \underbrace{\frac{\rd t}{t} \cdots \frac{\rd t}{t} \frac{\rd t}{1-t}}_{n_1} \cdots \cdots \underbrace{\frac{\rd t}{t} \cdots \frac{\rd t}{t} \frac{\rd t}{1-t}}_{n_r}. \]
The multiple polylogarithm functions $\Li_{n_1,\ldots,n_r}^p$ and $\log(z)$ belong to the ring $A_{\myCol}(X)$ of \emph{Coleman functions} on $X(\bZ_p)$, a sub-$\bQ_p$-algebra of the ring of locally analytic functions \cite[§4]{Besser}. (More precisely, they are \emph{Coleman algebraic functions}, cf.\ \cite[§4]{betts:effective}.)

Iterated Coleman integrals from the tangential basepoint $0 = \vec{1}_0$ to the tangential basepoint $1 = (\vec{-1})_1$ are called \emph{$p$-adic multiple zeta values} \cite{FurushoI, FurushoII}%
\footnote{\label{footnote:furusho}Note that Furusho uses different conventions: our $\Li_{n_1,\ldots,n_r}$ is his $\Li_{n_r,\ldots,n_1}$ and our $\zeta^p(n_1,\ldots,n_r)$ is his $\zeta(n_r,\ldots,n_1)$, with the order reversed.}:
\[ \zeta^p(n_1,\ldots,n_r) \coloneqq \Li_{n_1,\ldots,n_r}^p(1) = \int_0^1  \underbrace{\frac{\rd t}{t} \cdots \frac{\rd t}{t} \frac{\rd t}{1-t}}_{n_1} \cdots \cdots \underbrace{\frac{\rd t}{t} \cdots \frac{\rd t}{t} \frac{\rd t}{1-t}}_{n_r}. \]
We often omit the superscript $(-)^p$ when there is no danger of confusion.

\subsection{Mixed Tate motives and the motivic fundamental group}
\label{sec:mixed-tate-motives}

\subsubsection{The mixed Tate motivic Galois group of an integer scheme}
\label{sec:mixed-tate-galois-group}
Let $S$ be a finite set of primes and let $Z = \Spec \Zb[1/S] = \Spec(\bZ) \smallsetminus S$ be its complement in $\Spec(\bZ)$. Let $\MT(Z, \QQ)$ be the Tannakian category of mixed Tate motives over~$Z$ with $\QQ$-coefficients \cite{DelGon05}.
We have a \emph{de Rham realization functor}
\[
\mathrm{real}^{\dR} \colon \MT(Z, \QQ) \to \Vect_{\Qb}
\]
and a \emph{$p$-adic étale realization functor}
\[
\mathrm{real}^{\et} \colon \MT(Z, \QQ) \to \mathrm{Rep}^{f,Z}_{\Qp}(G_{\Qb})
\]
to Galois representations that are unramified at $\ell \in Z \setminus \{p\}$ and crystalline (resp.\ de Rham) at $p$ if $p \in Z$ (resp.\ $p \in S$). If we fix $p \in Z$, we also have a
\emph{crystalline realization functor}
\[ \real^{\cris}\colon \MT(Z,\QQ) \to \Vect_{\bQ_p}. \]
Note that every de Rham realization $\mathrm{real}^{\dR}(M)$ comes with a canonical Hodge filtration, and every crystalline realization $\mathrm{real}^{\cris}(M)$ comes with a Frobenius action.
The category $\MT(Z, \QQ)$ is a neutral $\Qb$-linear Tannakian category with fiber functor $\mathrm{real}^{\dR}$, and we let $G_Z$ denote its Tannakian fundamental group. The simple objects of $\MT(Z, \QQ)$ are precisely $\Qb(n) \coloneqq \Qb(1)^{\otimes n}$ for $n \in \Zb$, where $\Qb(1) \coloneqq h_2(\Pb^1;\Qb)$. Every $M \in \MT(Z,\QQ)$ has an increasing filtration $W_i M$ known as the \emph{weight filtration}. The quotient $W_i M / W_{i-1} M$ is trivial when $i$ is odd and is isomorphic to a direct sum of copies of $\Qb(-i/2)$ when $i$ is even. The weight filtration is split by the Hodge filtration on $\mathrm{real}^{\dR}$, which implies a canonical decomposition
\[
G_Z = \UZ \rtimes \Gm,
\]
where $\UZ$ is the maximal pro-unipotent subgroup of $G_Z$. We call $\UZ$ the \emph{unipotent mixed Tate motivic Galois group} of~$Z$. The action of $\Gm$ on $\UZ$ by conjugation gives an action of $\Gm$, or equivalently, a grading, on the Hopf algebra $\cO(\UZ)$. This graded Hopf algebra is denoted
\[
\bigoplus_{i=0}^\infty A(Z)_i = A(Z) := \Oo(\UZ),
\]
where $A(Z)_i$ denotes the graded piece on which $\lambda \in \bG_m$ acts as multiplication by $\lambda^{-i}$. 
We refer to the degree on $A(Z)$ as the \emph{half-weight}, since it is half the ordinary motivic weight. The pro-nilpotent Lie algebra of the pro-unipotent group $U_Z$ is denoted by
\[ \prod_{i=1}^\infty \fn(Z)_{-i} = \fn(Z).  \]
It is isomorphic to the space of derivations $A(Z) \to \bQ$ at the identity. Work of Voevodsky shows that $\Ext^1_{\MT(Z,\QQ)}(\Qb(0),\Qb(n)) \cong K_{2n-1}(Z)_\Qb = K_{2n-1}^{(n)}(Z)$ and $\Ext^i_{\MT(Z,\QQ)}=0$ for $i \ge 2$, which by work of Borel implies that $\nf(Z)$ is (non-canonically)\footnote{The isomorphism is canonical only on abelianization, which implies the isomorphism for $A(Z)$ is canonical only on primitive elements.} isomorphic to a free pro-nilpotent Lie algebra on the set $\Sigma \coloneqq \{\tau_{\ell}\}_{\ell \in S} \cup \{\sigma_{2n+1}\}_{n \in \mathbb{Z}_{\ge 1}}$. Such an isomorphism determines an isomorphism between $A(Z)$ and the shuffle algebra spanned by elements $a_w$ dual to the words $w$ in the generators~$\Sigma$.

\subsubsection{The Motivic Unipotent Fundamental Group of $X$}
\label{sec:motivic-fundamental-group}

For basepoints $x$ and $y$, each of which is either an element of $X(Z)$ or a $Z$-integral tangential basepoint (as defined in \cite[\S 15]{Deligne89}), we have the motivic path space $\pi_1^{\mot}(X; x,y)$ from $x$ to $y$ in $X$ \cite[§4.2]{DelGon05}. It is an affine scheme in the Tannakian category $\MT(Z,\QQ)$, meaning that its coordinate ring $\cO(\pi_1^{\mot}(X; x,y))$ is an algebra in $\Ind\mhyphen\MT(Z,\QQ)$. 
Equivalently, this means that $\mathrm{real}^{\dR}(\pi_1^{\mot}(X; x,y))$ is an affine scheme with algebraic action of $G_Z$. For $x=y$, we denote the motivic path space by $\pi_1^{\mot}(X,x)$ and call it the \emph{motivic fundamental group} of~$X$ at~$x$. Then $\Oc(\pi_1^{\mot}(X,x))$ is a Hopf algebra in $\Ind\mhyphen\MT(Z,\QQ)$, while its Lie algebra $\nf(X)$ is a Lie algebra in $\Pro\mhyphen\MT(Z,\QQ)$. For the particular tangential basepoint $x=\vec{1}_0$, we denote $\pi_1^{\mot}(X,x)$ by $\pi_1^{\mot}(X)$.

The image of the motivic path space $\pi_1^{\mot}(X; x,y)$ under the de Rham realization functor $\real^{\dR}$ (resp.\ $p$-adic étale realization functor $\real^{\et}$, resp.\ crystalline realization functor $\real^{\cris}$) is the de Rham (resp.\ \'etale, crystalline) torsor of paths from $x$ to $y$, i.e., the scheme of tensor isomorphisms from the fiber functor at $x$ to the fiber functor at $y$ on the category $\Un(X)$ of unipotent vector bundles with flat connection on $X_{\Qb}$ (resp.\ the category of unipotent $\Qp$-local systems on $X_{\overline{\Qb}}$, the category of unipotent overconvergent isocrystals on $X_{\Fb_p}$). In particular, the motivic Galois group $G_Z$ acts on the de Rham fundamental group $\pi_1^{\dR}(X)$. Having an explicit understanding of this Galois action is crucial for our results. We address this problem in §\ref{sec:galois-action} below.

\subsubsection{The $p$-adic period point}
\label{sec:period-point}

Let $p \in Z$ be a prime. The unipotent mixed Tate motivic Galois group $U_Z$ contains a special $\bQ_p$-point
\begin{equation}
	\label{eq:period-point}
	\eta_p \in U_Z(\bQ_p)
\end{equation}
called the \emph{$p$-adic period point}. The associated algebra homomorphism
\begin{equation}
	\label{eq:period-map}
	\per_p = \eta_p^{\sharp}\colon A(Z) \to \bQ_p
\end{equation}
is called the \emph{$p$-adic period map}.
The point $\eta_p$ is the inverse of the point $\eta_p^{\ur}$ constructed in \cite[Lemma~2.2.5]{ChatUnv}; see also \cite[§3.3]{BKL:chabauty_kim_sc}. It is defined by comparing two canonical splittings $s_{\dR}$ and $s_{\cris}$ of the weight filtration on the de Rham realization over $\bQ_p$. For a mixed Tate motive $M \in \MT(Z,\bQ)$, the Hodge filtration provides a splitting:
\[ M^{\dR} = \bigoplus_i W_{2i}M^{\dR} \cap F^i M^{\dR}. \]
On the other hand, $M^{\dR}_{\bQ_p} \cong M^{\cris}$ has the structure of a filtered $\varphi$-module, and the decomposition of $M^{\dR}_{\bQ_p}$ into Frobenius eigenspaces provides a second splitting:
\[ M^{\dR}_{\bQ_p} = \bigoplus_i (M^{\dR}_{\bQ_p})^{\varphi=p^i}. \]
The automorphism
\[ M^{\dR}_{\bQ_p} \xrightarrow[\sim]{s_{\cris}^{-1}} \bigoplus_i \gr^W_{2i} M^{\dR}_{\bQ_p} \xrightarrow[\sim]{s_{\dR}} M^{\dR}_{\bQ_p} \]
taking $(M^{\dR}_{\bQ_p})^{\varphi=p^i}$ to $W_{2i}M^{\dR}_{\bQ_p} \cap F^i M^{\dR}_{\bQ_p}$ is $\otimes$-natural in~$M$ and hence defines a point $\eta_p^{\ur} \in G_Z(\bQ_p)$. It is contained in the pro-unipotent subgroup $U_Z$ since it acts trivially on the pure motives $\bQ(n)$; its inverse is the $p$-adic period point~\eqref{eq:period-point}.

The property most important for us is that $\eta_p$ takes the canonical de Rham path (§\ref{sec:canonical-paths}) to the Frobenius-invariant path (§\ref{sec:frobenius-invariant-paths}) \cite[Lemma~3.9]{BKL:chabauty_kim_sc}:

\begin{lemma}
	\label{lem:period-point-of-de-rham-path}
	For any two $Z$-integral basepoints $x,y$ of $X$, we have
	\[ \eta_p(\gamma_{x,y}^{\dR}) = \gamma_{x,y}^{\cris}. \]
\end{lemma}

Elements in the image of the period map $\per_p\colon A(Z) \to \bQ_p$ are called \emph{$p$-adic periods}. Conjecturally, all algebraic relations between $p$-adic periods already hold in $A(Z)$ \cite[Conjecture~2.25]{CDC:polylog1} \cite[Conjecture 4]{YamashitaBounds}:

\begin{conj}[$p$-adic period conjecture]
	\label{conj:period-conjecture}
	The period map~\eqref{eq:period-map} is injective.
\end{conj}

\subsection{Motivic polylogarithms and MZVs}
\label{sec:motivic-polylogarithms}

\subsubsection{Motivic iterated integrals}
\label{sec:motivic-iterated-integrals}

If $z$ is a $Z$-integral basepoint of~$X$, and $(n_1,\ldots,n_r)$ is a tuple of positive integers, the (unipotent) \emph{motivic polylogarithm} 
\[ \Li_{n_1,\ldots,n_r}^{\fu}(z) \]
is the element of $A(Z) = \Oc(U_Z)$ defined by
\[ u \mapsto \Li_{n_1,\ldots,n_r}(u(\gamma_{0,z}^{\dR})) = \int_{u(\gamma_{0,z}^{\dR})} \underbrace{\frac{\rd t}{t} \cdots \frac{\rd t}{t} \frac{\rd t}{1-t}}_{n_1} \cdots \cdots \underbrace{\frac{\rd t}{t} \cdots \frac{\rd t}{t} \frac{\rd t}{1-t}}_{n_r}. \]
Here, $\gamma_{0,z}^{\dR} \in \pi_1^{\dR}(X;0,z)(\bQ)$ denotes the canonical de Rham path from the tangential basepoint $0 = \vec{1}_0$ to~$z$ (§\ref{sec:canonical-paths}); $u$ is an element of $U_Z(R)$ for some $\bQ$-algebra~$R$, and $u(\gamma_{0,z}^{\dR})$ is the $R$-valued de Rham path from $0$ to~$z$ obtained by letting $u$ act on the canonical path via the $U_Z$-action on $\pi_1^{\dR}(X;0,z)$. We also define the \emph{motivic logarithm}
\[ \log^{\fu}(z)\colon u \mapsto \int_{u(\gamma_{0,z}^{\dR})} \frac{\rd t}{t}. \]
Note that $\Li_1^{\fu}(z) = -\log^{\fu}(1-z)$.
The \emph{motivic multiple zeta value}
\[ \zeta^{\fu}(n_1,\ldots,n_r) \coloneqq \Li_{n_1,\ldots,n_r}^{\fu}(1) \in A(Z) \]
is defined as a motivic iterated integral from $0 = \vec{1}_0$ to $1 = (\vec{-1})_1$. We also sometimes use the alternative notation where the argument of~$\zeta^{\fu}$ is a word in $e_0,e_1$ corresponding to the sequence of differential forms appearing in the integral:
\[ \zeta^{\fu}(e_0^{n_1-1}e_1 \cdots e_0^{n_r-1}e_1) \coloneqq \zeta^{\fu}(n_1,\ldots,n_r). \]
We also use the notation $\zeta^{\fu}(w)$ when $w$ is a word in $\{e_0,e_1\}$ not necessarily beginning with $e_0$ and ending in $e_1$.

\begin{prop}
	Let $p \in Z$ be a prime. The period map~\eqref{eq:period-map} maps motivic iterated integrals to their $p$-adic counterparts:
	\begin{align*}
		\per_p(\log^{\fu}(z)) &= \log^p(z),\\
		\per_p(\Li_{n_1,\ldots,n_r}^{\fu}(z)) &= \Li_{n_1,\ldots,n_r}^p(z), \\
		\per_p(\zeta^{\fu}(n_1,\ldots,n_r)) &= \zeta^p(n_1,\ldots,n_r).
	\end{align*}
\end{prop}

\begin{proof}
	The period map is defined as evaluation at the $p$-adic point~$\eta_p$, which by \Cref{lem:period-point-of-de-rham-path} maps the canonical de Rham path to the crystalline path. For example:
	\[ \per_p(\log^{\fu}(z)) = \log^{\fu}(z)(\eta_p) = \int_{\eta_p(\gamma_{0,z}^{\dR})} \frac{\rd t}{t} =  \int_{\gamma_{0,z}^{\cris}} \frac{\rd t}{t} = \log^p(z). \]
	The other cases are similar. 
\end{proof}

\begin{lemma}
	\label{motivic-zetas-nonzero}
	For odd integers $n \geq 3$ we have $\zeta^{\fu}(n) \neq 0$.
\end{lemma}

\begin{proof}
	This is part of \cite[Theorem~6.4]{GonGal} but the non-vanishing is not explicitly shown there. We sketch a proof to fill this gap in the literature. We can assume $Z = \Spec(\bZ)$. The motivic zeta value $\zeta^{\fu}(n)$ is a primitive element of $A(Z)$ by Goncharov's coproduct formula, hence it lies in the image of the natural map
	\[ \Ext^1_{\MT(Z,\bQ)}(\bQ(0),\bQ(n)) \hookrightarrow A(Z). \]
	Let $\xi_n$ be the extension class representing $\zeta^{\fu}(n)$. Its image under the Hodge realization functor (to the category $\MHS(\bQ)$ of mixed Hodge structures over~$\bQ$) corresponds to the class of the real number $\zeta(n) = \sum_{k=1}^{\infty} k^{-n}$ under the identification \cite[Theorem~2.256]{mzv-book}
	\[ \Ext^1_{\MHS(\bQ)}(\bQ(0),\bQ(n)) \cong \CC/(2\pi i)^n\bQ. \]
	Since $n$ is odd, $(2\pi i)^n\bQ \subseteq i \bR$, whereas $\zeta(n) \in \bR_{>0}$, so this class is non-trivial. Since $\xi_n$ has non-trivial Hodge realization, $\xi_n \neq 0$, and thus $\zeta^{\fu}(n) \neq 0$.
\end{proof}

\subsubsection{Relations of motivic MZVs}
\label{sec:relations}

Multiple zeta values satisfy many interesting relations. Among the most important ones are the \emph{shuffle relations}:
\begin{equation}
	\label{eq:mzv-shuffle-relation}
	\zeta^{\fu}(w_1) \cdot \zeta^{\fu}(w_2) = \sum_{\pi \in \Sh(|w_1|,|w_2|)} \zeta^{\fu}(\pi(w_1 w_2))
\end{equation}
for all words $w_1$ and $w_2$ in the symbols $e_0, e_1$. Here, the sum is over all shuffle permutations of $w_1 w_2$, i.e., permutations which preserve the order of the letters within each $w_i$. For example, 
\[ \zeta^{\fu}(e_0 e_1) \zeta^{\fu}(e_0 e_0 e_1) = 6 \cdot \zeta^{\fu}(e_0e_0e_0e_1e_1) + 3\cdot \zeta^{\fu}(e_0e_0e_1e_0e_1) + \zeta^{\fu}(e_0e_1e_0e_0e_1), \]
or written differently,
\[ \zeta^{\fu}(2) \zeta^{\fu}(3) = 6 \zeta^{\fu}(4,1) + 3 \zeta^{\fu}(3,2) + \zeta^{\fu}(2,3). \]
The shuffle relations follow from the shuffle product formula for iterated integrals (§\ref{sec:tannakian-iterated-integrals}). There are, however, additional relations. We will need the identity $\zeta^{\fu}(3) = \zeta^{\fu}(2,1)$, which does not follow from the shuffle relations since these are homogeneous in the \emph{depth} (the number of times the letter $e_1$ occurs, or equivalently the number of arguments), whereas $\zeta^{\fu}(3)$ and $\zeta^{\fu}(2,1)$ have depth~1 and~2, respectively. Instead we will deduce this identity from the following duality theorem for motivic multiple zeta values. Like its analogue for classical (complex) multiple zeta values \cite[§9]{ZagierValues}, it is a simple consequence of the change of variables formula for iterated integrals.

In order to state the duality theorem, define $\tau$ to be the involution on words over $\{e_0, e_1\}$ given by
\[ \tau(e_{i_1} \cdots e_{i_r}) \coloneqq e_{1-i_r} \cdots e_{1-i_1}.  \]
In other words, $\tau$ reverses the order and interchanges $e_0$ and $e_1$. 

\begin{prop}[Duality Theorem]
	\label{duality theorem}
	For all words $w$ over $\{e_0, e_1\}$ we have
	\[ \zeta^{\fu}(w) = \zeta^{\fu}(\tau(w)). \]
\end{prop}

\begin{proof}
	Let $\iota$ be the involution $z \mapsto 1-z$ on $\bP^1 \smallsetminus \{0,1,\infty\}$. Observe that $\iota^*(\rd t/t) = - \rd t/(1-t)$ and $\iota^*(\rd t/(1-t)) = -\rd t/t$. In other words, $\iota^*(\omega_i) = - \omega_{1-i}$ for $i = 0,1$. Observe moreover that the canonical de Rham path $\gamma_{0,1}^{\dR}$ from~$0$ to~$1$ (viewed as tangential basepoints as usual) is mapped under $\iota$ to the canonical de Rham path $\gamma_{1,0}^{\dR}$ from $\iota(0) = 1$ to $\iota(1) = 0$. The inverse of the canonical path is the canonical path, hence we have $\iota(\gamma_{0,1}^{\dR}) = (\gamma_{0,1}^{\dR})^{-1}$. 
	
	Let $w = e_{i_1} \cdots e_{i_r}$ with $i_j \in \{0,1\}$ and let $\sigma$ be an element of $U_Z$ (valued in some $\bQ$-algebra omitted from the notation). Combining the change of variables formula and the path inversion formula for iterated integrals (§\ref{sec:tannakian-iterated-integrals}) we calculate
	\begin{align*}
		\zeta^{\fu}(w)(\sigma) &= \int_{\sigma(\gamma_{0,1}^{\dR})} \omega_{i_1} \cdots \omega_{i_r} \\
		&= \int_{\iota^{-1}(\sigma(\gamma_{0,1}^{\dR}))} (-\omega_{1-i_1}) \cdots (-\omega_{1-i_r}) \\
		&= (-1)^r \int_{\sigma(\gamma_{0,1}^{\dR})^{-1}} \omega_{1-i_1} \cdots \omega_{1-i_r} \\
		&= \int_{\sigma(\gamma_{0,1}^{\dR})} \omega_{1-i_r} \cdots \omega_{1-i_1} \\
		&= \zeta^{\fu}(\tau(w))(\sigma). \qedhere
	\end{align*}
\end{proof}

\begin{cor}
	\label{zeta3 relation}
	We have the identity of motivic multiple zeta values
	\[ \zeta^{\fu}(3) = \zeta^{\fu}(2,1). \]
\end{cor}

\begin{proof}
	Take $w = e_0 e_0 e_1$ in the Duality Theorem \ref{duality theorem}.
\end{proof}

\subsection{Chabauty--Kim theory}
\label{sec:chabauty-kim}

As before, let $Z = \Spec(\bZ) \smallsetminus S$ be a dense open subscheme of $\Spec(\bZ)$, and let $X$ denote the thrice-punctured line over~$Z$. We fix a quotient $\pi_1^{\mot}(X) \twoheadrightarrow \Pi$ of the fundamental group in $\MT(Z,\bQ)$. We also fix an auxiliary prime $p \in Z$. The $p$-adic étale realization $\Pi^{\et} \coloneqq \real^{\et}(\Pi)$ is a $G_{\bQ}$-equivariant quotient of the $\bQ_p$-pro-unipotent étale fundamental group $\pi_1^{\et,p}(X)$. The Chabauty--Kim method is centered around the following \emph{Chabauty--Kim diagram} (or \emph{Kim's Cutter}), whose construction we recall briefly. 

\begin{equation}
	\label{eq:ck-diagram}
	\begin{tikzcd}
		X(Z) \dar["\kappa^{\et}_Z"] \rar[hook] & X(\bZ_p) \dar["\kappa^{\et}_p"] \arrow[dr,"\kappa_p^{\dR}"] & \\
		\rH^1_Z(G_{\bQ}, \Pi^{\et})(\bQ_p) \rar["\loc_p"] & \rH^1_f(G_p, \Pi^{\et})(\bQ_p) \rar["\sim"', "\logBK"]& \Pi^{\dR}(\bQ_p).
	\end{tikzcd}
\end{equation}

\subsubsection{Selmer schemes and Kummer maps}
\label{sec:selmer-schemes}

Let $G = G_{\bQ}$ or $G_{\ell}$, the local Galois group at a prime~$\ell$. For a $\bQ_p$-algebra~$R$, define $\rH^1(G, \Pi^{\et})(R)$ as the set of isomorphism classes of $G$-equivariant $\Pi^{\et}$-torsors over~$R$. Restriction from $G_{\bQ}$ to $G_{\ell}$ defines the \emph{localization maps}
\[ \loc_{\ell}\colon \rH^1(G_{\bQ}, \Pi^{\et}) \to \rH^1(G_{\ell}, \Pi^{\et}), \]
which are morphisms of functors on $\bQ_p$-algebras. In the local case with $\ell = p$, denote by $\rH^1_f(G_p, \Pi^{\et})(R)$ the subset of \emph{crystalline} torsors, i.e., those torsors which become trivial after tensoring with Fontaine's period ring~$\Bcris$:
\[ \rH^1_f(G_p, \Pi^{\et})(R) \coloneqq \ker\bigl( \rH^1(G_p, \Pi^{\et})(R) \to \rH^1(G_p, \Pi^{\et})(R \otimes_{\bQ_p} \Bcris)\bigr). \]
On the global side, define
\[ \rH^1_Z(G_{\bQ}, \Pi^{\et})(R) \]
as the subset of $\rH^1(G_{\bQ}, \Pi^{\et})(R)$ consisting of those classes $\xi$ such that $\loc_{\ell}(\xi)$ is unramified (equivalently: trivial) for $\ell \in Z \smallsetminus \{p\}$ and $\loc_p(\xi)$ is crystalline.

The functors $R \mapsto \rH^1_Z(G_{\bQ}, \Pi^{\et})(R)$ and $R \mapsto \rH^1_f(G_p, \Pi^{\et})(R)$ are representable by affine $\bQ_p$-schemes called the \emph{global} and \emph{local Selmer scheme}, respectively. The localization map $\loc_p$ from the global to the local Selmer scheme is automatically a morphism of $\bQ_p$-schemes by Yoneda's Lemma. The vertical maps $\kappa_Z^{\et}$ and $\kappa_p^{\et}$ in Diagram~\eqref{eq:ck-diagram} are the \emph{non-abelian Kummer maps}, sending a point $x$ in $X(Z)$ resp.\ in $X(\bZ_p)$ to the $\bQ_p$-pro-unipotent path torsor $\pi_1^{\et,p}(X;0,x)$ pushed out along the quotient map $\pi_1^{\et,p}(X) \twoheadrightarrow \Pi^{\et}$.
The diagonal map $\kappa_p^{\dR}$ in the Chabauty--Kim diagram is the \emph{de Rham Kummer map} sending $x \in X(\bZ_p)$ to the de Rham loop
\[ \kappa_p^{\dR}(x) \coloneqq (\gamma_{0,x}^{\dR})^{-1} \gamma_{0,x}^{\cris} \in \Pi^{\dR}(\bQ_p) \]
obtained by composing the crystalline path from~$0$ to~$x$ (§\ref{sec:frobenius-invariant-paths}) with the inverse of the canonical de Rham path (§\ref{sec:canonical-paths}), and passing to the quotient $\pi_1^{\dR}(X) \twoheadrightarrow \Pi^{\dR}$. Algebraic functions on $\Pi^{\dR}$ pull back along $\kappa_p^{\dR}$ to Coleman functions on $X(\bZ_p)$. More precisely, functions on $\Pi^{\dR}$ are polynomials in logarithms and multiple polylogarithms, and it follows from the path composition formula and the fact that non-empty integrals over the canonical de Rham path vanish (\Cref{integral-over-canonical-de-rham-path}) that
\begin{align*}
	\log(\kappa_p^{\dR}(x)) &= \log^p(x),\\
	\Li_{n_1,\ldots,n_r}(\kappa_p^{\dR}(x)) &= \Li^p_{n_1,\ldots,n_r}(x)
\end{align*}
whenever the indicated coordinate function belongs to $\cO(\Pi^{\dR}) \subseteq \cO(\pi_1^{\dR}(X))$. Finally, the \emph{non-abelian Bloch--Kato logarithm}
\[ \logBK\colon \rH^1_f(G_p,\Pi^{\et}) \xrightarrow{\sim} \Pi^{\dR}_{\bQ_p} \]
sends a crystalline $G_p$-equivariant $\Pi^{\et}$-torsor $P^{\et}$ over a $\bQ_p$-algebra~$R$ to the loop $(\gamma^{\dR})^{-1}\gamma^{\cris} \in \Pi^{\dR}(R)$, where $\gamma^{\dR}$ and $\gamma^{\cris}$ are the Hodge element resp.\ Frobenius-fixed element of the admissible $\Pi^{\dR}_{\bQ_p}$-torsor $\rD_{\dR}(P^{\et})$ over~$R$ \cite[Proposition~1]{kim09}.

\subsubsection{Chabauty--Kim loci}
\label{sec:ck-loci}

Let $A_{\myCol}(X)$ denote the ring of Coleman functions on $X(\bZ_p)$. The de Rham Kummer map $\kappa_p^{\dR}\colon X(\bZ_p) \to \Pi_{\bQ_p}^{\dR}$ is \emph{Coleman analytic}, i.e., algebraic functions on $\Pi_{\bQ_p}^{\dR}$ pull back to Coleman analytic functions on $X(\bZ_p)$. This defines a homomorphism of $\bQ_p$-algebras
\begin{equation} 
	\label{eq:de-Rham-Kummer-pullback}
	(\kappa_p^{\dR})^{\sharp}\colon \cO(\Pi_{\bQ_p}^{\dR}) \to A_{\myCol}(X).
\end{equation}
If $f \in \cO(\Pi_{\bQ_p}^{\dR})$ is a function which vanishes on the scheme-theoretic image of the Selmer scheme $\rH^1_Z(G_{\bQ}, \Pi^{\et})$ under $\logBK \circ \loc_p$, then, by the commutativity of Diagram~\eqref{eq:ck-diagram}, the pullback $(\kappa_p^{\dR})^{\sharp} f \in A_{\myCol}(X)$ defines a function $X(\bZ_p) \to \bQ_p$ which vanishes on $X(Z)$. This motivates the following definition.

\begin{defn}
	\label{def:CK-locus}
	Let $\cI_{Z,\Pi}^{\dR} \subseteq \cO(\Pi_{\bQ_p}^{\dR})$ be the ideal of functions vanishing on the image of $\rH^1_Z(G_{\bQ}, \Pi^{\et})$ under $\logBK \circ \loc_p$, i.e.
	\[ \cI_{Z,\Pi}^{\dR} = \ker\bigl[ (\logBK \circ \loc_p)^{\sharp}\colon \cO(\Pi_{\bQ_p}^{\dR}) \to \cO(\rH^1_Z(G_{\bQ}, \Pi^{\et})) \bigr]. \]
	The \emph{Chabauty--Kim ideal} $\cI_{Z,\Pi} \subseteq A_{\myCol}(X)$ is the ideal generated by the image of $\cI_{Z,\Pi}^{\dR}$ under the pullback map~\eqref{eq:de-Rham-Kummer-pullback}. Elements of $\cI_{Z,\Pi}$ are called \emph{Kim functions}.
	The \emph{Chabauty--Kim locus} $X(\bZ_p)_{Z,\Pi} \subseteq X(\bZ_p)$ is defined as the vanishing locus of the Chabauty--Kim ideal $\cI_{Z,\Pi}$.
\end{defn}

The Chabauty--Kim locus $X(\bZ_p)_{Z,\Pi}$ is a not necessarily reduced locally analytic subspace of $X(\bZ_p)$. Its underlying set is the set of all points $x \in X(\bZ_p)$ such that $\kappa_p^{\dR}(x)$ is contained in the scheme-theoretic image of $\rH^1_Z(G_{\bQ}, \Pi^{\et})$ under $\logBK \circ \loc_p$; equivalently, $X(\bZ_p)_{Z,\Pi}$ is the common vanishing set of all Kim functions associated to $Z$ and~$\Pi$. If $f$ is a nonzero element of the ideal $\cI_{Z,\Pi}^{\dR}$, then the Kim function $(\kappa_p^{\dR})^{\sharp} f$ is a nonzero analytic function on each residue disc by \cite[Theorem~1]{kim09}, hence it has only finitely many zeros. Thus, if $\cI_{Z,\Pi}^{\dR} \neq 0$ then $X(\bZ_p)_{Z,\Pi}$ is finite.

By construction we have the inclusions
\[ X(Z) \subseteq X(\bZ_p)_{Z,\Pi} \subseteq X(\bZ_p), \]
hence the Chabauty--Kim loci constrain the location of the $Z$-integral points inside $X(\bZ_p)$. When $\Pi$ is the maximal $n$-step nilpotent quotient of $\pi_1^{\mot}(X)$, the associated Chabauty--Kim locus is denoted $X(\bZ_p)_{Z,n}$. This is often called the \emph{depth~$n$} (or \emph{level}~$n$) locus. Since larger fundamental group quotients correspond to smaller Chabauty--Kim loci, we have a descending chain
\[ X(\bZ_p) \supseteq X(\bZ_p)_{Z,1} \supseteq X(\bZ_p)_{Z,2} \supseteq \ldots  \]
of subsets of $X(\bZ_p)$ containing $X(Z)$.
In our specific case where $X = \bP^1 \smallsetminus \{0,1,\infty\}$, Kim proved that $X(\bZ_p)_{Z,n}$ is finite for $n \gg 0$ \cite{KimP1}. Kim's Conjecture states that eventually the Chabauty--Kim locus is not just a finite superset of but precisely equal to the set of $Z$-integral points \cite[§8.1]{nabsd}.

\begin{conj}[Kim's Conjecture]
	\label{conj:kim}
	$X(\bZ_p)_{Z,\Pi} = X(Z)$ for sufficiently large fundamental group quotients $\pi_1^{\mot}(X) \twoheadrightarrow \Pi$.
\end{conj}

While Kim's Conjecture is often formulated as an equality of sets, we interpret it in the slightly stronger sense of equality of locally analytic subspaces of $X(\bZ_p)$, with $X(Z)$ being given the reduced subspace structure. In other words, the conjecture predicts that the Chabauty--Kim ideal vanishes precisely on the $Z$-integral points, and that these are \emph{simple} roots.

\subsubsection{The motivic Selmer scheme}
\label{sec:motivic-selmer-scheme}

Again in the specific case of the thrice-punctured line, one can construct a \emph{motivic Selmer scheme} which provides a $\bQ$-structure for the usual (étale) Selmer scheme and at the same time admits a description in terms of algebraic cocycles which is more amenable to explicit calculations. We continue to denote by $\Pi$ a fixed quotient of the motivic fundamental group $\pi_1^{\mot}(X)$ in $\MT(Z,\bQ)$. Recall that $G_Z = \UZ \rtimes \Gm$ denotes the Tannakian fundamental group of $\MT(Z,\bQ)$ with respect to the de Rham fiber functor (§\ref{sec:mixed-tate-galois-group}). For a $\bQ$-algebra~$R$, define $\rH^1(G_Z, \Pi^{\dR})(R)$ as the set of isomorphism classes of $G_Z$-equivariant $\Pi^{\dR}$-torsors over~$R$. 

\begin{defn}
	\label{def:motivic-selmer-scheme}
	The \emph{motivic Selmer scheme} $\rH^1(G_Z, \Pi^{\dR})$ is the affine $\bQ$-scheme representing the functor $R \mapsto \rH^1(G_Z, \Pi^{\dR})(R)$.
\end{defn}

The representability is shown in \cite[Proposition~6.4]{BrownIntegral}. In §\ref{sec:coordinates-on-selmer-scheme}, we recall the construction of a set of coordinates on $\rH^1(G_Z, \Pi^{\dR})$ which identifies the motivic Selmer scheme more explicitly as an affine space over~$\bQ$ \cite{Luedtke:motivic_selmer_scheme}.

We have a \emph{motivic Kummer map}
\begin{equation}
	\label{eq:kummer-map}
	\kappa_Z\colon X(Z) \to \rH^1(G_Z, \Pi^{\dR})(\bQ)
\end{equation}
which sends $x \in X(Z)$ to the de Rham path torsor $\pi_1^{\dR}(X;0,x)$ pushed out along $\pi_1^{\dR}(X) \twoheadrightarrow \Pi^{\dR}$. It carries a $G_Z$-action since $\pi_1^{\dR}(X;0,x)$ arises as the de Rham realization of a motivic path torsor $\pi_1^{\mot}(X;0,x)$ (§\ref{sec:motivic-fundamental-group}). 

Via the Tannakian formalism, elements of $\rH^1(G_Z,\Pi^{\dR})(R)$ can be identified with isomorphism classes of $\Pi$-torsors over~$R$ in $\MT(Z,\bQ)$. Taking the $p$-adic étale realization of such a torsor results in a $G_{\bQ}$-equivariant $\Pi^{\et}$-torsor which is unramified at all $\ell \in Z \smallsetminus \{p\}$ and crystalline at~$p$. The resulting map
\begin{equation}
	\label{eq:motivic-etale-comparison}
	\rH^1(G_Z, \Pi^{\dR}) \otimes_{\bQ} \bQ_p \to \rH^1_Z(G_{\bQ}, \Pi^{\et}) 
\end{equation}
is an isomorphism of $\bQ_p$-schemes \cite[Theorem~3.2]{BKL:chabauty_kim_sc}. This expresses a compatibility between the motivic and the classical (étale) Chabauty--Kim method.

\subsubsection{Cohomology and cocycles}
\label{sec:cocycles}

A more explicit description of the motivic Selmer scheme is given in terms of $\bG_m$-equivariant cocycles. Recall that the mixed Tate motivic Galois group $G_Z$ has the structure of a semidirect product $U_Z \rtimes \bG_m$ (§\ref{sec:mixed-tate-galois-group}); in particular, $\bG_m$ acts on $U_Z$ as well as on $\Pi^{\dR}$. For a $\bQ$-algebra~$R$, denote by
\[ \rZ^1(U_Z, \Pi^{\dR})^{\bG_m}(R) \coloneqq \rZ^1((U_Z)_R, \Pi^{\dR}_R)^{\bG_m} \]
the set of \emph{$\bG_m$-equivariant cocycles} $c\colon (U_Z)_R \to \Pi^{\dR}_R$. This means that $c$ is a morphism of $R$-schemes, compatible with the $\bG_m$-action on $U_Z$ and $\Pi^{\dR}$, and satisfying the non-abelian cocycle condition
\[ c(u_1 u_2) = c(u_1) \cdot u_1(c(u_2)) \]
for $u_1, u_2 \in U_Z(R')$ and $R$-algebras~$R'$.

\begin{prop}[{\cite[Theorem~A.4]{corwin:mixed_elliptic}, \cite[§4]{Luedtke:motivic_selmer_scheme}}]
	\label{cocycle-description}
	There is a canonical isomorphism
	\begin{equation}
		\label{eq:H1-cocycles}
		\rH^1(G_Z,\Pi^{\dR}) \cong \rZ^1(U_Z, \Pi^{\dR})^{\bG_m}.
	\end{equation}
\end{prop}

\begin{proof}
	We just sketch the construction of the isomorphism~\eqref{eq:H1-cocycles}. Given a $G_Z$-equivariant $\Pi^{\dR}$-torsor~$P$ over some $\bQ$-algebra~$R$, there exists a unique $\bG_m$-invariant point $\gamma^{\dR} \in P(R)$. (The existence boils down to $\bG_m$ being reductive and $\Pi^{\dR}$ pro-unipotent; the uniqueness follows from $(\Pi^{\dR})^{\bG_m} = 1$, a consequence of $\Lie(\Pi^{\dR})$ living in strictly negative weights.) The canonical element $\gamma^{\dR}$ defines a cocycle $c_P\colon (G_Z)_R \to \Pi^{\dR}_R$ by
	\[ c_P(g) = (\gamma^{\dR})^{-1} \cdot g(\gamma^{\dR}). \]
	Since $\gamma^{\dR}$ is $\bG_m$-invariant, the restriction of $c_P$ to $U_Z$ is $\bG_m$-equivariant. The association $[P] \mapsto c_P\vert_{U_Z}$ defines the isomorphism~\eqref{eq:H1-cocycles}.
\end{proof}

For a path torsor $\pi_1^{\mot}(X;x,y)$, the unique $\bG_m$-invariant path in $\pi_1^{\dR}(X;x,y)(\bQ)$ is the canonical de Rham path $\gamma_{x,y}^{\dR}$ from §\ref{sec:canonical-paths}. Indeed, the corresponding algebra homomorphism $(\gamma_{x,y}^{\dR})^{\sharp}\colon \cO(\pi_1^{\dR}(X;x,y)) \to \bQ$ vanishes on $F^1$, hence it is given by projecting onto $\gr^W_0 \cO(\pi_1^{\dR}(X;x,y)) = \bQ$, which is a $\bG_m$-equivariant map. The corresponding Kummer cocycles can be used to give an alternative definition of the motivic iterated integrals from §\ref{sec:motivic-iterated-integrals}. This is the definition used in \cite[§3.2]{CDC:polylog1} for example.

\begin{lemma}
	\label{motivic-integrals-via-kummer-cocycles}
	Let $\alpha$ be an element of $X(Z)$ or a $Z$-integral tangent vector, and let $w$ be a word in the symbols $\{e_0,e_1\}$ such that $\Li_w \in \cO(\Pi^{\dR}) \subseteq \cO(\pi_1^{\dR}(X))$. Let $\kappa_\alpha\colon U_Z \to \Pi^{\dR}$ be the $\bG_m$-equivariant cocycle representing the path torsor $\pi_1^{\mot}(X;0,\alpha)$. Then we have
	\[ \Li_w^{\fu}(\alpha) = \kappa_{\alpha}^{\sharp}(\Li_w). \]
\end{lemma}

\begin{proof}
	We may assume that $\Pi = \pi_1^{\mot}(X)$ is the full fundamental group. Let $\gamma_{0,\alpha}^{\dR} \in \pi_1^{\dR}(X;0,\alpha)(\bQ)$ be the $\bG_m$-invariant path, so that the Kummer cocycle is given by $\kappa_{\alpha}(u) = (\gamma_{0,\alpha}^{\dR})^{-1} \cdot u(\gamma_{0,\alpha}^{\dR})$. Let $\omega_1\cdots\omega_n$ be the sequence of differential forms corresponding to~$w$. Then for $u \in U_Z$ (valued in some $\bQ$-algebra) we have
	\begin{align*}
		\Li_w(\kappa_{\alpha}(u)) &= \int_{(\gamma_{0,\alpha}^{\dR})^{-1} u(\gamma_{0,\alpha}^{\dR})} \omega_1 \cdots \omega_n \\
		&= \sum_{i=0}^n \left(\int_{\gamma_{\alpha,0}^{\dR}} \omega_1 \cdots \omega_i \right) \left(\int_{u(\gamma_{0,\alpha}^{\dR})} \omega_{i+1} \cdots \omega_n \right) \\
		&= \int_{u(\gamma_{0,\alpha}^{\dR})} \omega_1\cdots\omega_n \\
		&= \Li_w^{\fu}(\alpha)(u),
	\end{align*}
	where we have used the path composition formula for iterated integrals (§\ref{sec:tannakian-iterated-integrals}), the fact that $(\gamma_{0,\alpha}^{\dR})^{-1} = \gamma_{\alpha,0}^{\dR}$, and the fact that non-empty integrals along canonical de Rham paths vanish (\Cref{integral-over-canonical-de-rham-path}). 
\end{proof}

\subsubsection{The cocycle evaluation map}
\label{sec:cocycle-evaluation-map}

Recall from §\ref{sec:period-point} that we have a distinguished $p$-adic period point $\eta_p \in U_Z(\bQ_p)$. Evaluation at $\eta_p$ defines the \emph{$p$-adic cocycle evaluation map}
\begin{equation}
	\label{eq:p-adic-cocycle-evaluation}
	\ev_p \colon \rH^1(G_Z, \Pi^{\dR})_{\bQ_p} = \rZ^1(U_Z, \Pi^{\dR})_{\bQ_p}^{\bG_m} \to \Pi_{\bQ_p}^{\dR}.
\end{equation}
Under the identification~\eqref{eq:motivic-etale-comparison} of the motivic and the étale Selmer scheme, the cocycle evaluation map equals the localization map, in other words, the following diagram commutes \cite[Theorem~3.2]{BKL:chabauty_kim_sc}:
\[ 
\begin{tikzcd}
	\rH^1_Z(G_{\bQ}, \Pi^{\et}) \dar["\simeq", "\eqref{eq:motivic-etale-comparison}"'] \rar["\loc_p"] & \rH^1_f(G_p, \Pi^{\et}) \dar["\simeq", "\logBK"'] \\
	\rH^1(G_Z, \Pi^{\dR})_{\bQ_p} \rar["\ev_p"] & \Pi^{\dR}_{\bQ_p}.
\end{tikzcd}
\]
Using this diagram, the construction of Kim functions (\Cref{def:CK-locus}) boils down to finding functions on $\Pi_{\bQ_p}^{\dR}$ which vanish on the image of the cocycle evaluation map:
\[ \cI_{Z,\Pi}^{\dR} = \ker\bigl[ \ev_p^{\sharp}\colon \cO(\Pi_{\bQ_p}^{\dR}) \to \cO(\rH^1(G_Z, \Pi^{\dR})_{\bQ_p}) \bigr]. \]

\subsubsection{Motivic Kim functions}
\label{sec:motivic-kim-functions}

It is possible to perform this step uniformly for all primes~$p \in Z$ by considering the \emph{universal cocycle evaluation map}
\begin{equation}
	\label{eq:universal-cocycle-evaluation}
	\begin{split}
		\ev\colon \rZ^1(U_Z, \Pi^{\dR})^{\bG_m} \times U_Z \to \Pi^{\dR} \times U_Z,\\
		(c,u) \mapsto (c(u), u).
	\end{split}
\end{equation}
One can view this map as evaluating cocycles at the universal point $\eta_{\univ} \in U_Z(A(Z))$, where $A(Z) = \cO(U_Z)$ is the ring of motivic periods. The $p$-adic evaluation map for some prime $p \in Z$ is then obtained by specializing to the $p$-adic period point $\eta_p \in U_Z(\bQ_p)$, or equivalently by tensoring along the period map $\per_p\colon A(Z) \to \bQ_p$. 

\begin{defn}[{\cite[Definition~2.22]{CDC:polylog1}}]
	\label{def:motivic-kim-function}
	Let $K(Z)$ be the fraction field of $A(Z)$. The \emph{motivic Chabauty--Kim ideal} $\cI_{Z,\Pi}^{\mot} \subseteq \cO(\Pi^{\dR}) \otimes K(Z)$ is the ideal defining the image of the universal cocycle evaluation map~\eqref{eq:universal-cocycle-evaluation} base changed to $K(Z)$:
	\[ \cI_{Z,\Pi}^{\mot} \coloneqq \ker\bigl(\ev^{\sharp}\colon \cO(\Pi^{\dR}_{K(Z)}) \to \cO(\rH^1(G_Z, \Pi^{\dR})_{K(Z)}) \bigr). \]
	Elements of $\cI_{Z,\Pi}^{\mot}$ are called \emph{motivic Kim functions}.
\end{defn}

The motivic Chabauty--Kim ideal is thus an ideal in $K(Z) \otimes \cO(\Pi^{\dR})$, and its elements are polynomials in iterated integrals with $K(Z)$-coefficients. A motivic Kim function all of whose coefficients lie in $A(Z)$ (which can always be achieved by rescaling) specializes to a $p$-adic Kim function (\Cref{def:CK-locus}) for every prime $p \in Z$ via base change along the $p$-adic period map $\per_p\colon A(Z) \to \bQ_p$ and pulling back along $\kappa_p^{\dR}$.
The $p$-adic period conjecture (\Cref{conj:period-conjecture}) implies that a non-trivial motivic Kim function remains non-trivial $p$-adically, and moreover that the whole Chabauty--Kim ideal $\cI_{Z,\Pi}$ is generated by $p$-adic specializations of motivic Kim functions, cf. \cite[Remark 2.28]{CDC:polylog1} and \cite[Remark~3.3]{BKL:chabauty_kim_sc}.

\subsubsection{Normalized generators of the motivic Lie algebra}
\label{sec:normalized-generators}

Recall that the Lie algebra $\fn(Z)$ is freely generated by a set $\Sigma = \{\tau_{\ell}\}_{\ell \in S} \cup \{\sigma_3,\sigma_5,\sigma_7,\ldots\}$ with $\tau_{\ell}$ living in half-weight~$-1$ and $\sigma_n$ living in half-weight~$-n$ (§\ref{sec:mixed-tate-galois-group}). 
The generators $\sigma_n$ for $n \geq 3$ odd are not canonical. However, they can be chosen to satisfy
\begin{equation} \label{eq:sigma-normalisation}
	\langle \sigma_n, \zeta^{\fu}(n) \rangle = 1.
\end{equation}
This is the content of \cite[Proposition~4.1]{CDC:polylog1}, whose proof we briefly summarize. Define $E_n(Z)$ as the space of primitive elements in $A(Z)_n$. There is a canonical isomorphism
\[ E_n(Z) \cong \Ext^1_{\MT(Z,\bQ)}(\bQ(0),\bQ(n)) \]
and $E_n(Z)$ is called the subspace of \emph{extensions}. The Lie algebra $\fn(Z)$ embeds into the dual space $A(Z)^{\vee}$ as the space of derivations at the identity. Restricting to the extension space $E_n(Z)$ gives a surjective linear map $\fn(Z)_{-n} \twoheadrightarrow E_n^{\vee}$ which is canonically isomorphic to the abelianization map $\fn(Z)_{-n} \twoheadrightarrow \fn(Z)_{-n}^{\ab}$. Thus, an element of $\fn(Z)_{-n}^{\ab}$ is given by specifying its values on a basis of the extension space $E_n(Z)$. When $n \geq 3$ is odd, $E_n(Z)$ is one-dimensional and spanned by the motivic zeta value $\zeta^{\fu}(n)$. The fact that $\zeta^{\fu}(n)$ is primitive follows from the coproduct formula \cite[Proposition~3.1]{CDC:polylog1}; the fact that $\zeta^{\fu}(n)$ is non-zero is \Cref{motivic-zetas-nonzero}. Thus, $\zeta^{\fu}(n) \mapsto 1$ defines a normalized generator of $\fn(Z)_{-n}^{\ab} = E_n(Z)^{\vee}$, and we choose $\sigma_n \in \fn(Z)_{-n}$ as an arbitrary lift. 

The generators $\{\tau_{\ell}\}_{\ell \in S}$ in half-weight~$-1$ are canonical and can be defined as the dual basis of the basis $\{\log^{\fu}(\ell)\}_{\ell \in S}$ of $E_1(Z) = A_1(Z)$ given by motivic logarithms.

Let $\eta_{\univ} \in U_Z(A(Z))$ be the \emph{universal point}, i.e., the $A(Z)$-valued point corresponding to the identity $\Spec A(Z) \to U_Z$, and let 
\begin{equation}
	\label{eq:universal-point-expansion}
	\eta_{\univ} = \sum_u a_u u \in A(Z)\llangle \Sigma \rrangle
\end{equation}
be its expansion as a group-like element in the chosen generators $\Sigma$ of $U_Z$, with $u$ running over all words on~$\Sigma$. Thus, $a_u \in A(Z)$ is the coefficient extraction functional $f_u$ (\Cref{def:coeff-extraction-functions}) sending a point of $U_Z$ to its $u$-coefficient when expanded as a group-like power series. By~\eqref{eq:shuffle-product}, the elements $a_u$ satisfy the \emph{shuffle identity} 
\begin{equation}
	\label{eq:shuffle-identity}
	a_{u_1} a_{u_2} = \sum_{\pi \in \Sh(|u_1|,|u_2|)} a_{\pi(u_1 u_2)}
\end{equation}
for all words $u_1,u_2$ on~$\Sigma$. Applying the shuffle identity in the case of a power $a_{\sigma}^n$ yields
\begin{equation}
	\label{eq:shuffle-identity-for-power}
	a_{\sigma^n} = \frac1{n!} a_{\sigma}^n
\end{equation}
for all $\sigma \in \Sigma$ and $n \geq 0$.

A single-letter motivic period $a_{\sigma}$, for $\sigma \in \Sigma$ in half-weight~$-n$, is determined by the image of $\sigma$ in the abelianization $\fn(Z)^{\ab}$, i.e., by the restriction of $\sigma$ to the extension space $E_n(Z)$.

\begin{prop}[{\cite[Proposition~4.1]{CDC:polylog1}}]
	\label{single-letter-periods}
	With normalized generators $\Sigma = \{\tau_{\ell}\}_{\ell \in S} \cup \{\sigma_3,\sigma_5,\sigma_7,\ldots\}$ of $\fn(Z)$ chosen as above, we have
	\begin{align}
		\label{eq:a-tau}
		a_{\tau_{\ell}} &= \log^{\fu}(\ell) \quad \text{for $\ell \in S$},\\
		\label{eq:a-sigma}
		a_{\sigma_n} &= \zeta^{\fu}(n) \quad \text{for $n \geq 3$ odd}. 
	\end{align}
\end{prop}

\subsubsection{Coordinates on the motivic Selmer scheme}
\label{sec:coordinates-on-selmer-scheme}

We recall the construction of a set of coordinates on the motivic Selmer scheme $\rH^1(G_Z, \Pi^{\dR})$ which exhibit it as an affine space and allow us to derive formulas for the cocycle evaluation map~\eqref{eq:p-adic-cocycle-evaluation}. Such coordinates were constructed in \cite[§3.3]{CDC:polylog1} in the case of the polylogarithmic quotient and in \cite{Luedtke:motivic_selmer_scheme} in the general case. The construction depends on a choice of free generators $\Sigma = \{\tau_{\ell}\}_{\ell \in S} \cup \{\sigma_{2k+1}\}_{k\geq 1}$ of the Lie algebra $\fn(Z)$ (§\ref{sec:normalized-generators}). The $\tau_{\ell}$ are canonical but the $\sigma_{2k+1}$ are not. Assume that $\pi_1^{\mot}(X) \twoheadrightarrow \Pi_n$ is the quotient by the $(n+1)$th step of the descending central filtration. A basis of the associated Lie algebra $\Lie(\Pi_n^{\dR})$ is given by the standard bracketings $\lambda(w)$ of the Lyndon words $w$ of length~$\leq n$ in the symbols $\{e_0,e_1\}$ (§\ref{sec:lyndon-basis}). Denote the dual basis by $\{L_w\}$. By \Cref{cocycle-description}, the motivic Selmer scheme $\rH^1(G_Z, \Pi_n^{\dR})$ can be identified with the space of $\bG_m$-equivariant cocycles $U_Z \to \Pi_n^{\dR}$, which can in turn be identified with the space of graded Lie algebra cocycles $c\colon \fn(Z) \to \Lie(\Pi_n^{\dR})$. Such a cocycle is completely determined by the images of the generators in~$\Sigma$. Set
\begin{align*}
	x_{\ell}(c) &\coloneqq L_{e_0}(c(\tau_{\ell})), \\
	y_{\ell}(c) &\coloneqq L_{e_1}(c(\tau_{\ell})), \\
	z_w(c) &\coloneqq L_w(c(\sigma_{2k+1}))
\end{align*}
for $\ell \in S$, $k \geq 1$ and Lyndon words~$w$ of length $2k+1$, with $1 \leq k \leq \lfloor (n-1)/2 \rfloor$.

\begin{thm}[{\cite[Theorem~6.4]{Luedtke:motivic_selmer_scheme}}]
	\label{coordinates-on-selmer-scheme}
	The coordinates $x_{\ell}$, $y_{\ell}$, $z_w$ define an isomorphism
	\[ \rH^1(G_Z, \Pi_n^{\dR}) \cong \Spec \bQ[(x_{\ell})_{\ell \in S}, (y_{\ell})_{\ell \in S}, (z_w)_w]. \]
\end{thm}

Any Lyndon word $w$ of length $\geq 2$ can be written as $w = e_0^{k_1-1}e_1\cdots e_0^{k_r-1}e_1$ and we use $z_{k_1,\ldots,k_r}$ as an alternative notation for $z_w$.

\begin{example}
	\label{ex:coordinates-depth-4}
	Taking $n = 4$, the motivic Selmer scheme $\rH^1(G_Z, \Pi_4^{\dR})$ is an affine space of dimension $2\# S + 2$ with coordinates $(x_{\ell})_{\ell \in S}$, $(y_{\ell})_{\ell \in S}$, $z_3$, $z_{2,1}$. Here, $z_3$ and $z_{2,1}$ correspond to the two Lyndon words $e_0e_0e_1$ and $e_0e_1e_1$ of length~$3$.
\end{example}

On the image of the motivic Kummer map \eqref{eq:kummer-map}, the coordinates $x_{\ell}$ and $y_{\ell}$ are simply given by $\ell$-adic valuations:

\begin{lemma}
	\label{xy-on-kummer}
	Assume that the generators $\tau_{\ell} \in \fn(Z)_{-1}$ are dual to the motivic logarithms $\log^{\fu}(\ell)\in A_1(Z)$ ($\ell \in S$). Then, for all $\alpha \in X(Z)$, the element $\kappa_Z(\alpha) \in \rH^1(G_Z,\Pi_n^{\dR})$ satisfies
	\begin{align*}
		x_{\ell}(\kappa_Z(\alpha)) &= v_{\ell}(\alpha), \\
		y_{\ell}(\kappa_Z(\alpha)) &= -v_{\ell}(1-\alpha)	
	\end{align*}
	for all $\ell \in S$.
\end{lemma}

\begin{proof}
	Let $\kappa_{\alpha}\colon U_Z \to \Pi_n^{\dR}$ be the $\bG_m$-equivariant cocycle representing $\kappa_Z(\alpha)$, so that $x_{\ell}(\kappa_Z(\alpha)) = L_{e_0}(\kappa_{\alpha}(\tau_{\ell}))$ and $y_{\ell}(\kappa_Z(\alpha)) = L_{e_1}(\kappa_{\alpha}(\tau_{\ell}))$ by definition of $x_{\ell}, y_{\ell}$.
	We may assume that $\Pi_n = \Pi_{\infty} = \pi_1^{\mot}(X)$ is the full fundamental group. The Lie algebra $\Lie(\Pi_{\infty}^{\dR})$ is the subspace of primitive elements in $\cO(\Pi_{\infty}^{\dR})^{\vee} = \bQ\llangle e_0,e_1 \rrangle$, and the linear functional $L_{e_i}$ is the restriction of the coefficient extraction functional $f_{e_i}\colon \bQ\llangle e_0,e_1 \rrangle \to \bQ$, which agrees with the function $\Li_{e_i} \in \cO(\Pi_{\infty}^{\dR})$, for $i=0,1$ (\Cref{paths-as-series}). We obtain
	\[ L_{e_i}(\kappa_{\alpha}(\tau_{\ell})) = \langle \kappa_{\alpha}^{\sharp} \Li_{e_i}, \tau_{\ell} \rangle = \langle \Li_{e_i}^{\fu}(\alpha), \tau_{\ell} \rangle, \]
	where $\kappa_{\alpha}^{\sharp} \Li_{e_i} = \Li_{e_i}^{\fu}(\alpha)$ by \Cref{motivic-integrals-via-kummer-cocycles}. We have
	\[ \Li_{e_0}^{\fu}(\alpha) = \log^{\fu}(\alpha) = \sum_{\ell' \in S} v_{\ell'}(\alpha) \log^{\fu}(\ell') \]
	and 
	\[ \Li_{e_1}^{\fu}(\alpha) = \Li_1^{\fu}(\alpha) = -\log^{\fu}(1-\alpha) = -\sum_{\ell' \in S} v_{\ell'}(1-\alpha) \log^{\fu}(\ell'), \]
	and pairing with $\tau_{\ell}$ yields the claimed valuations $v_{\ell}(\alpha)$ and $-v_{\ell}(1-\alpha)$, respectively.
\end{proof}

\subsubsection{Refined Chabauty--Kim theory}
\label{sec:refined-chabauty-kim}

We also consider a strengthening of the Chabauty--Kim method developed by Betts--Dogra \cite{BD:refined} based on the notion of \emph{refined Selmer schemes}. We recall the definition of refined Selmer schemes and associated refined Chabauty--Kim loci, and define a motivic variant with $\bQ$-coefficients.

Let $S$ be a finite set of primes, $Z = \Spec(\bZ[1/S])$, and let $p$ be a prime not contained in~$S$. Let $\Pi$ be a quotient of the unipotent motivic fundamental group of $X$ which surjects onto $\rH_1^{\mot}(X)$. For $\ell \in S$, we let $\fk{S}_{\ell}$ be the Zariski closure of the image of the $p$-adic unipotent local Kummer map at $\ell$
\[
\ka_{\ell}:X(\QQ_{\ell}) \to \rH^1(G_{\ell}, \Pi^{\et});
\]
for $\ell \notin S$ we let $\fk{S}_{\ell}$ be the Zariski closure of $\ka_{\ell} \big( X(\ZZ_{\ell}) \big)$. Recall that $\m{loc}_{\ell}$ denotes the localization map
\[
\rH^1_Z(G_{\bQ}, \Pi^{\et}) \to \rH^1(G_{\ell}, \Pi^{\et}) 
\]
from the global étale Selmer scheme to the local Selmer scheme at~$\ell$ (§\ref{sec:selmer-schemes}).

\begin{defn}
	\label{20105B}
	The \emph{refined $p$-adic \'etale Selmer scheme} of Betts--Dogra \cite{BD:refined} is the closed subscheme of the global $p$-adic \'etale Selmer scheme $\rH^1_Z(G_{\bQ}, \Pi^{\et})$ defined by
	\[
	\Sel_{Z,\Pi}^{\et,\min}(X) := \bigcap_{\ell} \m{loc}_{\ell}^{-1} \fk{S}_{\ell},
	\]
	where the intersection runs over all primes $\ell$. The \emph{refined Chabauty--Kim ideal} $\cI_{Z,\Pi}^{\min}$ and the \emph{refined Chabauty--Kim locus} $X(\bZ_p)_{Z,\Pi}^{\min}$ are defined as in \Cref{def:CK-locus} but using the refined Selmer scheme in place of the whole Selmer scheme; i.e., $\cI_{Z,\Pi}^{\min}$ is the ideal of $A_{\myCol}(X)$ generated by all $(\kappa_p^{\dR})^{\sharp} f$ for $f \in \cO(\Pi_{\bQ_p}^{\dR})$ vanishing on the image of $\Sel_{Z,\Pi}^{\et,\min}(X)$ under $\logBK \circ \loc_p$. 
\end{defn}

The refined Chabauty--Kim locus is potentially smaller than the unrefined one but we still have the inclusion
\[ X(Z) \subseteq X(\bZ_p)_{Z,\Pi}^{\min} \subseteq X(\bZ_p) \]
since the global Kummer map $\kappa_Z^{\et}$ in the Chabauty--Kim diagram~\eqref{eq:ck-diagram} lands in the refined Selmer scheme by construction. It therefore provides potentially stronger constraints on the location of the $Z$-integral points inside the $p$-adic points. 
We have the following weaker version of Kim's Conjecture:

\begin{conj}[``Refined Kim's Conjecture'']
	\label{conj:refined-kim}
	$X(\bZ_p)_{Z,\Pi}^{\min} = X(Z)$ for sufficiently large fundamental group quotients $\pi_1^{\mot}(X) \twoheadrightarrow \Pi$.
\end{conj}

\begin{rmk}
	\label{rem:refinement-conditions}
	At primes $\ell \not \in S$, the local refined Selmer schemes $\fk{S}_{\ell}$ are given by
	\[ \fk{S}_{\ell} = \begin{cases}
		\emptyset, & \text{if $\ell = 2$},\\
		\{\ast\}, & \text{if $\ell \neq 2,p$},\\
		\rH^1_f(G_p, \Pi^{\et}), & \text{if $\ell = p \neq 2$}.
	\end{cases} \]
	Thus, the only additional local conditions defining the global refined Selmer scheme inside $\rH^1_Z(G_{\bQ},\Pi^{\et})$ are at the primes inside~$S$ (and possibly at $2$). 
\end{rmk}

Recall from Eq.~\eqref{eq:motivic-etale-comparison} that the unrefined motivic Selmer scheme provides the $p$-adic \'etale Selmer scheme $\rH^1_Z(G_{\bQ}, \Pi^{\et})$ with a canonical $\bQ$-structure, i.e., the evident map is an isomorphism of affine schemes
\begin{equation}
	\label{eq:motivic-etale-comparison-again}
	\rH^1(G_Z, \Pi^{\dR})_\Qp \cong \rH^1_Z(G_{\bQ}, \Pi^{\et}).
\end{equation}
It is natural to ask whether the refined $p$-adic étale Selmer scheme admits a $\bQ$-structure as a subscheme of the motivic Selmer scheme as well. This is indeed the case; however, as the following remark shows, even if one assumes motivic conjectures, the naively expected construction of such a $\QQ$-structure does not work.

\begin{rmk}
	\label{20105D}
	Assume Beilinson--Soul\'e vanishing holds for $\QQ_{\ell}$ so that we have a Tannakian category of mixed Tate motives over $\Spec \QQ_{\ell}$ with fundamental group $G_{\ell}^{\rMT}$. Then one might expect to be able to substitute $\rH^1(G_{\ell}^{\rMT}, \Pi^{\dR})$ for $\rH^1(G_{\ell}, \Pi^{\et})$ in Definition \ref{20105B}, and in this way to obtain a motivic version of the refined Selmer scheme of Betts--Dogra. However, the functor on $\bQ$-algebras
	\begin{equation*}
		R \mapsto 
		\rH^1((G_{\ell}^{\rMT})_R, \Pi^{\dR}_R)
	\end{equation*}
	(which we denote simply by $\rH^1(G_{\ell}^{\rMT}, \Pi^{\dR})$) is not as well-behaved as $\rH^1(G_{\ell}, \Pi^{\et})$. As an example, we may consider the case $\Pi = \rH_1^{\mot}(X) = \QQ(1)^2$ and drop the power of 2. Restricting attention to $\QQ$-points, we then have
	\begin{align*}
		\rH^1(G_{\ell}^{\rMT}, \QQ(1))
		&\simeq 
		K_1^{(1)}(\QQ_{\ell})
		\\
		&\simeq
		\QQ_{\ell}^* \otimes \QQ
		\\
		& \simeq
		\Big(\ZZ \times \ZZ_{\ell} \times (\QQ_{\ell}^*)_{\m{tor}} \Big)
		\otimes \QQ
		\\
		& \simeq
		\QQ \times \QQ_{\ell}
	\end{align*}
	which is not the set of $\QQ$-points of a vector group over $\QQ$. Our goal here is to show in an ad-hoc way that the refined Selmer scheme of Definition \ref{20105B} is nevertheless rationally defined. 
\end{rmk}

Note that when $2 \not\in S$, the fact that $X(\bZ_2) = \emptyset$ implies that the refined étale Selmer scheme is empty, so a $\bQ$-structure is trivially given by the empty scheme. We thus assume $2 \in S$ in the following.
As in §\ref{sec:coordinates-on-selmer-scheme}, the motivic Selmer scheme is equipped with canonical functions $x_{\ell}$ and $y_{\ell}$ for $\ell \in S$. They are defined by
\[ c(\tau_{\ell}) = x_{\ell}(c) e_0 + y_{\ell}(c) e_1 \]
for graded Lie algebra cocycles $c\colon \fn(Z) \to \Lie(\Pi^{\dR})$, with $\{\tau_{\ell}\}_{\ell}$ being the canonical Lie algebra generators of $\fn(Z)$ in half-weight~$-1$ dual to the motivic logarithms $\{\log^{\fu}(\ell)\}_{\ell}$ in $A(Z)_1$.

\begin{defn}
	\label{def:refined-motivic-selmer-scheme}
	Assume $2 \in S$. The \emph{refined motivic Selmer scheme} is the closed subscheme 
	\[ \Sel_{Z,\Pi}^{\min}(X) \subseteq \rH^1(G_Z,\Pi^{\dR}) \]
	defined by the equations
	\[ x_{\ell} y_{\ell}(x_{\ell} + y_{\ell}) = 0 \quad \text{for all $\ell \in S$}. \]
\end{defn}

The refined motivic Selmer scheme can be written as a union
\[ \Sel_{Z,\Pi}^{\min}(X) = \bigcup_{\Sigma} \Sel_{Z,\Pi}^{\Sigma}(X) \]
over the $3^{\# S}$ \emph{refinement conditions} $\Sigma = (\Sigma_{\ell})_{\ell \in S} \in \{0,1,\infty\}^S$, where $\Sel_{Z,\Pi}^{\Sigma}(X)$ is defined by the equations
\begin{align*}
	y_{\ell} &= 0 \quad \text{if $\Sigma_{\ell} = 0$},\\
	x_{\ell} &= 0 \quad \text{if $\Sigma_{\ell} = 1$},\\
	x_{\ell} + y_{\ell} &= 0 \quad \text{if $\Sigma_{\ell} = \infty$}
\end{align*}
for $\ell \in S$. 

\begin{rmk}
	The equations defining the refined Selmer scheme come from the fact that the map 
	\[ X(\bQ_{\ell}) \to \bQ^2,\quad \alpha \mapsto (v_{\ell}(\alpha),-v_{\ell}(1-\alpha)) \]
	lands inside the union of the three lines $\{y = 0\} \cup \{x = 0\} \cup \{x+y = 0\}$, see \cite[Lemma~2.9]{BBKLMQSX}.
\end{rmk}

\begin{prop}
	\label{20105F}
	The base change
	\[
	\Sel_{Z,\Pi}^{\min}(X)_{\QQ_p} \subseteq \rH^1(G_Z, \Pi^{\dR})_{\QQ_p}
	\]
	corresponds under the isomorphism \eqref{eq:motivic-etale-comparison-again} to the refined $p$-adic \'etale Selmer scheme
	\[
	\Sel_{Z,\Pi}^{\et,\min}(X)
	\subseteq
	\rH^1_Z(G_{\bQ}, \Pi^{\et}).
	\]
\end{prop}

\begin{proof}
	This follows from the explicit description of the refined $p$-adic étale Selmer scheme given in \cite[§4.3--§4.4]{BKL:chabauty_kim_sc}.
\end{proof}

Using the decomposition
\begin{equation}
	\label{eq:etale-refined-Selmer-scheme-decomposition}
	\Sel_{Z,\Pi}^{\et,\min}(X) = \bigcup_{\Sigma} \Sel_{Z,\Pi}^{\et,\Sigma}(X),
\end{equation}
we also define the $\Sigma$-refined Chabauty--Kim ideal $\cI_{Z,\Pi}^{\Sigma}$ and $\Sigma$-refined Chabauty--Kim locus $X(\bZ_p)_{Z,\Pi}^{\Sigma}$. As a set, $X(\bZ_p)_{Z,\Pi}^{\Sigma}$ consists of all $x \in X(\bZ_p)$ for which $\kappa_p^{\dR}(x)$ lies in the scheme-theoretic image of $\Sel_{Z,\Pi}^{\et,\Sigma}(X)$ under $\logBK \circ \loc_p$.
From \eqref{eq:etale-refined-Selmer-scheme-decomposition} we get
\[ X(\bZ_p)_{Z,\Pi}^{\min}=\bigcup_{\Sigma} X(\bZ_p)_{Z,\Pi}^{\Sigma}, \]
corresponding to the decomposition $X(Z) = \bigcup_{\Sigma} X(Z)_{\Sigma}$ where
\[ X(Z)_{\Sigma} = \left\{x \in X(Z) : \red_{\ell}(x) \in (X \cup \{\Sigma_{\ell}\})(\bF_{\ell}) \text{ for $\ell \in S$} \right\}. \]
Here, the mod-$\ell$ reduction map is
\[ \red_{\ell}\colon X(\bQ_{\ell}) \subseteq \bP^1(\bQ_{\ell}) = \bP^1(\bZ_{\ell}) \twoheadrightarrow \bP^1(\bF_{\ell}). \]
It is natural to formulate a $\Sigma$-wise version of Kim's Conjecture:

\begin{conj}[$\Sigma$-refined Kim's Conjecture]
	\label{conj:Sigma-refined-Kims-conjecture}
	Let $\Sigma \in \{0,1,\infty\}^S$ be a refinement condition. Then $X(\bZ_p)_{Z,\Pi}^{\Sigma} = X(Z)_{\Sigma}$ for sufficiently large fundamental group quotients~$\Pi$.
\end{conj}

As in the unrefined case, we interpret this to include the assertion that the Chabauty--Kim locus is reduced. If for some $\Pi$, \Cref{conj:Sigma-refined-Kims-conjecture} holds for every $\Sigma$ then the total refined Kim's Conjecture (\Cref{conj:refined-kim}) holds with the same~$\Pi$. The converse is not true, cf.\ \cite[Remark~2.4]{luedtke:refined}.

We also have the refined variants 
\begin{equation}
	\label{eq:refined-motivic-ck-ideal}
	\cI_{Z,\Pi}^{\mot,\min},\; \cI_{Z,\Pi}^{\mot,\Sigma} \subseteq K(Z) \otimes \cO(\Pi^{\dR})
\end{equation}
of the motivic Chabauty--Kim ideal (\Cref{def:motivic-kim-function}), which define the image of $\Sel_{Z,\Pi}^{\min}(X)$ resp.\ $\Sel_{Z,\Pi}^{\Sigma}(X)$ under the universal cocycle evaluation map \eqref{eq:universal-cocycle-evaluation}.

As noted in \cite[Lemma~4.12]{BKL:chabauty_kim_sc}, when $\Pi = \Pi_n$ is the full level~$n$ quotient of the fundamental group for some $1 \leq n \leq \infty$, the natural $S_3$-action on $X$ generated by $x \mapsto 1/x$ and $x \mapsto 1-x$ preserves the refined Chabauty--Kim locus $X(\bZ_p)_{Z,n}^{\min}$ and permutes the $\Sigma$-refined subsets according to the rule
\begin{equation}
	\label{eq:S3-action-on-refined-loci}
	\sigma(X(\bZ_p)_{Z,n}^{\Sigma}) = X(\bZ_p)_{Z,n}^{\sigma(\Sigma)}
\end{equation}
for all $\sigma \in S_3 = S_{\{0,1,\infty\}}$ and $\Sigma \in \{0,1,\infty\}^S$. As a consequence, in order to compute $X(\bZ_p)_{Z,n}^{\min}$, it suffices to compute the subsets $X(\bZ_p)_{Z,n}^{\Sigma}$ with $\Sigma$ running through a set of representatives of the orbits for $S_3$ acting on $\{0,1,\infty\}^S$, and then take the $S_3$-orbit closure. 
One can also transform motivic Kim functions in $\cI_{Z,n}^{\mot,\Sigma}$ into motivic Kim functions in $\cI_{Z,n}^{\mot,\sigma(\Sigma)}$ for $\sigma \in S_3$ but writing down the $S_3$-action on $\cO(\Pi_n^{\dR})$ is not completely straightforward, so we omit this here.

\section{The motivic Galois action}
\label{sec:galois-action}

A key ingredient in our computations is an explicit description of the action of the unipotent motivic Galois group $U_Z$ (§\ref{sec:mixed-tate-galois-group}) on the de Rham fundamental group $\pi_1^{\dR}(X)$ of $X = \bP^1 \smallsetminus \{0,1,\infty\}$. We give a formula for this action based on the Kummer cocycle of the tangential basepoint $1 = (\vec{-1})_1$, and then explain how this can be made explicit using the calculus of motivic multiple zeta values. 

\subsection{The Galois action on the fundamental group}
\label{sec:galois-action-fundamental-group}

Let $Z = \Spec(\bZ) \smallsetminus S$ be a dense open subscheme of $\Spec(\bZ)$. Recall that we write $\pi_1^{\dR}(X)$ for the fundamental group based at the tangential basepoint $0 = \vec{1}_0$, and that this group is free pro-unipotent on two canonical generators $\gamma_0$ and $\gamma_1$ (§\ref{sec:structure-of-fundamental-group}). It carries an action of the mixed Tate motivic Galois group $G_Z = U_Z \rtimes \bG_m$ as it arises as the de Rham realization of the motivic fundamental group $\pi_1^{\mot}(X)$ in $\MT(Z,\bQ)$ (§\ref{sec:motivic-fundamental-group}). 
The action of $\bG_m$ is simply given through the grading by half-weight, with $\lambda \in \bG_m$ acting as $\lambda^{-i}$ in half-weight~$i$. 
Since $\gamma_0$ and $\gamma_1$ come from $\pi_1^{\dR}(\bG_m)$, whose Lie algebra $\rH_1^{\dR}(\bG_m)$ is pure of half-weight $-1$, the $\bG_m$-action on the generators is given by 
\[ \lambda.\gamma_i = \gamma_i^{\lambda} = \exp(\lambda \log(\gamma_i)). \]
We therefore focus on the action of the unipotent part $U_Z$. It turns out that it is completely determined by the action on one particular element: the canonical de Rham path $\gamma_{0,1}^{\dR} \in \pi_1^{\dR}(X;0,1)(\bQ)$ from the tangential basepoint at~$0$ to the one at~$1$ (§\ref{sec:canonical-paths}). This action is encoded in the $\bG_m$-equivariant Kummer cocycle
\begin{equation}
	\label{eq:kummer-cocycle}
	\begin{split}
		\kappa_1\colon U_Z &\to \pi_1^{\dR}(X), \\
		\kappa_1(u) &= (\gamma_{0,1}^{\dR})^{-1} u(\gamma_{0,1}^{\dR}),
	\end{split}
\end{equation}
which represents the isomorphism class of the path torsor $\pi_1^{\dR}(X;0,1)$ under the identification~\eqref{eq:H1-cocycles}.

\begin{thm}[The motivic Galois action]
	\label{galois-action-group-version}
	The action of $U_Z$ on $\pi_1^{\dR}(X)$ satisfies and is completely determined by the following:
	\begin{enumerate}[label=(\alph*)]
		\item \label{item:group-action-factors}
		the action factors through $U_Z \twoheadrightarrow U_{\bZ}$;
		\item \label{item:action-on-gamma0}
		$U_Z$ acts trivially on $\gamma_0$, i.e., $u(\gamma_0) = \gamma_0$ for all $u \in U_Z$;
		\item \label{item:action-on-gamma1}
		the action of $u \in U_Z$ on $\gamma_1$ is given by
		\[ u(\gamma_1) = \kappa_1(u)^{-1} \gamma_1 \kappa_1(u). \]
	\end{enumerate}
\end{thm}

Here, $U_{\bZ}$ is shorthand for $U_{\Spec(\bZ)}$. 

\begin{proof}
	The action factors through $U_Z \twoheadrightarrow U_{\bZ}$ since $\pi_1^{\mot}(X)$ is contained in the full subcategory $\MT(\Spec(\bZ),\bQ) \subseteq \MT(Z,\bQ)$, using that the tangential basepoint $\vec{1}_0$ is $\bZ$-integral. Recall that $\gamma_0 \in \pi_1^{\dR}(X)(\bQ)$ is defined as the image of the canonical generator of $\pi_1^{\dR}(T_0^{\times},1) = \pi_1^{\dR}(\bG_m, 1)$ under the local monodromy homomorphism~\eqref{eq:monodromy-homomorphism}. Since this map is motivic (i.e., arising from a map of motivic fundamental groups $\pi_1^{\mot}(T_0^{\times},1) \to \pi_1^{\mot}(X)$ in $\MT(Z,\bQ)$ under de Rham realization) \cite[§5.4]{DelGon05}, and $U_Z$ acts trivially on the 1-dimensional vector group $\pi_1^{\mot}(\bG_m,1)$, the action of $U_Z$ on $\gamma_0$ is trivial as well. Likewise, $U_Z$ acts trivially on $\gamma_1' \in \pi_1^{\dR}(X,1)(\bQ)$, the inverse of the image of the canonical generator of $\pi_1^{\dR}(T_1^{\times}, 1)$ under the local monodromy map at~$1$. By definition, $\gamma_1 \in \pi_1^{\dR}(X)(\bQ)$ is the loop obtained by transporting $\gamma_1'$ along the canonical de Rham path $\delta \coloneqq \gamma_{0,1}^{\dR}$, see Eq.~\eqref{eq:gamma1-def}, so for $u \in U_Z$ we have
	\begin{align*}
		u(\gamma_1) &= u(\delta)^{-1} \gamma_1' u(\delta) \\
		&= u(\delta)^{-1} \delta \gamma_1 \delta^{-1} u(\delta) \\
		&= (\delta^{-1} u(\delta))^{-1} \gamma_1 (\delta^{-1} u(\delta)) \\
		&= \kappa_1(u)^{-1} \gamma_1 \kappa_1(u). \qedhere
	\end{align*}
\end{proof}

\begin{rmk}
	\label{rmk: Galois action references}
	The statements of \Cref{galois-action-group-version} appear in a similar form in the proof of \cite[Proposition~5.9]{DelGon05}.
\end{rmk}

\subsection{Motivic Galois actions of Lie algebras}
\label{sec:galois-action-lie-algebras}

For explicit calculations it is convenient to work with the Lie algebras and the cocommutative Hopf algebras associated to the pro-unipotent groups $U_Z$ and $\pi_1^{\dR}(X)$. Denote the Lie algebras of $U_Z$ and $\pi_1^{\dR}(X)$ by $\fn(Z)$ and $\fn(X)$, respectively. The action of $U_Z$ on $\pi_1^{\dR}(X)$ induces a homomorphism of pro-algebraic groups
\[ U_Z \to \Aut(\pi_1^{\dR}(X)) \to \Aut(\fn(X)), \]
which induces a homomorphism of Lie algebras
\begin{equation}
	\label{eq:lie-algebra-galois-action}
	\phi\colon \fn(Z) \to \Der(\fn(X)),
\end{equation}
where $\Der(\fn(X))$ is the Lie algebra of derivations on $\fn(X)$. We write $\sigma.a \coloneqq \phi(\sigma)(a)$ and say that $\fn(Z)$ acts on $\fn(X)$ by derivations, i.e.,
\[ \sigma.[a,b] = [\sigma.a, b] + [a, \sigma.b] \]
for $\sigma \in \fn(Z)$ and $a,b \in \fn(X)$. To describe this action, we need the Lie algebra version of the Kummer cocycle~\eqref{eq:kummer-cocycle}. As explained in \cite[§5]{Luedtke:motivic_selmer_scheme}, for any cocycle $c\colon U_Z \to \pi_1^{\dR}(X)$ we can look at the induced map on tangent spaces at the identity to obtain a \emph{Lie algebra cocycle}
\[ c\colon \fn(Z) \to \fn(X), \]
denoted by the same letter. The cocycle property is
\begin{equation}
	\label{eq:lie-algebra-cocycle}
	c([\sigma,\tau]) = [c(\sigma), c(\tau)] + \sigma.c(\tau) - \tau.c(\sigma)
\end{equation}
for $\sigma, \tau \in \fn(Z)$.
In particular, we can apply this construction to the Kummer cocycle~\eqref{eq:kummer-cocycle} to obtain a Lie algebra cocycle
\begin{equation}
	\label{eq:lie-algebra-kummer-cocycle}
	\kappa_1\colon \fn(Z) \to \fn(X).
\end{equation}
As $\fn(X)$ is free pro-nilpotent on the two canonical generators $e_0 \coloneqq \log(\gamma_0)$ and $e_1 \coloneqq \log(\gamma_1)$, the Lie action~\eqref{eq:lie-algebra-galois-action} is completely determined by the action of $\fn(Z)$ on these two elements.

\begin{thm}[The motivic Galois action, Lie algebra version]
	\label{galois-action-lie-algebras}
	The action of $\fn(Z)$ on $\fn(X)$ satisfies and is completely determined by the following:
	\begin{enumerate}[label=(\alph*)]
		\item \label{item:action factors}
		the action factors through $\nf(Z) \twoheadrightarrow \nf(\mathbb{Z})$;
		\item \label{item:action on e0 trivial}
		$\nf(Z)$ acts trivially on $e_0$; that is, $\sigma . e_0 = 0$ for all $\sigma \in \nf(Z)$;
		\item \label{item:action on e1}
		the action of $\sigma \in \nf(Z)$ on $e_1$ is given by
		\[ \sigma . e_1 = [e_1, \kappa_1(\sigma)]. \]
	\end{enumerate}
\end{thm}

Here, $\fn(\bZ)$ is shorthand for $\fn(\Spec(\bZ))$. 

\begin{proof}
	Part~\ref{item:action factors} follows immediately from the corresponding statement in \Cref{galois-action-group-version}. Since $U_Z$ acts trivially on~$\gamma_0$, it also acts trivially on $e_0 = \log(\gamma_0)$, in other words, the composition
	\[ U_Z \to \Aut(\pi_1^{\dR}(X)) \to \Aut(\fn(X)) \xrightarrow{\ev_{e_0}} \fn(X) \]
	is constant with value $e_0$. Here, $\ev_{e_0}$ is the map ``evaluation at~$e_0$''. Taking the induced map on tangent spaces shows that
	\[ \fn(Z) \xrightarrow{\phi} \Der(\fn(X)) \xrightarrow{\ev_{e_0}} \fn(X) \]
	is constant zero, proving part~\ref{item:action on e0 trivial}.
	\Cref{galois-action-group-version}\ref{item:action-on-gamma1} implies that the action of $U_Z$ on $e_1 \in \fn(X)$ is given by
	\[ u(e_1) = \Ad(\kappa_1(u)^{-1})(e_1), \]
	where $\Ad$ denotes the adjoint action of $\pi_1^{\dR}(X)$ on its Lie algebra. In other words, the orbit map $U_Z \to \fn(X)$, $u \mapsto u(e_1)$ agrees with the composition
	\[ U_Z \xrightarrow{\kappa_1} \pi_1^{\dR}(X) \xrightarrow{(-)^{-1}} \pi_1^{\dR}(X) \xrightarrow{\Ad} \Aut(\fn(X)) \xrightarrow{\ev_{e_1}} \fn(X). \]
	Taking the induced map on tangent spaces shows that the orbit map $\fn(Z) \to \fn(X)$, $\sigma \mapsto \sigma.e_1$ equals
	\[ \fn(Z) \xrightarrow{\kappa_1} \fn(X) \xrightarrow{-1} \fn(X) \xrightarrow{\ad} \Der(\fn(X)) \xrightarrow{\ev_{e_1}} \fn(X). \]
	This means we have
	\[ \sigma.e_1 = \ad(-\kappa_1(\sigma))(e_1) = [-\kappa_1(\sigma), e_1] = [e_1, \kappa_1(\sigma)]. \qedhere \]
\end{proof}

\subsection{Hopf algebra actions and cocycles}
\label{sec:hopf-algebra-action}

Recall that to a pro-unipotent group~$U$ we can associate a cocommutative pro-Hopf algebra $\cH(U) = \cO(U)^{\vee}$ (§\ref{sec:hopf-algebra}). We denote the Hopf algebras associated to $U_Z$ and $\pi_1^{\dR}(X)$ by
\[ \cH(Z) \coloneqq \cO(U_Z)^{\vee}, \qquad \cH(X) \coloneqq \cO(\pi_1^{\dR}(X))^{\vee}. \]
There is a canonical isomorphism $\cH(X) = \bQ\llangle e_0,e_1 \rrangle$ corresponding to the canonical choice of generators of the Lie algebra~$\fn(X)$ (§\ref{sec:structure-of-fundamental-group}), and a non-canonical isomorphism $\cH(Z) = \bQ\llangle\Sigma \rrangle$ corresponding to a choice of generators $\Sigma$ of $\fn(Z)$ (§\ref{sec:mixed-tate-galois-group}). (In the latter case, we also allow infinite sums of bounded degree like $\sum_{\sigma \in \Sigma} \sigma$ as elements of $\bQ\llangle \Sigma\rrangle$, as the pro-object is constructed as a formal limit over finite subsets of~$\Sigma$.)

The action map $U_Z \times \pi_1^{\dR}(X) \to \pi_1^{\dR}(X)$ is given by an algebra homomorphism $\cO(\pi_1^{\dR}(X)) \to \cO(U_Z) \otimes \cO(\pi_1^{\dR}(X))$, which dualizes to a coalgebra homomorphism 
\[ \rho\colon \cH(Z) \, \hat\otimes \, \cH(X) \to \cH(X). \]
We write this action as $\sigma.a \coloneqq \rho(\sigma \otimes a)$ for $\sigma \in \cH(Z)$ and $a \in \cH(X)$. The group-action axioms give $1.a = a$ and $(\sigma \tau).a = \sigma.(\tau.a)$ for all $\sigma,\tau \in \cH(Z)$ and $a \in \cH(X)$, and the fact that $U_Z$ acts by group automorphisms is expressed by the formula
\begin{equation}
	\label{eq:hopf-action}
	\sigma.(ab) = \sum_{(\sigma)} (\sigma_{(1)}.a)(\sigma_{(2)}.b)
\end{equation}
for $\sigma \in \cH(U_Z)$ and $a,b \in \cH(X)$. Here, we are using Sweedler notation, so the sum on the right hand side of~\eqref{eq:hopf-action} means $\sum_i (\sigma_{(1)}^i.a)(\sigma_{(2)}^i.b)$ when the coproduct of $\sigma$ is written as $\Delta \sigma = \sum_i \sigma_{(1)}^i \otimes \sigma_{(2)}^i$.
Restricting the Hopf algebra action to grouplike or Lie-like elements recovers the group action and the Lie algebra action from the previous two subsections. 

A cocycle $c\colon U_Z \to \pi_1^{\dR}(X)$ corresponds to a coalgebra homomorphism
\[ c\colon \cH(Z) \to \cH(X), \]
which we still denote by the same letter, and the cocycle property translates into
\begin{equation}
	\label{eq:hopf-cocycle}
	c(\sigma \tau) = \sum_{(\sigma)} c(\sigma_{(1)}) \cdot \sigma_{(2)}.c(\tau),
\end{equation}
for $\sigma, \tau \in \cH(Z)$, again employing Sweedler notation. Most important for us is the case where $\sigma \in \fn(Z)$ is Lie-like, in which case the cocycle formula simplifies to
\begin{equation}
	\label{eq:cocycle-on-lie-element}
	c(\sigma \tau) = c(\sigma)c(\tau) + \sigma.c(\tau).
\end{equation}

We can apply this construction to extend the Kummer cocycle~\eqref{eq:kummer-cocycle} to a cocycle of cocommutative Hopf algebras
\begin{equation}
	\label{eq:kummer-cocycle-on-hopf-algebras}
	\kappa_1\colon \cH(Z) \to \bQ\llangle e_0,e_1\rrangle.
\end{equation}
It can be expressed in terms of motivic multiple zeta values (§\ref{sec:motivic-polylogarithms}) as follows. 

\begin{lemma}
	\label{lem:kummer-cocycle-via-mzv}
	For $\sigma \in \cH(Z)$ we have
	\begin{equation}
		\label{eq:kummer-cocycle-via-mzv}
		\kappa_1(\sigma) = \sum_w \langle \sigma, \zeta^{\fu}(w) \rangle w,
	\end{equation}
	where the sum is over all words $w$ in the letters $\{e_0,e_1\}$, and $\langle\cdot,\cdot\rangle$ is the natural pairing between $\cH(Z) = \cO(U_Z)^\vee$ and $A(Z) = \cO(U_Z)$.
\end{lemma}

\begin{proof}
	As $\kappa_1\colon \cH(Z) \to \bQ\llangle e_0,e_1\rrangle$ is defined as the linear dual of $\kappa_1^{\sharp}$, the $w$-coefficient of $\kappa_1(\sigma)$ is equal to
	\[ f_w(\kappa_1(\sigma)) = \langle \sigma, \kappa_1^{\sharp}(f_w) \rangle. \]	
	But the function $f_w \in \cO(\pi_1^{\dR}(X))$ is the multiple polylogarithm function $\Li_w$ (\Cref{paths-as-series}), and we have $\kappa_1^{\sharp} \Li_w = \zeta^{\fu}(w)$ by \Cref{motivic-integrals-via-kummer-cocycles} applied to the tangent vector $\alpha = (\vec{-1})_1$.
\end{proof}

Since the Kummer cocycle $\kappa_1$ is $\bG_m$-equivariant, its Hopf algebra extension~\eqref{eq:kummer-cocycle-on-hopf-algebras} respects the grading by half-weight, where $e_0$ and $e_1$ are placed in half-weight~$-1$. This means that whenever $\sigma \in \cH(Z)$ is pure of half-weight~$-n$, the sum on the right-hand side of~\eqref{eq:kummer-cocycle-via-mzv} involves only words $w$ of length~$n$.

\subsection{Example of a nontrivial Galois action}
\label{sec:galois-action-example}

Take $\sigma_3$ to be the generator of the one-dimensional space $\fn(\bZ)_{-3}$ satisfying $\langle \sigma_3, \zeta^{\fu}(3) \rangle = 1$ as in §\ref{sec:normalized-generators}. We determine the action of $\sigma_3$ on $e_1 \in \nf(X)$. Since the action is graded, $\sigma_3.e_1$ lives in half-weight~$-4$. A basis of $\nf(X)_{-4}$ is given by the three Lyndon words in $\{e_0, e_1\}$ of length~$4$ with their standard bracketing (§\ref{sec:lyndon-basis}), hence we have
\begin{align*}
	\sigma_3.e_1 &= \alpha \lambda(e_0e_0e_0e_1) + \beta \lambda(e_0e_0e_1e_1) + \gamma \lambda(e_0e_1e_1e_1)\\
	&= \alpha [e_0,[e_0,[e_0,e_1]]] + \beta [e_0,[[e_0,e_1],e_1]] + \gamma [[[e_0,e_1],e_1],e_1],
\end{align*}
for certain $\alpha, \beta, \gamma \in \bQ$. These coefficients of the Lyndon basis elements coincide with the coefficients of the Lyndon words after expanding the brackets. This is because for a Lyndon word $w$ one has $\lambda(w) = w + \text{lexicographically larger words}$ \cite[Lemma~10.2]{Luedtke:motivic_selmer_scheme}, and there is no mixing between the three Lyndon words of length~4 as they have different $e_1$-degrees. (This is not true in general, see \cite[Example 10.4]{Luedtke:motivic_selmer_scheme}.)
Using \Cref{galois-action-lie-algebras}, we can calculate
\begin{align*}
	\alpha &= f_{e_0e_0e_0e_1}(\sigma_3.e_1) \\
 &= f_{e_0e_0e_0e_1}([e_1, \ka_1(\sigma_3)]) \\
 &= -f_{e_0e_0e_0e_1}([\ka_1(\sigma_3), e_1]) \\
 &= -f_{e_0e_0e_0}(\ka_1(\sigma_3))\\
 &= 0.
\end{align*}
Here we are using that $\ka_1(\sigma_3)$ as an element of the free Lie algebra on $\{e_0, e_1\}$ has its $e_0e_0e_0$-coefficient equal to zero. Furthermore,
\begin{align*}
	\beta &= f_{e_0e_0e_1e_1}(\sigma_3.e_1) \\
 &= f_{e_0e_0e_1e_1}([e_1, \ka_1(\sigma_3)]) \\
 &= - f_{e_0e_0e_1e_1}([\ka_1(\sigma_3), e_1]) \\
 &= - f_{e_0e_0e_1}(\ka_1(\sigma_3)) \\
 &= - \langle \sigma_3, \zeta^{\fu}(3) \rangle \\
 &= -1.
\end{align*}
Here we used \Cref{lem:kummer-cocycle-via-mzv} and the choice of $\sigma_3$. Lastly,
\begin{align*}
	\gamma &= f_{e_0e_1e_1e_1}(\sigma_3.e_1) \\
 &= f_{e_0e_1e_1e_1}([e_1, \ka_1(\sigma_3)]) \\
 &= - f_{e_0e_1e_1e_1}([\ka_1(\sigma_3), e_1]) \\
 &= - f_{e_0e_1e_1}(\ka_1(\sigma_3)) \\
 &= - \langle \sigma_3, \zeta^{\fu}(2,1) \rangle \\
 &= - \langle \sigma_3, \zeta^{\fu}(3) \rangle \\
 &= -1,
\end{align*}
where we used \Cref{lem:kummer-cocycle-via-mzv}, the identity of motivic MZVs $\zeta^{\fu}(2,1) = \zeta^{\fu}(3)$ (\Cref{zeta3 relation}), and the normalization of~$\sigma_3$. In conclusion, we find
\[ \sigma_3.e_1 = -[e_0,[[e_0,e_1],e_1]] - [[[e_0,e_1],e_1],e_1]. \]

\section{The cocycle evaluation map in level~4}
\label{sec: cocycle evaluation map in level 4}

Let $\Pi_4 \coloneqq \pi_1(X)_4$ be the level~4 quotient of the motivic fundamental group of the thrice-punctured line, i.e., the quotient by the fifth step of the descending central series. By \Cref{lyndon-basis-of-coordinate-ring}, the unipotent group $\Pi_4^{\dR}$ has dimension~8 with coordinates given by
\[ \log, \; \Li_1, \Li_2,\; \Li_3,\; \Li_{2,1},\; \Li_4,\; \Li_{3,1},\; \Li_{2,1,1}, \]
corresponding to the Lyndon words of length at most four (\Cref{lyndon-words-of-length-le-4}). Let $S$ be a finite set of primes and let $\Sigma = \{ \tau_\ell : \ell \in S \} \cup \{\sigma_3, \sigma_5, \ldots \}$ be free generators of~$\nf(Z)$ chosen as in §\ref{sec:normalized-generators}.
The motivic Selmer scheme is the space of cocycles
\[ \rZ^1(\UZ, \Pi_4^{\dR})^{\bG_m}. \]

By \Cref{ex:coordinates-depth-4}, this is an affine space of dimension $2\#S + 2$ with coordinates
\[ x_{\ell} \; (\ell \in S), \quad y_{\ell} \; (\ell \in S), \quad z_3, \quad z_{2,1}, \]
defined by
\begin{equation}
	\label{eq: level 4 coordinates}
	\begin{split}
		x_{\ell}(c) &= L_{e_0}(c(\tau_{\ell})), \\
		z_3(c) &= L_{e_0 e_0 e_1}(c(\sigma_3)),
	\end{split}
	\qquad
	\begin{split}
		y_{\ell}(c) &= L_{e_1}(c(\tau_{\ell})),\\
		z_{2,1}(c) &= L_{e_0 e_1 e_1}(c(\sigma_3)),
	\end{split}
\end{equation}
where $\{L_w\}_w$ with $w$ Lyndon of length $\leq 4$ denotes the dual basis of the Lyndon basis $\{ \lambda(w)\}_{w}$ of $\Lie(\Pi_4^{\dR})$, and the extension of the group cocycle~$c$ to a graded cocycle of Hopf algebras $\cH(Z) \to \cH(\Pi_4^{\dR})$, as well as the associated graded Lie algebra cocycle $\fn(Z) \to \Lie(\Pi_4^{\dR})$ are denoted by the same letter.

Let $A(Z) = \cO(\UZ)$ and let 
\begin{equation}
	\label{eq:eta-univ-expansion}
	\eta_{\univ} = \sum_{u} a_u u \in A(Z)\llangle \Sigma \rrangle^{\gp} = \UZ(A(Z)) 
\end{equation}
be the universal point of $\UZ$. Thus, $a_u \in A(Z)$, for a word $u$ over~$\Sigma$, is the coefficient extraction functional $f_u$ (\Cref{def:coeff-extraction-functions}) mapping each $\gamma \in \UZ$ to its $u$-coefficient when expanded as a group-like noncommutative power series in the chosen generators~$\Sigma$. We write $a_u$ instead of $f_u$ to emphasize their role as coefficients in the universal cocycle evaluation map, and to distinguish them notationally from the functions $f_w$ on $\pi_1^{\dR}(X)$ for words $w$ on $\{e_0,e_1\}$.

Given a fundamental group quotient $\pi_1^{\mot}(X) \twoheadrightarrow \Pi$, the \emph{universal cocycle evaluation map} is the map
\[ \ev\colon \rZ^1(\UZ, \Pi^{\dR})^{\bG_m} \otimes_{\bQ} A(Z) \lto \Pi^{\dR} \otimes_{\bQ} A(Z), \quad c \mapsto c(\eta_{\univ}). \]
We compute this map in the case where $\Pi = \Pi_4$ is the full level-$4$ quotient of $\pi_1^{\mot}(X)$. Using the coordinates above, this is a map of affine $A(Z)$-schemes
\begin{align*}
\ev&\colon \Spec A(Z)[(x_{\ell})_{\ell \in S}, (y_{\ell})_{\ell \in S}, z_3, z_{2,1}] \\
&\quad \to \Spec A(Z)[\log,\Li_1,\Li_2,\Li_3,\Li_{2,1},\Li_4,\Li_{3,1},\Li_{2,1,1}].
\end{align*}

\begin{thm}
\label{univcoord}
	The universal cocycle evaluation map in level 4 is given in coordinates as follows:
	\begin{align*}
		\ev^\sharp \log &= \sum_{\ell \in S} a_{\tau_{\ell}} x_{\ell},\\
		\ev^\sharp \Li_1 &= \sum_{\ell \in S} a_{\tau_{\ell}} y_{\ell},\\
		\ev^\sharp \Li_2 &= \sum_{\ell,q \in S} a_{\tau_{\ell} \tau_q} x_{\ell} y_q,\\
		\ev^\sharp \Li_3 &= \sum_{\ell_1,\ell_2,q \in S} a_{\tau_{\ell_1} \tau_{\ell_2}\tau_q} x_{\ell_1} x_{\ell_2} y_q + a_{\sigma_3} z_3,\\
		\ev^\sharp \Li_{2,1} &= \sum_{\ell,q_1,q_2 \in S} a_{\tau_{\ell} \tau_{q_1}\tau_{q_2}} x_{\ell} y_{q_1} y_{q_2} + a_{\sigma_3} z_{2,1},\\
		\ev^\sharp \Li_{4} &= \sum_{\ell_1,\ell_2,\ell_3,q \in S} a_{\tau_{\ell_1} \tau_{\ell_2} \tau_{\ell_3}\tau_{q}} x_{\ell_1} x_{\ell_2} x_{\ell_3} y_q + \sum_{\ell \in S} a_{\tau_{\ell} \sigma_3} x_{\ell} z_3,\\
		\ev^\sharp \Li_{3,1} &= \sum_{\ell_1,\ell_2,q_1,q_2 \in S} a_{\tau_{\ell_1} \tau_{\ell_2} \tau_{q_1}\tau_{q_2}} x_{\ell_1} x_{\ell_2} y_{q_1} y_{q_2} + \sum_{\ell \in S} a_{\tau_{\ell} \sigma_3} x_{\ell} z_{2,1} + \sum_{\ell \in S} a_{\sigma_3 \tau_{\ell}} y_{\ell} (z_3 - 1),\\
		\ev^\sharp \Li_{2,1,1} &= \sum_{\ell, q_1,q_2, q_3 \in S} a_{\tau_{\ell} \tau_{q_1} \tau_{q_2}\tau_{q_3}} x_{\ell} y_{q_1} y_{q_2} y_{q_3} + \sum_{\ell \in S} a_{\sigma_3 \tau_{\ell}} y_{\ell} (z_{2,1} - 1).
	\end{align*}
\end{thm}

\begin{proof}
	Let $c \in \rZ^1(\UZ, \Pi_4^{\dR})^{\bG_m} \otimes_{\bQ} R$ be a $\bG_m$-equivariant cocycle defined over some $A(Z)$-algebra~$R$. Denote its coordinates~\eqref{eq: level 4 coordinates} by $x_{\ell} \coloneqq x_{\ell}(c)$, $y_{\ell} \coloneqq y_{\ell}(c)$, $z_{3} \coloneqq z_3(c)$, $z_{2,1} \coloneqq z_{2,1}(c) \in R$. We need to calculate 
	\[ (\ev^\sharp f_w)(c) = f_w(c(\eta_{\univ})) = \sum_u a_u f_w(c(u)) \]
	for all Lyndon words $w$ over $\{e_0, e_1\}$ of length $\leq 4$. As $c$ is graded, we can restrict the sum to those words~$u$ over~$\Sigma$ which have the same half-weight as $w$. In order to evaluate $f_w(c(u))$, the idea is to reduce to the case where $u$ is a single-letter word by using the cocycle property~\eqref{eq:cocycle-on-lie-element}, the Galois action (\Cref{galois-action-lie-algebras}), and the fact that 
	\[ f_{w_1 w_2}(X_1 X_2) = f_{w_1}(X_1)\cdot f_{w_2}(X_2) \]
	whenever $X_1, X_2 \in \bQ\llangle e_0,e_1\rrangle$ are homogeneous of the same half-weights as the words $w_1, w_2$. When $u$ is a single-letter word, i.e., one of the Lie algebra generators $\tau_{\ell}$ or $\sigma_n$, we have 
	\[ f_{e_0}(c(\tau_{\ell})) = x_{\ell}, \qquad f_{e_1}(c(\tau_{\ell})) = y_{\ell}, \]
	and, by writing $c(\sigma_n) = \sum_{w'} L_{w'}(c(\sigma_n)) \lambda(w')$ in the Lyndon basis:
	\[ f_w(c(\sigma_n)) = \sum_{\substack{w'\text{ Lyndon}\\ |w'|=n}} L_{w'}(c(\sigma_n)) f_w(\lambda(w')) =  \sum_{\substack{w'\text{ Lyndon}\\ |w'|=n}} f_w(\lambda(w')) z_{w'}, \]
	a linear combination of Selmer coordinates. For $n = 3$, the Lyndon words are $e_0 e_0 e_1$ and $e_0 e_1 e_1$, and we simply have
	\[ f_{e_0e_0e_1}(c(\sigma_3)) = z_3, \qquad f_{e_0e_1e_1}(c(\sigma_3)) = z_{2,1}. \]
	
	We begin computing the cocycle evaluation map in half-weight~$-1$, i.e., computing $\ev^{\sharp}\log$ and $\ev^{\sharp}\Li_1$, corresponding to the words $w = e_0$ and $w = e_1$, respectively. On the arithmetic side, the words~$u$ over~$\Sigma$ of half-weight~$-1$ are the single-letter words $\tau_{\ell}$ for $\ell \in S$, so we get
	\begin{align*}
		\log(c(\eta_{\univ})) &= \sum_{\ell \in S} a_{\tau_{\ell}} f_{e_0}(c(\tau_{\ell})) = \sum_{\ell \in S} a_{\tau_{\ell}} x_{\ell}
	\intertext{and}
		\Li_1(c(\eta_{\univ})) &= \sum_{\ell \in S} a_{\tau_{\ell}} f_{e_1}(c(\tau_{\ell})) = \sum_{\ell \in S} a_{\tau_{\ell}} y_{\ell}.
	\end{align*}
	
	In half-weight~$-2$ we have the arithmetic words $\tau_{\ell} \tau_{q}$ with $\ell, q \in S$. To evaluate the cocycle $c$ on the product $\tau_{\ell} \tau_q$ we use the cocycle condition~\eqref{eq:cocycle-on-lie-element} and the fact that the $\tau_{\ell}$ act trivially on~$\nf(X)$ by \Cref{galois-action-lie-algebras}\,\ref{item:action factors}:
	\begin{align*}
		c(\tau_{\ell} \tau_q) = c(\tau_{\ell}) c(\tau_q) + \tau_{\ell}.c(\tau_q) = c(\tau_{\ell}) c(\tau_q).
	\end{align*}
	Using this, we can calculate
	\begin{align*}
		\Li_2(c(\eta_{\univ})) &= \sum_{\ell,q \in S} a_{\tau_{\ell} \tau_q} f_{e_0 e_1} (c(\tau_{\ell} \tau_q)) = \sum_{\ell,q \in S} a_{\tau_{\ell} \tau_q} f_{e_0 e_1} (c(\tau_{\ell})c(\tau_q))\\
		&= \sum_{\ell,q \in S} a_{\tau_{\ell} \tau_q} f_{e_0}(c(\tau_{\ell})) f_{e_1}(c(\tau_q)) = \sum_{\ell,q \in S} a_{\tau_{\ell} \tau_q} x_{\ell} y_q.
	\end{align*}
	
	In half-weight~$-3$, we have the words $\tau_{\ell_1} \tau_{\ell_2} \tau_{\ell_3}$ with $\ell_1, \ell_2, \ell_3 \in S$ as well as the single-letter word $\sigma_3$. As above, the evaluation of~$c$ on a product of $\tau_{\ell}$'s is given by
	\[ c(\tau_{\ell_1} \tau_{\ell_2} \tau_{\ell_3}) = c(\tau_{\ell_1}) c(\tau_{\ell_2}) c(\tau_{\ell_3}). \]
	We get
	\begin{align*}
		\Li_3(c(\eta_{\univ})) &= \sum_{\ell_1, \ell_2, q \in S} a_{\tau_{\ell_1} \tau_{\ell_2} \tau_q} f_{e_0e_0e_1}(c(\tau_{\ell_1} \tau_{\ell_2} \tau_q)) + a_{\sigma_3} f_{e_0e_0e_1}(c(\sigma_3)) \\
		&= \sum_{\ell_1, \ell_2, q \in S} a_{\tau_{\ell_1} \tau_{\ell_2} \tau_q} f_{e_0}(c(\tau_{\ell_1})) f_{e_0}(c(\tau_{\ell_2})) f_{e_1}(c(\tau_q)) + a_{\sigma_3} z_3\\
		&= \sum_{\ell_1,\ell_2,q \in S} a_{\tau_{\ell_1} \tau_{\ell_2}\tau_q} x_{\ell_1} x_{\ell_2} y_q + a_{\sigma_3} z_3,
	\intertext{and similarly}
		\Li_{2,1}(c(\eta_{\univ})) &= \sum_{\ell, q_1, q_2 \in S} a_{\tau_{\ell} \tau_{q_1} \tau_{q_2}} f_{e_0e_1e_1}(c(\tau_{\ell} \tau_{q_1} \tau_{q_2})) + a_{\sigma_3} f_{e_0e_1e_1}(c(\sigma_3)) \\
		&= \sum_{\ell,q_1,q_2 \in S} a_{\tau_{\ell} \tau_{q_1}\tau_{q_2}} x_{\ell} y_{q_1} y_{q_2} + a_{\sigma_3} z_{2,1}.
	\end{align*}
	Now we consider words over $\Sigma$ of half-weight~$-4$. They are given by $\tau_{\ell_1} \tau_{\ell_2} \tau_{\ell_3} \tau_{\ell_4}$ with $\ell_i \in S$, as well as $\tau_{\ell} \sigma_3$ and $\sigma_3 \tau_{\ell}$ for $\ell \in S$. Evaluating the cocycle~$c$ on these yields
	\begin{align*}
		c(\tau_{\ell_1} \tau_{\ell_2} \tau_{\ell_3} \tau_{\ell_4}) &= c(\tau_{\ell_1}) c(\tau_{\ell_2}) c(\tau_{\ell_3}) c(\tau_{\ell_4}),\\
		c(\tau_{\ell} \sigma_3) &= c(\tau_{\ell}) c(\sigma_3),\\
		c(\sigma_3 \tau_{\ell}) &= c(\sigma_3) c(\tau_{\ell}) + \sigma_3 . c(\tau_{\ell}).
	\end{align*}
	The last equation is the first instance where the motivic Galois action enters the calculation nontrivially. In order to evaluate $\sigma_3 . c(\tau_{\ell})$, note that since $c$ is graded, $c(\tau_\ell)$ is an element of $\Lie(\Pi_4^{\dR})$ of half-weight~$-1$, so it is a linear combination of $e_0$ and $e_1$, namely
	\[ c(\tau_\ell) = x_{\ell} e_0 + y_{\ell} e_1 \]
	by definition of $x_{\ell}$ and $y_{\ell}$. By \Cref{galois-action-lie-algebras}\,\ref{item:action on e0 trivial}, the action of~$\sigma_3$ on~$e_0$ is trivial, i.e., $\sigma_3.e_0 = 0$. Thus, we have
	\begin{equation}
		\label{eq: sigma3 c tau_ell}
		\sigma_3 . c(\tau_{\ell}) = y_{\ell} \cdot (\sigma_3 . e_1).
	\end{equation}
	In §\ref{sec:galois-action-example}, we calculated the action of~$\sigma_3$ on~$e_1$. Specifically, we calculated the coefficients $f_w(\sigma_3.e_1)$ for all Lyndon words~$w$ of length~$4$:
	\begin{subequations}
		\begin{equation}
			\label{eq:s3 e1 coeff 0001}
			f_{e_0 e_0 e_0 e_1}(\sigma_3.e_1) = 0,
		\end{equation}
		\begin{equation}
			\label{eq:s3 e1 coeff 0011}
			f_{e_0 e_0 e_1 e_1}(\sigma_3.e_1) = -1,
		\end{equation}
		\begin{equation}
			\label{eq:s3 e1 coeff 0111}
			f_{e_0 e_1 e_1 e_1}(\sigma_3.e_1) = -1.
		\end{equation}
	\end{subequations}
	This allows us to compute $f_w(c(u))$ for all $w$ and $u$ of half-weight~$-4$, and hence the cocycle evaluation map in half-weight~$-4$:
	\begin{align*}
		\Li_4(c(\eta_{\univ})) &= \sum_{\ell_1,\ell_2,\ell_3,q \in S} a_{\tau_{\ell_1} \tau_{\ell_2} \tau_{\ell_3} \tau_q} f_{e_0e_0e_0e_1}(c(\tau_{\ell_1}) c(\tau_{\ell_2}) c(\tau_{\ell_3}) c(\tau_{q})) \\
		& \qquad + \sum_{\ell \in S} a_{\tau_{\ell} \sigma_3} f_{e_0e_0e_0e_1}(c(\tau_{\ell}) c(\sigma_3)) \\
		& \qquad + \sum_{\ell \in S} a_{\sigma_3 \tau_{\ell}} f_{e_0e_0e_0e_1}(c(\sigma_3) c(\tau_{\ell}) + \sigma_3 . c(\tau_{\ell})) 
		\\[2mm]
		&= \sum_{\ell_1,\ell_2,\ell_3,q \in S} a_{\tau_{\ell_1} \tau_{\ell_2} \tau_{\ell_3} \tau_q} x_{\ell_1} x_{\ell_2} x_{\ell_3} y_{q} \\
		& \qquad + \sum_{\ell \in S} a_{\tau_{\ell} \sigma_3} f_{e_0}(c(\tau_{\ell})) f_{e_0e_0e_1}(c(\sigma_3)) \\
		& \qquad + \sum_{\ell \in S} a_{\sigma_3 \tau_{\ell}} \Bigl( f_{e_0e_0e_0}(c(\sigma_3)) f_{e_1}(c(\tau_{\ell})) + y_{\ell} f_{e_0e_0e_0e_1}(\sigma_3 . e_1) \Bigr) & \text{(by~\eqref{eq: sigma3 c tau_ell})} 
		\\[2mm]
		&= \sum_{\ell_1,\ell_2,\ell_3,q \in S} a_{\tau_{\ell_1} \tau_{\ell_2} \tau_{\ell_3} \tau_q} x_{\ell_1} x_{\ell_2} x_{\ell_3} y_{q} + \sum_{\ell \in S} a_{\tau_{\ell} \sigma_3} x_{\ell} z_3.
	\end{align*}
	In the last equality the third sum disappeared because $f_{e_0 e_0 e_0}(c(\sigma_3)) = 0$ (as a consequence of $c(\sigma_3)$ being primitive) and $f_{e_0e_0e_0e_1}(\sigma_3 . e_1) = 0$ by Eq.~\eqref{eq:s3 e1 coeff 0001}.
	
	The calculation of $\Li_{3,1}(c(\eta_{\univ}))$ corresponding to the Lyndon word $e_0e_0e_1e_1$ is similar, so we focus on the differences. Firstly, from the words $\tau_{\ell} \sigma_3$ we get the terms
	\[ f_{e_0e_0e_1e_1}(c(\tau_{\ell}) c(\sigma_3)) = f_{e_0}(c(\tau_{\ell})) f_{e_0e_1e_1}(c(\sigma_3)) = x_{\ell} z_{2,1}. \]
	Secondly, from the words $\sigma_3 \tau_{\ell}$ we get the terms
	\begin{align*}
		f_{e_0e_0e_1e_1}(c(\sigma_3) c(\tau_{\ell})) &= f_{e_0e_0e_1}(c(\sigma_3)) f_{e_1}(c(\tau_{\ell})) = z_3 y_{\ell}
	\intertext{and}
		f_{e_0e_0e_1e_1}(\sigma_3 . c(\tau_{\ell})) &\overset{\eqref{eq: sigma3 c tau_ell}}{=} y_{\ell} \cdot f_{e_0e_0e_1e_1}(\sigma_3 . e_1) \overset{\eqref{eq:s3 e1 coeff 0011}}{=} - y_{\ell}.
	\end{align*}
	This implies the claimed formula for $\ev^\sharp \Li_{3,1}$. Finally, for $\Li_{2,1,1}$, which corresponds to $e_0e_1e_1e_1$, we use
	\[ f_{e_0e_1e_1e_1}(c(\tau_{\ell}) c(\sigma_3)) = f_{e_0}(c(\tau_{\ell})) f_{e_1e_1e_1}(c(\sigma_3)) = 0, \]
	where $f_{e_1 e_1 e_1}(c(\sigma_3)) = 0$ by primitivity; moreover,
	\begin{align*}
		f_{e_0e_1e_1e_1}(c(\sigma_3) c(\tau_{\ell})) &= f_{e_0e_1e_1}(c(\sigma_3)) f_{e_1}(c(\tau_{\ell})) = z_{2,1} y_{\ell}
	\intertext{and}
		f_{e_0e_1e_1e_1}(\sigma_3 . c(\tau_{\ell})) &\overset{\eqref{eq: sigma3 c tau_ell}}{=} y_{\ell} \cdot f_{e_0e_1e_1e_1}(\sigma_3 . e_1) \overset{\eqref{eq:s3 e1 coeff 0111}}{=} - y_{\ell}. \qedhere
	\end{align*}
\end{proof}

\section{Constructing Kim functions and verifying Kim's Conjecture}
\label{sec:kim-functions}

We now use the results of the previous section to construct motivic Kim functions (§\ref{sec:motivic-kim-functions}) in level~4, i.e., elements of $\cI_{Z,4}^{\mot}$, for various choices of $Z = \Spec \bZ[1/S]$, where $S$ is a finite set of primes. The Kim functions are explicit polynomials in 
\[ A(Z)[\log,\Li_1,\Li_2,\Li_3,\Li_{2,1},\Li_4, \Li_{3,1}, \Li_{2,1,1}] \]
which pull back to zero along the universal cocycle evaluation map. The coefficients will be elements of $\bQ[(a_u)_u] \subseteq A(Z)$, with $u$ ranging over words in the symbols $\tau_{\ell}$ ($\ell \in S$) and $\sigma_3$. This step is called the \emph{geometric step} in \cite{CDC:polylog1, PolGonII, GeometricStep} and is largely independent of the specific choice of~$S$; only the cardinality of~$S$ matters. In order to specialize the motivic Kim functions to $p$-adic Kim functions for some prime $p \not\in S$, one applies the $p$-adic period map $\per_p\colon A(Z) \to \bQ_p$ (§\ref{sec:period-point}). Determining the $p$-adic period values $\per_p(a_u) \in \bQ_p$ is known as the \emph{arithmetic step}, for which an algorithm is described in \cite[§3]{MTMUEII}. In this paper we do not carry out new instances of the arithmetic step but import the results from the existing papers \cite{CDC:polylog1, DCJ:M05}. This allows us to obtain new Kim functions in the cases $Z = \Spec \bZ[1/2], \Spec \bZ[1/3], \Spec \bZ[1/6]$. (We omit the case $Z = \Spec \bZ$ as we already know infinitely many polylogarithmic Chabauty--Kim functions, and they are enough to prove Kim's Conjecture for any choice of~$p$ \cite[§5.6]{BKL:chabauty_kim_sc}.) By computing the zero locus of our Kim functions, we obtain many new instances of Kim's Conjecture.

\subsection{The geometric step for $\#S = 1$}
\label{sec:S=1}

We start with the case $Z = \Spec(\bZ)\smallsetminus S$ where $S = \{\ell\}$ contains a single prime. Write $x \coloneqq x_{\ell}$, $y \coloneqq y_{\ell}$. Note that the choice of the Lie algebra generator $\sigma_3$ in the one-dimensional space $\fn(Z)_{-3}$ is uniquely determined by the condition $\langle \sigma_3, \zeta^{\fu}(3)\rangle = 1$.
We know the periods 
\[ a_{\tau_{\ell}^k} = \frac1{k!}\log^{\fu}(\ell)^k, \qquad a_{\sigma_3} = \zeta^{\fu}(3) \]
by \Cref{single-letter-periods} and Eq.~\eqref{eq:shuffle-identity-for-power}, so \Cref{univcoord} specializes as follows:
\begin{thm}
	\label{localization map 2}
	For $S = \{\ell\}$, the universal cocycle evaluation map in level~4,
	\[ \ev\colon \Spec A(Z)[x, y, z_3, z_{2,1}] \to \Spec A(Z)[\log,\Li_1,\Li_2,\Li_3,\Li_{2,1},\Li_4, \Li_{3,1}, \Li_{2,1,1}], \]
	is given by
	\begin{align*}
		\ev^\sharp \log &= \log^{\fu}(\ell) x,\\
		\ev^\sharp \Li_1 &= \log^{\fu}(\ell) y,\\
		\ev^\sharp \Li_2 &= \frac12 \log^{\fu}(\ell)^2 xy,\\
		\ev^\sharp \Li_3 &= \frac16 \log^{\fu}(\ell)^3 x^2 y + \zeta^{\fu}(3) z_3, \\
		\ev^\sharp \Li_{2,1} &= \frac16 \log^{\fu}(\ell)^3 x y^2 + \zeta^{\fu}(3) z_{2,1},\\
		\ev^\sharp \Li_{4} &= \frac1{24} \log^{\fu}(\ell)^4 x^3 y + a_{\tau_{\ell} \sigma_3} x z_3,\\
		\ev^\sharp \Li_{3,1} &= \frac1{24} \log^{\fu}(\ell)^4 x^2 y^2 + a_{\tau_{\ell} \sigma_3} x z_{2,1} + a_{\sigma_3 \tau_{\ell}} y (z_3 - 1),\\
		\ev^\sharp \Li_{2,1,1} &= \frac1{24} \log^{\fu}(\ell)^4 x y^3 + a_{\sigma_3 \tau_{\ell}} y (z_{2,1} - 1). 
	\end{align*}
\end{thm}

\begin{thm}
\label{4functions}
	Let $Z = \Spec \bZ[1/\ell]$. The following four functions belong to the motivic Chabauty--Kim ideal $\cI_{Z,4}^{\mot}$:
	\begin{small}
		\begin{gather*}
			\Li_2 - \tfrac12 \log \Li_1,\\
			\log^{\fu}(\ell) \zeta^{\fu}(3) \Li_4 - a_{\tau_{\ell} \sigma_3} \log \Li_3 - \tfrac1{24}\left(\log^{\fu}(\ell) \zeta^{\fu}(3) - 4 a_{\tau_{\ell} \sigma_3}\right) \log^3 \Li_1,\\
			\log^{\fu}(\ell) \zeta^{\fu}(3) \left( \Li_{3,1} + \tfrac18 \log^2 \Li_1^2\right) - a_{\tau_{\ell} \sigma_3} \log \Li_{2,1} - a_{\sigma_3 \tau_{\ell}} \Li_1 \Li_3 + \zeta^{\fu}(3) a_{\sigma_3 \tau_{\ell}} \Li_1,\\
			\log^{\fu}(\ell) \zeta^{\fu}(3) \Li_{2,1,1} - a_{\sigma_3 \tau_{\ell}} \Li_1 \Li_{2,1} - \tfrac1{24}\left(\log^{\fu}(\ell) \zeta^{\fu}(3) - 4 a_{\sigma_3\tau_{\ell}}\right) \log \Li_1^3 + \zeta^{\fu}(3) a_{\sigma_3 \tau_{\ell}} \Li_1.
		\end{gather*}
	\end{small}
\end{thm}

\begin{proof}
	One verifies that the four given functions pull back to zero along the universal cocycle evaluation map from \Cref{localization map 2}.
\end{proof}

\begin{rmk}
	\label{rmk:S3-symmetrisation}
	The first two functions in \Cref{4functions} were previously derived in \cite[Theorem~1.17]{MTMUE} (in the case $Z = \Spec \bZ[1/2]$) and \cite[Prop.~4.5]{CDC:polylog1} (in general), using the polylogarithmic quotient of the fundamental group. It was observed in \cite{CDC:polylog1} that the level-$n$ \emph{polylogarithmic} Chabauty--Kim locus $X(\bZ_p)_{Z,\PL,n}$, unlike the full level-$n$ locus $X(\bZ_p)_{Z,n}$, is in general not stable under the natural $S_3$-action generated by $z \mapsto 1/z$ and $z \mapsto 1-z$, and the authors proposed to consider an intermediate \emph{$S_3$-symmetrized} locus \cite[§5.2]{CDC:polylog1}
	\[ X(\bZ_p)_{Z,n} \subseteq X(\bZ_p)_{Z,\PL,n}^{S_3} \subseteq X(\bZ_p)_{Z,\PL,n}, \]
	defined as the largest $S_3$-stable subspace of $X(\bZ_p)_{Z,\PL,n}$. In terms of Chabauty--Kim ideals this corresponds to taking the closure of the ideal $\cI_{Z,\PL,n}$ defining $X(\bZ_p)_{Z,\PL,n}$ under the operations $f \mapsto f(1/z)$ and $f \mapsto f(1-z)$. 
	In fact, the ideal $\cI_{Z,\PL,n}$ is already stable under the first operation as a consequence of the fact that the automorphism $z \mapsto 1/z$ of the fundamental groupoid of~$X$ descends to the polylogarithmic quotient. This implies that, if $\cI_{Z,\PL,n}$ is generated by a set of Kim functions $\{f_i\}$ then the $S_3$-symmetrization $\cI_{Z,\PL,n}^{S_3}$ is generated by $f_i(z), f_i(1-z), f_i(1-1/z)$, with $1-z$ and $1-1/z$ representing the nontrivial right cosets of $\langle 1/z \rangle$ in $\Aut(X) = S_3$. 
	
	In the setting of \Cref{4functions}, denote by $F_2$, $F_4$, $F_{3,1}$, $F_{2,1,1}$ the $p$-adic specializations of the four Kim functions from the theorem. Assuming that $\zeta(3) \neq 0$ in $\bQ_p$ (as implied by \Cref{conj:period-conjecture}), we have
	\[ \cI_{Z,4} = \langle F_2, F_4, F_{3,1}, F_{2,1,1} \rangle \quad \text{and} \quad \cI_{Z,\PL,4} = \langle F_2, F_4 \rangle. \]
	Using transformation formulas for the iterated integrals $\Li_n(1-z)$ and $\Li_n(1-1/z)$ one can show the identities
	\begin{align*}
		F_{3,1}(z) &= F_4(z) - F_4(1-z) - F_4(1-1/z) - a_{\tau_{\ell}\sigma_3} \log(z) \Li_1(z) F_2(z),\\
		F_{2,1,1}(z) &= -F_4(1-z) + \frac12(a_{\tau_{\ell}\sigma_3} - a_{\sigma_3 \tau_{\ell}}) \Li_1(z)^2 F_2(z),
	\end{align*}
	which imply that $F_{3,1}, F_{2,1,1} \in \cI_{Z,\PL,4}^{S_3}$ and thus
	\[ X(\bZ_p)_{Z,4} = X(\bZ_p)_{Z,\PL,4}^{S_3}. \]
	Therefore, in this case, the polylogarithmic quotient combined with $S_3$-symmetri\-zation is strong enough to recover the full level-4 Chabauty--Kim locus. 
	However, this is not always the case. In §\ref{sec:S=2-geometric-step} below we treat the case $\#S = 2$ where the polylogarithmic level~4 quotient does not produce a single Kim function, whereas we obtain two using the full level~4 quotient.
\end{rmk}

\subsection{The case \texorpdfstring{$Z = \Spec \bZ[1/2]$}{Z = Spec ℤ[1/2]}}
\label{sec:Z-one-half}

We now specialize further to the case $\ell = 2$, i.e., $Z = \Spec \bZ[1/2]$.
In order to make the Coleman functions in \Cref{4functions} fully explicit, it remains to determine the values of the periods $a_{\tau_{\ell}\sigma_3}$ and $a_{\sigma_3\tau_{\ell}}$, a procedure which is called the \emph{arithmetic step} in \cite{CDC:polylog1}.

\begin{prop}
	\label{periods-Z-one-half}
	The motivic periods $a_{\tau_2 \sigma_3}$ and $a_{\sigma_3\tau_2}$ are given by
	\begin{equation}
		\label{eq:a-tau-sigma}
		a_{\tau_2 \sigma_3} = \frac87 \Li_4^\fu(2) \quad \text{and} \quad a_{\sigma_3\tau_2} = \frac47 \Li_{3,1}^\fu(-1).
	\end{equation}
\end{prop}

\begin{proof}
	Let $\kappa_2 \colon \UZ \to \Pi_4^{\dR}$ be the Kummer cocycle for the $\{2\}$-integral point $2 \in X(\bZ[1/2])$. Let $(x, y, z_3, z_{2,1})$ be its coordinates. By \Cref{xy-on-kummer}, we have
	\[ x = v_2(2) = 1, \qquad y = - v_2(1-2) = 0. \]
	From \Cref{localization map 2} (and \Cref{motivic-integrals-via-kummer-cocycles}) we get
	\begin{align}
		\label{eq:Li3-2}
		\Li_3^\fu(2) &= \frac16 \log^\fu(2)^3 x^2 y + \zeta^\fu(3) z_3 = \zeta^\fu(3) z_3,\\
		\label{eq:Li4-2}
		\Li_4^\fu(2) &= \frac1{24} \log^\fu(2)^4 x^3 y + a_{\tau_2 \sigma_3} x z_3 = a_{\tau_2 \sigma_3} z_3.
	\end{align}
	Equation~\eqref{eq:Li3-2} shows that the ratio of motivic periods $\Li_3^{\fu}(2)/\zeta^{\fu}(3) = z_3$ is a rational number. Its value can thus be determined after applying the period map for an arbitrary odd prime~$p$ for which $\zeta(3) \neq 0$ in $\bQ_p$. The non-vanishing of odd $p$-adic zeta values is known for every regular prime \cite[Remark~2.20(1)]{FurushoI}, so $p=3$ works. (One could also check numerically that $\zeta(3) \neq 0$ in $\bQ_3$; the zeta value is computed as \texttt{Qp(3)(1).polylog(3)} in SageMath \cite{sagemath}.) The Coleman--Sinnott functional equation \cite[Proposition~6.4(i)]{Coleman} relating $\Li_n(z)$ and $\Li_n(1/z)$ yields
	\[ \Li_3(2) - \Li_3(1/2) = -\frac16 \log(2)^3 \]
	in $\bQ_p$. Appendix~B of \cite{MTMUE} shows the identity of motivic periods
	\[ \Li_3^{\fu}(1/2) = \frac78 \zeta^{\fu}(3) + \frac16 \log^{\fu}(2)^3, \]
	which yields an identity in $\bQ_p$ upon applying the $p$-adic period map. Combining the last two equations yields
	\begin{equation}
		\label{eq:Li3-2-zeta}
		\Li_3(2) = \frac78 \zeta(3),
	\end{equation}
	hence $z_3 = 7/8$. 
	Finally, Eq.~\eqref{eq:Li4-2} yields
	\[ a_{\tau_2 \sigma_3} = \Li_4^\fu(2)/z_3 = \frac87 \Li_4^\fu(2). \]
	A similar calculation with the $\{2\}$-integral point $-1 \in X(\bZ[1/2])$ and the functions $\Li_{3}$ and $\Li_{3,1}$ yields
	\[ a_{\sigma_3 \tau_2} = \frac1{1- \Li_3^\fu(-1) / \zeta^\fu(3)} \Li_{3,1}^\fu(-1) = \frac47 \Li_{3,1}^\fu(-1). \qedhere \]
\end{proof}

\begin{rmk}
	The period $a_{\tau_2 \sigma_3}$ agrees with $f_{\sigma \tau}$ from~\cite[§4.3.1]{CDC:polylog1}. (The reversal of symbols in the subscript comes from differing conventions.) The authors find the expression
	\[ a_{\tau_2 \sigma_3} = -\frac87\left(\frac{\log^{\fu}(2)^4}{24} + \Li_4^{\fu}(1/2)\right) \]
	which is slightly more complicated than Eq.~\eqref{eq:a-tau-sigma}.
\end{rmk}

\begin{rmk}
	\label{rem: period sigma tau}
	The period $a_{\sigma_3 \tau_2}$ can also be obtained from $a_{\tau_2 \sigma_3}$ via the shuffle relation~\eqref{eq:shuffle-identity}
	\[ a_{\tau_2\sigma_3} + a_{\sigma_3\tau_2} = a_{\tau_2}a_{\sigma_3} = \log^{\fu}(2) \zeta^{\fu}(3), \]
	yielding the alternative expression
	\[ a_{\sigma_3\tau_2} = \zeta^\fu(3) \log^\fu(2) - \frac87 \Li_4^\fu(2). \]
	The value of $a_{\sigma_3 \tau_2}$ given in \cite[§5.23]{DCJ:M05} (denoted $f_{\tau\sigma}$ there) is incorrect.
\end{rmk}

Using computations in SageMath \cite{sagemath}, we verified new cases of Kim's Conjecture:

\begin{thm}
	\label{kims-conj-Z-one-half}
	Kim's Conjecture (\Cref{conj:kim}) for $Z = \Spec \bZ[1/2]$ holds in level~4, i.e., the inclusion
	\[ \left\{2,-1,\tfrac12\right\} = X(\bZ[1/2]) \subseteq X(\bZ_p)_{Z,4} \]
	is an equality, for all primes $3 \leq p < 1000$.
\end{thm}

\begin{proof}
	\Cref{4functions}, specialized to $Z = \Spec \bZ[1/2]$ with the period values from \Cref{periods-Z-one-half}, provides four explicit Kim functions which vanish on $X(\bZ_p)_{Z,4}$. In this case it is enough to compute the common zero locus of only the first two functions, which are both polylogarithmic, and verify that this locus contains no points besides $\{2,-1,\tfrac12\}$. The method for finding roots of polylogarithmic functions is described in \cite[§4.3--§4.4]{luedtke:refined}. We note that the reducedness of the Chabauty--Kim locus is implied by the fact that the $p$-adic root-finding procedure via Hensel's Lemma succeeds; a double root would result in a \texttt{PrecisionError}. The SageMath code for verifying Kim's Conjecture for $Z = \Spec \bZ[1/2]$ is contained in the file \texttt{kims\_conj\_Z\_one\_half.sage}, available at \url{https://github.com/martinluedtke/BeyondPolylog}.
\end{proof}

As noted in the proof, \Cref{kims-conj-Z-one-half} holds already for the \emph{polylogarithmic} level~4 locus $X(\bZ_p)_{Z,\PL,4}$; it is thus not substantially new. It had previously been computationally verified for primes in the range $3 \leq p \leq 29$ \cite{MTMUE}.

\subsection{The case \texorpdfstring{$Z = \Spec \bZ[1/3]$}{Z = Spec ℤ[1/3]}}
\label{sec:Z-one-third}

We can also specialize §\ref{sec:S=1} to the case $Z = \Spec \bZ[1/3]$, where the values of the motivic periods $a_{\tau_3 \sigma_3}$ and $a_{\sigma_3 \tau_3}$ in $A(Z)$ are known thanks to prior work by Corwin and Dan-Cohen.

\begin{prop}
	\label{periods-Z-one-third}
	The motivic periods $a_{\tau_3 \sigma_3}$ and $a_{\sigma_3\tau_3}$ are given by
	\begin{equation}
		\label{eq:a-tau3-sigma}
		a_{\tau_3 \sigma_3} = \frac{18}{13} \Li_4^{\fu}(3) - \frac{3}{52} \Li_4^{\fu}(9) \quad \text{and} \quad a_{\sigma_3\tau_3} = \log^{\fu}(3) \zeta^{\fu}(3) - a_{\tau_3 \sigma_3}.
	\end{equation}
\end{prop}

\begin{proof}
	The period $a_{\tau_3 \sigma_3}$ is computed in \cite[§4.3.3]{CDC:polylog1}; it is denoted $f_{\sigma \tau}$ there. This determines $a_{\sigma_3 \tau_3}$ via the shuffle relation \eqref{eq:shuffle-identity}.
\end{proof}

\begin{thm}
	\label{kims-conj-Z-one-third}
	Kim's Conjecture (\Cref{conj:kim}) for $Z = \Spec \bZ[1/3]$ holds in level~4, i.e., the inclusion
	\[ \emptyset = X(\bZ[1/3]) \subseteq X(\bZ_p)_{Z,4} \]
	is an equality, for all primes $p < 1000$ with $p \neq 3$.
\end{thm}

\begin{proof}
	The case $p = 2$ is trivial since $X(\bZ_2) = \emptyset$, so assume $p \neq 2$. \Cref{4functions}, specialized to $Z = \Spec \bZ[1/3]$ with the period values from \Cref{periods-Z-one-third}, provides four explicit Kim functions which vanish on $X(\bZ_p)_{Z,4}$. In this case it is actually enough to compute the common zero locus in $X(\bZ_p)$ of the first two functions, which are both polylogarithmic. This locus contains $-1$ by \cite[Lemma~5.1]{CDC:polylog1} and it suffices to verify that it contains no other points, for then the $S_3$-symmetrized polylogarithmic Chabauty--Kim locus $X(\bZ_p)_{Z,\PL,4}^{S_3} \supseteq X(\bZ_p)_{Z,4}$ is empty. The SageMath code for verifying Kim's Conjecture for $Z = \Spec \bZ[1/3]$ is contained in the file \texttt{kims\_conj\_Z\_one\_third.sage}, available at \url{https://github.com/martinluedtke/BeyondPolylog}.
\end{proof}

As noted in the proof, \Cref{kims-conj-Z-one-third} already holds for the $S_3$-symmetrized polylogarithmic Chabauty--Kim locus $X(\bZ_p)_{Z,\PL,4}^{S_3}$. For the two primes $p = 5, 7$, the fact that this locus is empty was previously observed in \cite[Theorem~5.5]{CDC:polylog1}.

\subsection{Refined Chabauty--Kim for \texorpdfstring{$Z = \Spec \bZ[1/2]$}{Z = Spec ℤ[1/2]}}
\label{sec:Z-one-half-refined}

We also consider the refined variant of the Chabauty--Kim method (§\ref{sec:refined-chabauty-kim}) in the case $\#S = 1$. Since the refined Selmer scheme is automatically empty whenever $2 \not\in S$, the only interesting case is $S = \{2\}$, so we take $Z = \Spec \bZ[1/2]$. By working with refined Selmer schemes, the formulas for the localization map and the resulting equations for the Chabauty--Kim locus simplify. 
Write $x \coloneqq x_2$ and $y \coloneqq y_2$ as before. The refined motivic Selmer scheme in level~$4$
\[
\Sel_{Z,4}^{\min}(X)
\subseteq
\rH^1(G_Z, \Pi_4^{\dR}) = \Spec \bQ[x,y,z_3,z_{2,1}]
\]
is by definition the union of the three hyperplanes
\[ \Sel_{Z,4}^{\min}(X) = \Sel_{Z,4}^{(0)}(X) \cup \Sel_{Z,4}^{(1)}(X) \cup \Sel_{Z,4}^{(\infty)}(X) \]
defined by
\begin{itemize}
	\item $\Sel_{Z,4}^{(0)}(X): \; y = 0$;
	\item $\Sel_{Z,4}^{(1)}(X): \; x = 0$;
	\item $\Sel_{Z,4}^{(\infty)}(X): \; x + y = 0$.
\end{itemize}
For each of the three refinement conditions~$\Sigma \in \{0,1,\infty\}$, denote by $i_{\Sigma}$ the inclusion of the $\Sigma$-refined Selmer scheme $\Sel_{Z,4}^{\Sigma}(X)$ into the unrefined motivic Selmer scheme $\rH^1(G_Z, \Pi_4^{\dR})$. By Eq.~\eqref{eq:S3-action-on-refined-loci} it suffices to compute any single one of the three refined Chabauty--Kim loci $X(\bZ_p)_{Z,4}^{\Sigma}$; the full refined locus $X(\bZ_p)_{Z,4}^{\min}$ is then obtained by taking the $S_3$-orbit closure. We choose $\Sigma = (1)$.

\begin{prop}
	\label{refined localisation map 2}
	For $Z = \Spec \bZ[1/2]$, the restriction of the universal cocycle evaluation map in level~4 to the refined Selmer scheme component $\Sel_{Z,4}^{(1)}(X)$ is given by
	\begin{align*}
		(\ev \circ i_{(1)})^\sharp \log &= 0,\\
		(\ev \circ i_{(1)})^\sharp \Li_1 &= \log^\fu(2) y,\\
		(\ev \circ i_{(1)})^\sharp \Li_2 &= 0,\\
		(\ev \circ i_{(1)})^\sharp \Li_3 &= \zeta^{\fu}(3) z_3, \\
		(\ev \circ i_{(1)})^\sharp \Li_{2,1} &= \zeta^{\fu}(3) z_{2,1},\\
		(\ev \circ i_{(1)})^\sharp \Li_{4} &= 0,\\
		(\ev \circ i_{(1)})^\sharp \Li_{3,1} &= a_{\sigma_3 \tau_2} y (z_3 - 1),\\
		(\ev \circ i_{(1)})^\sharp \Li_{2,1,1} &= a_{\sigma_3 \tau_2} y (z_{2,1} - 1).
	\end{align*}
\end{prop}

\begin{proof}
	Restricting the universal cocycle evaluation map to $\Sel_{Z,4}^{(1)}(X)$ amounts to setting $x = 0$ in \Cref{localization map 2}.
\end{proof}

\begin{thm}
	\label{thm:kim_functions_1}
	Let $Z = \Spec \bZ[1/2]$. The following five functions belong to the motivic refined Chabauty--Kim ideal $\cI_{Z,4}^{\mot,(1)}$ for the refinement condition $\Sigma = (1)$:
	\begin{gather*}
		\log, \qquad \Li_2, \qquad \Li_4,\\
		\zeta^{\fu}(3) \log^{\fu}(2) \Li_{3,1} - a_{\sigma_3 \tau_2} \Li_1(\Li_3 - \zeta^{\fu}(3)),\\
		\zeta^{\fu}(3)\log^{\fu}(2)\Li_{2,1,1} - a_{\sigma_3 \tau_2} \Li_1(\Li_{2,1} - \zeta^{\fu}(3)).
	\end{gather*}
\end{thm}

The value of the motivic period $a_{\sigma_3 \tau_2}$ is equal to the two equivalent expressions
\[ a_{\sigma_3 \tau_2} = \frac47 \Li_{3,1}^{\fu}(-1) = \zeta^{\fu}(3) \log^{\fu}(2) - \frac87 \Li_4^{\fu}(2) \]
by \Cref{periods-Z-one-half} and \Cref{rem: period sigma tau}.

\begin{proof}
	Using the formulas from \Cref{refined localisation map 2}, one verifies that the given functions pull back to zero under $\ev \circ i_{(1)}$.
\end{proof}

\begin{rem}
	\label{rem:expected_locus_1}
	According to \Cref{conj:Sigma-refined-Kims-conjecture}, the locus $X(\bZ_p)_{Z,n}^{(1)}$ for sufficiently large~$n$ should consist of those $\bZ[1/2]$-integral points whose mod-2 reduction lies in $X \cup \{1\}$. The only such point is $z = -1$.
	In fact, the first two equations $\log(z) = 0$ and $\Li_2(z) = 0$ from \Cref{thm:kim_functions_1}, which define the level~2 Chabauty--Kim locus $X(\bZ_p)_{Z,2}^{(1)}$, were previously found in \cite[Theorem~A]{BBKLMQSX}, and the authors verified numerically for all odd primes $p  < 10^5$ that $z=-1$ is indeed the only common solution. Working in infinite level but restricting to the polylogarithmic quotient, work by Alex Betts, Theresa Kumpitsch and one of the authors (M.L.) shows that the equations
	\[ \log(z) = 0 \qquad \text{and} \qquad \Li_k(z) = 0 \quad \text{for $k \geq 2$ even} \]
	hold on $X(\bZ_p)_{Z,\infty}^{(1)}$, and that for any odd prime~$p$, the element $z=-1$ is the only common solution, thus proving the refined Kim's Conjecture for $S = \{2\}$ \cite[Theorem~B]{BKL:chabauty_kim_sc}. The reducedness of the Chabauty--Kim locus was not explicitly stated in loc.\ cit.\ but this follows from the fact that $\log(z)$ has no double roots. 
\end{rem}

\subsection{The geometric step for $\#S = 2$}
\label{sec:S=2-geometric-step}

We now consider the case where $Z = \Spec(\bZ) \smallsetminus S$ with $\#S = 2$. Let $S = \{\ell_1, \ell_2\}$ and abbreviate the associated Selmer scheme coordinates as $x_i \coloneqq x_{\ell_i}$, $y_i \coloneqq y_{\ell_i}$ for $i=1,2$. The universal cocycle evaluation map in level~4 (§\ref{sec: cocycle evaluation map in level 4}) is thus a map $\bA_{A(Z)}^6 \to \bA_{A(Z)}^8$, so there exist at least two independent Kim functions. We manage to compute two such Kim functions explicitly.
If we used only the polylogarithmic quotient, the universal cocycle evaluation map in level~4 would be a map $\bA_{A(Z)}^5 \to \bA_{A(Z)}^5$,
\[ \ev_{\PL}\colon \Spec A(Z)[x_1,x_2,y_1,y_2,z_3] \to \Spec A(Z)[\log,\Li_1,\Li_2,\Li_3,\Li_4], \]
so we would not expect to find a polylogarithmic Kim function. This expectation is confirmed by the results of \cite{GeometricStep} where the authors computed an explicit polylogarithmic Kim function in level~6 and showed that no such function exists in smaller level. 

\begin{thm}
	\label{kim-function-S=2}
	Let $Z = \Spec(\bZ) \smallsetminus S$ with $\#S = 2$. The construction described below produces two elements $F_1^{\mot}$, $F_2^{\mot}$ of the motivic Chabauty--Kim ideal $\cI_{Z,4}^{\mot}$ in level~$4$ whose common vanishing locus in $\Pi^{\dR}_4 \otimes K(Z)$ is pure of codimension two. 
\end{thm}

The formulas for $F_1^{\mot}$, $F_2^{\mot}$ are too long to be reproduced here: $F_1^{\mot}$ has half-weight $10$ and $53$ terms, while $F_2^{\mot}$ has half-weight $12$ and $82$ terms, and each coefficient in $A(Z)$ is a polynomial over $\bQ$ in the Lyndon shuffle algebra generators $a_1, a_2, a_{12}, a_{\sigma}, a_{112}, a_{122}, a_{1112}, a_{1122}, a_{1222}, a_{1\sigma}, a_{2\sigma}$. Here we simplify the notation by writing $1$ and $2$ instead of $\tau_{\ell_1}$ and $\tau_{\ell_2}$, and $\sigma$ instead of $\sigma_3$. Complete formulas for $F_1^{\mot}$ and $F_2^{\mot}$ (\texttt{abstract\_kim\_functions.txt}) along with code in SageMath \cite{sagemath} to produce them (\texttt{abstract\_kim\_functions.sage}), are available at \url{https://github.com/martinluedtke/BeyondPolylog}.

We now describe the construction, using elimination theory on the universal cocycle evaluation map from §\ref{sec: cocycle evaluation map in level 4}. 
Let 
\[ R \coloneqq K(Z)[x_1,y_1,x_2,y_2,z_3,z_{2,1},\, \log,\Li_1,\Li_2,\Li_3,\Li_{2,1},\Li_4,\Li_{3,1}, \Li_{2,1,1}] \] 
and let $I \subseteq R$ be the graph ideal for the cocycle evaluation map, i.e., $I$ is generated by the eight polynomials
\begin{align*}
	\log &- \sum_{i \in \{1,2\}} a_i x_i, \\
	\Li_1 &- \sum_{i \in \{1,2\}} a_i y_i, \\
	&\ldots\\
	\Li_{2,1,1} &- \sum_{i,j_1,j_2,j_3 \in \{1,2\}} a_{i, j_1, j_2, j_3} x_i y_{j_1} y_{j_2} y_{j_3} - \sum_{i \in \{1,2\}} a_{\sigma, i} y_{i} (z_{2,1} - 1)
\end{align*}
coming from \Cref{univcoord}. Thus, $R/I$ is the coordinate ring of the graph of the cocycle evaluation map, and our aim is to construct elements of
\[ I \cap K(Z)[\log,\Li_1,\Li_2,\Li_3,\Li_{2,1},\Li_4,\Li_{3,1}, \Li_{2,1,1}] \]
by eliminating the Selmer variables $x_i,y_i,z_3,z_{2,1}$. 
We remark that both the definition of the Selmer coordinates $z_3$ and $z_{2,1}$ and the definition of the motivic periods $a_u$ (see Eq.~\eqref{eq:eta-univ-expansion}) depend on the choice of the Lie algebra generator $\sigma_3 \in \fn(Z)_{-3}$. As usual, we assume that it is chosen to satisfy $\langle \sigma_3, \zeta^{\fu}(3) \rangle = 1$. Unlike for $\#S = 1$, this does not determine it uniquely.

First, the equations for $\Li_3$ and $\Li_{2,1}$ can be used to eliminate $z_3$ and $z_{2,1}$:
\begin{align}
	\label{eq:elim-z3}
	z_3 &\equiv \bigl(\Li_3 - \sum_{i_1,i_2,j \in \{1,2\}} a_{i_1,i_2,j} x_{i_1} x_{i_2} y_j \bigr) / a_{\sigma} \bmod I,\\
	\label{eq:elim-z21}
	z_{2,1} &\equiv \bigl(\Li_{2,1} - \sum_{i,j_1,j_2 \in \{1,2\}} a_{i,j_1,j_2} x_{i} y_{j_1} y_{j_2} \bigr) / a_{\sigma} \bmod I.
\end{align}
Next, the equations for $\log$ and $\Li_1$ can be used to eliminate $x_2$ and $y_2$:
\begin{align}
	\label{eq:elim-x2}
	x_2 &\equiv (\log - a_1 x_1)/a_2 \bmod I,\\
	\label{eq:elim-y2}
	y_2 &\equiv (\Li_1 - a_1 y_1)/a_2 \bmod I.
\end{align}
In order to eliminate $y_1$, we write the equations for $\Li_1$ and $\Li_2$ in matrix form as
\[ \begin{pmatrix}
	a_1 & a_2 \\
	a_{11} x_1 + a_{21} x_2 & a_{12} x_1 + a_{22} x_2
\end{pmatrix} \begin{pmatrix}
y_1 \\ y_2
\end{pmatrix} = \begin{pmatrix}
\Li_1 \\ \Li_2
\end{pmatrix} \]
in $R/I$. The determinant $\Delta$ of the matrix equals
\[ \Delta =  \frac12 (a_{12} - a_{21}) (a_1 x_1 + a_2 x_2) \equiv \frac12 (a_{12} - a_{21}) \log\; \bmod I \]
by \cite[Lemma~6.1]{GeometricStep}; so multiplying by the adjugate matrix yields
\[ \Delta y_1 \equiv (a_{12} x_1 + a_{22} x_2) \Li_1 - a_2 \Li_2 \, \bmod I. \]
Combining this with \eqref{eq:elim-x2} and using the shuffle relations $a_{22} = \frac12 a_2^2$ and $a_1 a_2 = a_{12} + a_{21}$ one finds
\begin{equation}
	\label{eq:elim-y1}
	\log \cdot y_1 \equiv \Li_1 x_1 - \frac{2 a_2}{a_{12} - a_{21}}\bigl(\Li_2 - \frac12 \log \cdot \Li_1\bigr) \; \bmod I.
\end{equation}

Now let $\tilde P_4, \tilde P_{3,1}, \tilde P_{2,1,1} \in A(Z)[x_i,y_i,z_3,z_{2,1}]$ be the three polynomials from the right hand side in \Cref{univcoord} such that
\[ P_4 \coloneqq \Li_4 - \tilde P_4, \quad P_{3,1} \coloneqq \Li_{3,1} - \tilde P_{3,1}, \quad P_{2,1,1} \coloneqq \Li_{2,1,1} - \tilde P_{2,1,1} \]
are contained in $I$. From each $P \in \{P_4,P_{3,1},P_{2,1,1}\}$ we construct a polynomial $\nu_P \in I \cap K(Z)[x_1,\log,\ldots,\Li_{2,1,1}]$ as follows: replace $z_3$ and $z_{2,1}$ with the right hand sides of \eqref{eq:elim-z3} and \eqref{eq:elim-z21}, respectively. Then replace $x_2$ and $y_2$ with the right hand sides of \eqref{eq:elim-x2} and \eqref{eq:elim-y2}, respectively. Multiply by a power of $\log$ so that in each monomial, the exponent of $\log$ is at least the exponent of $y_1$. Then use \eqref{eq:elim-y1} to eliminate $y_1$. 

We have now constructed three polynomials $\nu_4, \nu_{3,1}, \nu_{2,1,1} \in I$ in which $x_1$ is the only remaining Selmer variable. The computation in SageMath shows that these polynomials are quadratic in $x_1$, so we can write
\begin{empheq}[left=\empheqlbrace]{equation}
	\label{eq:nus}
	\begin{aligned}
		\nu_4 &= p_0 + p_1 x_1 + p_2 x_1^2, \\
		\nu_{3,1} &= q_0 + q_1 x_1 + q_2 x_1^2, \\
		\nu_{2,1,1} &= r_0 + r_1 x_1 + r_2 x_1^2
	\end{aligned}
\end{empheq}
with $p_i,q_i,r_i \in K(Z)[\log,\ldots,\Li_{2,1,1}]$. We then define
\begin{align}
	\label{eq:F1mot}
	\tilde F_1^{\mot} &\coloneqq \det \begin{pmatrix}
		p_0 & p_1 & p_2 \\
		q_0 & q_1 & q_2 \\
		r_0 & r_1 & r_2
	\end{pmatrix},\\
	\label{eq:F2mot}
	\tilde F_2^{\mot} &\coloneqq \Res_{x_1}(\nu_4,\nu_{3,1}),
\end{align}
where $\Res_{x_1}$ denotes the resultant with respect to the variable $x_1$. Both $\tilde F_1^{\mot}$ and $\tilde F_2^{\mot}$ are elements of $K(Z)[\log,\ldots,\Li_{2,1,1}]$. They are also contained in the graph ideal. Indeed, writing the equations~\eqref{eq:nus} in matrix form and multiplying by the adjugate matrix yields
\[  \opnm{adj}\begin{pmatrix}
	p_0 & p_1 & p_2 \\
	q_0 & q_1 & q_2 \\
	r_0 & r_1 & r_2
\end{pmatrix} \begin{pmatrix}
\nu_4\\ \nu_{3,1}\\ \nu_{2,1,1}
\end{pmatrix} = \tilde F_1^{\mot} \begin{pmatrix}
1\\ x_1 \\ x_1^2
\end{pmatrix};\]
in particular, by the first row, $\tilde F_1^{\mot}$ is a linear combination of $\nu_4,\nu_{3,1},\nu_{2,1,1} \in I$, hence $\tilde F_1^{\mot} \in I$. To see that $\tilde F_2^{\mot}$ is contained in~$I$, it suffices to note that $\nu_4$ and $\nu_{3,1}$ have $x_1 \bmod I$ as a common root in the domain $R/I$. 

We now have constructed two motivic Chabauty--Kim functions $\tilde F_1^{\mot}, \tilde F_2^{\mot} \in \cI_{Z,4}^{\mot}$. We multiply each $\tilde F_i^{\mot}$ by an element of $A(Z)$ to clear denominators; denote the resulting polynomials in $I \cap A(Z)[\log,\ldots, \Li_{2,1,1}]$ still by $\tilde F_i^{\mot}$. This step is not really necessary but it is more convenient for $p$-adic specializations (and computationally faster) for the motivic Kim functions to have $A(Z)$-coefficients. As was also the case in \cite{GeometricStep}, we need one additional step. It turns out that both $\tilde F_1^{\mot}$ and $\tilde F_2^{\mot}$ are divisible by $G \coloneqq \Li_2 - \frac12 \log \cdot \Li_1$, so we can write
\[ \tilde F_i^{\mot} = G F_i^{\mot} \quad (i=1,2) \]
with $F_i^{\mot} \in A(Z)[\log,\ldots,\Li_{2,1,1}]$. The function $G$ does not pull back to zero under the universal cocycle evaluation map:
\begin{align*}
	\ev^{\sharp}(G) &= (a_{11} x_1y_1 + a_{12}x_1y_2 + a_{21}x_2y_1 + a_{22}x_2y_2) - \frac12(a_1x_1 + a_2x_2)(a_1y_1 + a_2y_2)\\
	&= \tfrac12 (a_{12} - a_{21})(x_1 y_2 - x_2 y_1) \neq 0.
\end{align*}
Hence, the functions $F_1^{\mot}$ and $F_2^{\mot}$ obtained by cancelling the factor $G$ also belong to $\cI_{Z,4}^{\mot}$. The polynomials have no other common factor, which implies the claim about the codimension of their common vanishing locus. This finishes the construction and the proof of \Cref{kim-function-S=2}.

\begin{thm}
	\label{kim-function-weight}
	Let $Z = \Spec(\bZ) \smallsetminus S$ with $\#S = 2$. The motivic Kim function $F_1^{\mot}$ from \Cref{kim-function-S=2} is the unique element in $\cI_{Z,4}^{\mot}$ of half-weight $\leq 10$ up to scaling.
\end{thm}

\begin{proof}
	We adapt the method from \cite{GeometricStep} to the full level 4 quotient of the fundamental group. For $\#S = 2$, pullback along the universal cocycle evaluation map, restricted to functions in half-weight $\leq 10$, is a $K(Z)$-linear map
	\[ \ev^{\sharp}_{\hwt\leq 10}\colon K(Z) \otimes \cO(\Pi_4^{\dR})_{\leq 10} \to K(Z) \otimes \cO(\rH^1(G_Z,\Pi_4^{\dR}))_{\leq 10} \]
	of finite-dimensional $K(Z)$-vector spaces $K(Z)^{658} \to K(Z)^{1891}$. Its kernel dimension cannot efficiently be computed directly, but one gets an upper bound by specializing the underlying $A(Z)$-linear map $A(Z)^{658} \to A(Z)^{1891}$ along any ring homomorphism $A(Z) \to K$ for some field~$K$, say $K = \bQ$, and computing the kernel dimension over~$K$. This yields the upper bound 
	\[ \dim_{K(Z)} \left(\cI_{Z,4}^{\mot} \cap K(Z) \otimes \cO(\Pi_4^{\dR})_{\leq 10}\right) \leq 1, \]
	which matches the lower bound provided by the construction of $F_1^{\mot}$. The Sage code for this computation is contained in the file \texttt{weight\_bound.sage} at \url{https://github.com/martinluedtke/BeyondPolylog}.
\end{proof}

\subsection{Refined Chabauty--Kim for sets~$S$ of size 2}
\label{sec:S=2-refined}

Before specializing the results of the previous subsection to the case $Z = \Spec \bZ[1/6]$, we also carry out the geometric step for the \emph{refined} Chabauty--Kim method (§\ref{sec:refined-chabauty-kim}) in level~4 for sets $S$ of size~$2$.
This is trivial unless $2 \in S$, so we assume $S=\{2,q\}$ for an odd prime~$q$ and set $Z = \Spec \bZ[1/2q]$. The refined Chabauty--Kim locus $X(\bZ_p)_{Z,4}^{\min}$ for a prime $p \not\in S$ is then equal to the union of the 9 sets
\[  X(\bZ_p)_{Z,4}^{\min} = \bigcup_{\Sigma \in \{0,1,\infty\}^2} X(\bZ_p)_{Z,4}^{\Sigma}. \]
We determine Kim functions for the loci with refinement conditions $\Sigma = (1,1)$ and $\Sigma = (1,0)$. Kim functions for the other refinement conditions can be obtained from these by applying the $S_3$-action and using Eq.~\eqref{eq:S3-action-on-refined-loci}. 

\begin{thm}
	\label{thm:kim_functions_11}
	For $Z = \Spec \bZ[1/2q]$, the following functions belong to the motivic refined Chabauty--Kim ideal $\cI_{Z,4}^{\mot,(1,1)}$ for the refinement condition $\Sigma = (1,1)$:
	\begin{gather*}
		\log, \quad \Li_2, \quad \Li_4, \quad
		(\Li_{2,1} - \zeta^{\fu}(3)) \Li_{3,1} - (\Li_3 - \zeta^{\fu}(3))\Li_{2,1,1}.
	\end{gather*}
\end{thm}

\begin{proof}
	We continue to denote by $i_{\Sigma}$ the inclusion of the $\Sigma$-refined Selmer scheme into the unrefined motivic Selmer scheme. Restricting the universal cocycle evaluation map to $\Sel_{Z,4}^{(1,1)}(X)$ corresponds to setting $x_2$ and $x_q$ equal to zero in \Cref{univcoord}:
	\begin{align*}
		(\ev \circ i_{(1,1)})^{\sharp} \log &= 0,\\
		(\ev \circ i_{(1,1)})^{\sharp} \Li_1 &= a_{\tau_2} y_2 + a_{\tau_q} y_q,\\
		(\ev \circ i_{(1,1)})^{\sharp} \Li_2 &= 0,\\
		(\ev \circ i_{(1,1)})^{\sharp} \Li_3 &= a_{\sigma_3} z_3,\\
		(\ev \circ i_{(1,1)})^{\sharp} \Li_{2,1} &= a_{\sigma_3} z_{2,1},\\
		(\ev \circ i_{(1,1)})^{\sharp} \Li_4 &= 0,\\
		(\ev \circ i_{(1,1)})^{\sharp} \Li_{3,1} &= (a_{\sigma_3 \tau_2} y_2 + a_{\sigma_3 \tau_q} y_q) (z_3-1),\\
		(\ev \circ i_{(1,1)})^{\sharp} \Li_{2,1,1} &= (a_{\sigma_3 \tau_2} y_2 + a_{\sigma_3 \tau_q} y_q) (z_{2,1}-1).
	\end{align*}
	It is obvious for the first three functions and easily verified for the fourth that they pull back to zero along $\ev \circ i_{(1,1)}$.
\end{proof}

\begin{rem}
	\label{rem:kim_functions_11_independent_of_q}
	Observe that the Kim functions in \Cref{thm:kim_functions_11} do not depend on the prime~$q$. In light of \Cref{conj:Sigma-refined-Kims-conjecture}, this is consistent with the fact that the only $\{2,q\}$-integral point of~$X$ whose mod-2 and mod-$q$ reduction are both contained in $X \cup \{1\}$ is $z = -1$, independently of~$q$.
\end{rem}

\begin{thm}
	\label{thm:kim_functions_10}
	For $Z = \Spec \bZ[1/2q]$, the following functions belong to the motivic refined Chabauty--Kim ideal $\cI_{Z,4}^{\mot,(1,0)}$ for the refinement condition $\Sigma = (1,0)$:
	\begin{align}
		&a_{\tau_2} a_{\tau_q} \Li_2 - a_{\tau_q \tau_2} \log \Li_1,\\[2mm]
		&a_{\sigma_3} a_{\tau_2} a_{\tau_q}^3 \Li_4 - a_{\tau_2} a_{\tau_q}^2 a_{\tau_q \sigma_3} \log \Li_3 - \bigl(a_{\sigma_3} a_{\tau_q\tau_q\tau_q\tau_2} - a_{\tau_q\sigma_3} a_{\tau_q \tau_q \tau_2}\bigr) \log^3 \Li_1, \\[2mm]
		&a_{\tau_2}^2 a_{\tau_q}^2 a_{\sigma_3} \Li_{3,1} - a_{\tau_2} a_{\tau_q}^2 a_{\sigma_3 \tau_2} \Li_1(\Li_3 - a_{\sigma_3}) - a_{\tau_2}^2 a_{\tau_q} a_{\tau_q \sigma_3} \log \Li_{2,1} \\
		&\qquad-
		 (a_{\sigma_3}a_{\tau_q\tau_q\tau_2\tau_2} - a_{\sigma_3 \tau_2} a_{\tau_q\tau_q\tau_2} - a_{\tau_q\sigma_3} a_{\tau_q\tau_2\tau_2}) \log^2 \Li_1^2, \notag \\[2mm]
		&a_{\sigma_3} a_{\tau_2}^3 a_{\tau_q} \Li_{2,1,1} - a_{\sigma_3 \tau_2} a_{\tau_2}^2 a_{\tau_q} \Li_1 ( \Li_{2,1} - a_{\sigma_3}) \\
		&\qquad - (a_{\sigma_3} a_{\tau_q \tau_2 \tau_2 \tau_2} - a_{\sigma_3 \tau_2} a_{\tau_q \tau_2 \tau_2}) \log \Li_1^3. \notag
	\end{align}
\end{thm}

\begin{proof}
	Restricting the universal cocycle evaluation map to the refined Selmer scheme $\Sel_{Z,4}^{(1,0)}(X)$ corresponds to setting $x_2$ and $y_q$ equal to zero in \Cref{univcoord}:
	\begin{align*}
		(\ev \circ i_{(1,0)})^{\sharp} \log &= a_{\tau_q} x_q,\\
		(\ev \circ i_{(1,0)})^{\sharp} \Li_1 &= a_{\tau_2} y_2,\\
		(\ev \circ i_{(1,0)})^{\sharp} \Li_2 &= a_{\tau_q \tau_2} x_q y_2,\\
		(\ev \circ i_{(1,0)})^{\sharp} \Li_3 &= a_{\tau_q \tau_q \tau_2} x_q^2 y_2 + a_{\sigma_3} z_3,\\
		(\ev \circ i_{(1,0)})^{\sharp} \Li_{2,1} &= a_{\tau_q \tau_2 \tau_2} x_q y_2^2 + a_{\sigma_3} z_{2,1},\\
		(\ev \circ i_{(1,0)})^{\sharp} \Li_4 &= a_{\tau_q \tau_q \tau_q \tau_2} x_q^3 y_2 + a_{\tau_q \sigma_3} x_q z_3,\\
		(\ev \circ i_{(1,0)})^{\sharp} \Li_{3,1} &= a_{\tau_q \tau_q \tau_2 \tau_2} x_q^2 y_2^2 + a_{\tau_q \sigma_3} x_q z_{2,1} + a_{\sigma_3 \tau_2} y_2(z_3-1),\\
		(\ev \circ i_{(1,0)})^{\sharp}\Li_{2,1,1} &= a_{\tau_q \tau_2 \tau_2 \tau_2} x_q y_2^3 + a_{\sigma_3 \tau_2} y_2 (z_{2,1} - 1).
	\end{align*}
	It is clear that the function $a_{\tau_2} a_{\tau_q} \Li_2 - a_{\tau_q \tau_2} \log\cdot \Li_1$ pulls back to zero along the universal cocycle evaluation map, yielding the first function. To obtain the other three equations, start with $\Li_4$ (respectively, $\Li_{3,1}$ or $\Li_{2,1,1}$) and eliminate the variables $z_3$ and $z_{2,1}$ by subtracting multiples of $\Li_{3}$ and $\Li_{2,1}$. The result is a monomial in $x_q$ and $y_2$, which can be eliminated using $\log$ and $\Li_1$. 
\end{proof}

\begin{rem}
	\label{rem:S=2-refined}
	The first three functions in \Cref{thm:kim_functions_11} and the first two functions in \Cref{thm:kim_functions_10} only involve single polylogarithms and were first derived in \cite{luedtke:refined}, Theorems~3.2 and~3.4.
\end{rem}

\begin{rem}
	The refined level 4 Chabauty--Kim locus $X(\bZ_p)_{Z,4}^{\min}$ is finite for $Z = \Spec \bZ[1/S]$ with $\#S \leq 5$, so we could perform the geometric step for $S$ containing more than two primes. However, it would become tedious to deal with the growing number of different reduction types; moreover, the arithmetic step so far has only been carried out for $S \subseteq \{2,3\}$, so we would not be able to verify new cases of the refined Kim's Conjecture without some additional work.
\end{rem}

\subsection{The case \texorpdfstring{$Z = \Spec \bZ[1/6]$}{Z = Spec ℤ[1/6]}}
\label{sec:Z-one-sixth}

We now specialize to the case $Z = \Spec \bZ[1/6]$, i.e., $S = \{2,3\}$, where the set of $S$-integral points is given by the 21 elements
\[ X(\bZ[1/6]) = \left\{2, \frac12, -1, 3, \frac13, \frac23, \frac32, -\frac12,-2, 4, \frac14, \frac34, \frac43, -\frac13, -3, 9, \frac19, \frac89, \frac98, -\frac18, -8 \right\}. \]
(This essentially goes back to the book \emph{The Harmony of Numbers} from 1342 by Levi ben Gershon.)
For the arithmetic step (expressing the abstract motivic periods $a_u$ in terms of multiple polylogarithms) we can build on prior work by Corwin, Dan-Cohen, Jarossay, and others. The motivic periods depend on the choice of the Lie algebra generator $\sigma_3 \in \fn(Z)_{-3}$. It is shown in \cite[§4.3.2]{CDC:polylog1} that there is a unique choice of $\sigma_3$ satisfying 
\begin{equation*}
	\label{eq:sigma3-choice}
	\langle \sigma_3, \zeta^{\fu}(3) \rangle = 1 \quad \text{and} \quad \langle \sigma_3, \Li_3^{\fu}(-2) \rangle = \langle \sigma_3, \Li_3^{\fu}(3) \rangle = 0.
\end{equation*}

\begin{prop}
	\label{periods-Z-one-sixth}
	Let $Z = \Spec \bZ[1/6]$. With $\sigma_3$ chosen as above, the motivic periods appearing in \Cref{kim-function-S=2} and \Cref{thm:kim_functions_10} are given as follows:
	\begin{align*}
		a_{\tau_2} &= \log^{\fu}(2), \\
		a_{\tau_3} &= \log^{\fu}(3), \\
		a_{\tau_3 \tau_2} &= -\Li_2^{\fu}(3), \\ %
		a_{\sigma_3} &= \zeta^{\fu}(3),\\
		a_{\tau_3 \tau_2 \tau_2} &= -\Li_3^{\fu}(-2) + \logu(2) \Liu_2(-2) + \frac12 \logu(2)^2 \logu(3),\\
		a_{\tau_3 \tau_3 \tau_2} &= -\Li_3^{\fu}(3),\\
		a_{\sigma_3 \tau_2} &= \zeta^{\fu}(3) \log^{\fu}(2) - \frac87 \Li_4^{\fu}(2) = \frac47 \Li_{3,1}^{\fu}(-1),\\
		a_{\tau_3 \sigma_3} &= \frac{18}{13} \Li_4^{\fu}(3) - \frac{3}{52} \Li_4^{\fu}(9),\\
		a_{\tau_3 \tau_3 \tau_3 \tau_2} &= -\Li_4^{\fu}(3),\\
		a_{\tau_3\tau_3\tau_2\tau_2} &= -\Li_4^{\fu}\left(\frac{2}{3}\right)-\Li_4^{\fu}(-2)
		-\frac{8}{7}\left(\frac{\log^{\fu}(2)^4}{24}
		+\Li_4^{\fu}\left(\frac12\right)\right) \\
		&\qquad -\frac{5}{13}\Li^{\fu}_4(3)
		+\frac{3}{52}\Li^{\fu}_4(9)
		-\frac14\log^{\fu}(2)^2\log^{\fu}(3)^2
		-\log^{\fu}(2)\Li^{\fu}_3(3) \\
		&\qquad +\frac16\log^{\fu}(2)\log^{\fu}(3)^3
		-\frac1{24}\log^{\fu}(3)^4,\\
		a_{\tau_3\tau_2\tau_2\tau_2} &= \Li_4^{\fu}(-2) - \log^{\fu}(2)\Li_3^{\fu}(-2) + \frac12 \log^{\fu}(2)^2 \Li_2^{\fu}(-2) + \frac16 \log^{\fu}(2)^3\log^{\fu}(3).
	\end{align*}
\end{prop}

\begin{proof}
	The values of the single-letter periods $a_{\tau_2}$, $a_{\tau_3}$, $a_{\sigma_3}$ were already recorded in \Cref{single-letter-periods}. The value of the period $a_{\tau_3\tau_2}$ can be found using the $\{2,3\}$-integral point $3 \in X(\bZ[1/6])$; see \cite[§3.5]{BBKLMQSX} for the $p$-adic calculation, where the period is denoted $a_{3,2}$. The periods $a_{\tau_3\tau_3\tau_3\tau_2}$ and $a_{\tau_3\sigma_3}$ are computed in \cite[§4.3.3]{CDC:polylog1}; they are denoted $f_{\tau_2\tau\tau\tau}$ and $f_{\sigma\tau}$ there. 
	The remaining periods are computed in \cite[§5]{DCJ:M05}, using the same choice of $\sigma_3$. Note, however, that the values of $a_{\sigma_3\tau_2}$, $a_{\tau_3\tau_3\tau_3\tau_2}$ and $a_{\tau_3\tau_3\tau_2\tau_2}$ listed in \cite[§5.23]{DCJ:M05} (denoted $f_{\tau\sigma}$, $f_{\tau vvv}$ and $f_{\tau\tau v v}$ there) are incorrect. This is a consequence of an error in Proposition~5.8 of loc.\ cit. (the second formula should have $-8/7$ instead of $-7/8$) and an error at the end of §5.18 (the formula for $f_{2333}$ should have $-2E$ instead of $-E/2$). After fixing these errors one gets the values listed above. However, we have to supply an additional argument to deal with the fact that some of the rational coefficients in \cite[§5.23]{DCJ:M05} were only computed numerically up to high $p$-adic precision rather than derived exactly. In order to know that the coefficient of $\zeta^{\fu}(3)$ in the identity \cite[§5.16]{DCJ:M05}
	\[ \Li^{\fu}_3\left(\frac{2}{3}\right) = -\frac 1 2 \log^{\fu}(2)\log^{\fu}(3)^2 +\frac 1 6 \log^{\fu}(3)^3 -\Li^{\fu}_3(-2) - \Li^{\fu}_3(3) + \zeta^{\fu}(3) \]
	is exactly equal to $1$, one may argue as follows: the coefficient is known to be a rational number, so it suffices to prove that the identity holds (exactly) in $\bQ_p$ after applying the $p$-adic period map for a single prime~$p$ with $\zeta^p(3) \neq 0$, say $p = 5$. One may then invoke the trilogarithm identity \cite[Proposition~A.9.1(3)]{MTMUE}
	\[ \Li_3(z) + \Li_3(1-z) + \Li_3\left(\frac{z}{z-1}\right) - \frac16 \log(1-z)^3 + \frac12 \log(z) \log(1-z)^2 = \zeta(3), \]
	the $p$-adic version of the \emph{Landen functional equation} (see \cite{NS:landen-functional-equation} for its $\ell$-adic variant and \cite[pp.\ 112--118]{Landen1780} for its classical form). Substituting $z = -2$ yields the desired identity. We use a similar argument to get an exact proof of the identity \cite[§5.21]{DCJ:M05}
	\begin{equation}
		\label{eq:Li3-43}
		\Li_3^{\fu}\left(\frac43\right) + 4 \Li_3^{\fu}(-2) + 2\Li_3^{\fu}(3) + \log^{\fu}(2) \log^{\fu}(3)^2 - \frac16 \log^{\fu}(3)^3 = -\frac13 \zeta^{\fu}(3). 
	\end{equation}
	First apply the Landen functional equation with $z = 4$ to obtain
	\begin{equation}
		\label{eq:landen-4}
		\Li_3(4) + \Li_3(-3) + \Li_3\left(\frac43\right) - \frac16 \log(3)^3 + \log(2) \log(3)^2 = \zeta(3).
	\end{equation}
	The duplication formula (\cite[Proposition~6.1]{Coleman} with $m=2$)
	\[ \Li_n(z^2) = 2^{n-1}(\Li_n(z) + \Li_n(-z)) \]
	for $z = 2$ and $n = 3$, combined with the identity $\Li_3(2) = \frac78 \zeta(3)$ from Eq.~\eqref{eq:Li3-2-zeta} yields
	\begin{equation}
		\label{eq:Li3-4}
		\Li_3(4) = \frac72 \zeta(3) + 4 \Li_3(-2).
	\end{equation}
	We also have the identity
	\begin{equation}
		\label{eq:Li3-neg-3}
		\Li_3(-3) - 2\Li_3(3) = -\frac{13}{6} \zeta(3)
	\end{equation}
	from \cite[§6]{CDC:polylog1}. Substituting \eqref{eq:Li3-4} and \eqref{eq:Li3-neg-3} into \eqref{eq:landen-4} yields the desired identity \eqref{eq:Li3-43}.
\end{proof}

In the case of the refined Chabauty--Kim method, the polylogarithmic computations from \cite{luedtke:refined} were already sufficient to verify the refined Kim's Conjecture (\Cref{conj:refined-kim}) in level~$4$ for all primes $5 \leq p \leq 10{,}000$. We are now able to also verify Kim's original conjecture:

\begin{thm}
	\label{kims-conj-Z-one-sixth}
	Kim's Conjecture (\Cref{conj:kim}) for $Z = \Spec \bZ[1/6]$ holds in level~4, i.e., the inclusion
	\[ X(\bZ[1/6]) \subseteq X(\bZ_p)_{Z,4} \]
	is an equality, for all primes $5 \leq p < 1000$.
\end{thm}

\begin{proof}
	The functions $F_1^{\mot}$ and $F_2^{\mot}$ from \Cref{kim-function-S=2}, specialized to $Z = \Spec \bZ[1/6]$ with the period values from \Cref{periods-Z-one-sixth}, yield two explicit Kim functions 
	\[ F_1, F_2 \in \bQ_p[\log,\Li_1,\Li_2,\Li_3,\Li_{2,1},\Li_4,\Li_{3,1}, \Li_{2,1,1}] \]
	which vanish on $X(\bZ_p)_{Z,4}$. In order to compute their common zero locus we proceed as follows. We work residue disc by residue disc. Each disc in $X(\bZ_p)$ is parametrized by a parameter $t \in \bZ_p$ via $t \mapsto z \coloneqq \zeta + pt$, where $\zeta$ is the unique $(p-1)$st root of unity in the disc. Each of the multiple polylogarithm functions $\log,\Li_1,\ldots,\Li_{2,1,1}$ is represented by a convergent power series in the variable~$t$. For the single polylogarithms $\log,\Li_1,\ldots,\Li_4$ the computation of these power series is described in \cite[§4.3]{luedtke:refined} and we can reuse the existing Sage code. For the multiple polylogarithms, we use the differential equation $\rd\!\Li_{e_0 w}(z) = \Li_w(z) \frac{\rd z}{z}$, where $w$ is any word in the symbols $\{e_0,e_1\}$. If $\Li_w(\zeta + pt) = \sum_k a_{w,k} t^k$, then multiplying by $\frac{\rd z}{z} = \frac{p\, \rd t}{\zeta + pt}$ and integrating term by term yields
	\[ a_{e_0 w,k} = -\frac1{k} \sum_{j=0}^{k-1} \left(-\frac{p}{\zeta}\right)^{k-j} a_{w,j} \quad \text{for $k \geq 1$}, \]
	so the power series expansion for $\Li_w$ determines the one for $\Li_{e_0 w}$ up to the constant of integration $a_{e_0w,0} = \Li_{e_0 w}(\zeta)$. We take $w = e_1 e_1, e_0 e_1 e_1, e_1 e_1 e_1$, which gives
	\[ \rd\! \Li_{2,1}(z) = \frac12 \Li_1(z)^2 \frac{\rd z}{z} , \quad \rd\! \Li_{3,1}(z) = \Li_{2,1}(z) \frac{\rd z}{z}, \quad \rd\! \Li_{2,1,1}(z) = \frac16 \Li_1(z)^3 \frac{\rd z}{z}, \]
	so we only need the power series expansion of the single polylogarithm $\Li_1(z)$, which is easy. To get the constants of integration, we exploit the fact that all multiple polylogarithm values up to level~4 at some point $z \in X(\bZ_p)$ can be expressed in terms of single polylogarithms evaluated at $z$ and its $S_3$-transforms, cf.\ \cite[Eqs.~(B.6)--(B.8)]{BlackHolePerturbationTheory}:
	\begin{align*}
		\Li_{2,1}(z) &= \zeta(3) + \frac12 \log(z) \Li_1(z)^2 - \Li_1(z) \Li_2(1-z) - \Li_3(1-z),\\
		\Li_{3,1}(z) &= \Li_4(z) - \Li_4(1-z) + \Li_4\left(\frac{z}{z-1}\right) + \Li_1(z) \Li_3(z) \\
		&\qquad + \frac1{24} \Li_1(z)^4 + \frac16 \log(z) \Li_1(z)^3 - \zeta(3) \Li_1(z),\\
		\Li_{2,1,1}(z) &= \frac16 \log(z) \Li_1(z)^3 - \Li_4(1-z) - \Li_1(z) \Li_3(1-z) - \frac12 \Li_1(z)^2 \Li_2(1-z).
	\end{align*}
	Once we have the power series expansions of all the multiple polylogarithm functions, and hence those of $F_1$ and $F_2$, we compute the vanishing locus of $F_1$ via Hensel lifting (using the \texttt{Zproots} method from \url{https://github.com/martinluedtke/RefinedCK}) and check which of the roots are also roots of $F_2$. We verified that for all primes up to $1000$, the common roots are precisely the 21 points in $X(\bZ[1/6])$.	The SageMath code is contained in the file \texttt{kims\_conj\_Z\_one\_sixth.sage}, available at \url{https://github.com/martinluedtke/BeyondPolylog}.
\end{proof}

\section{Algorithms}%

The methods developed above allow us to construct an equation-solving algorithm based on Chabauty--Kim theory for $X = \Poneminusthreepoints$ which, in turn, leads to a strengthening of Proposition 1.2.2(2) of \cite{MTMUEII}. That is, the algorithm constructed in loc.\ cit.\ is only proved to halt if a particularly strong form of Kim's Conjecture holds --- one in which we restrict attention to the polylogarithmic quotient of the unipotent fundamental group of $X$ and then symmetrize the resulting ideal of Kim functions. We now know that it is too optimistic to hope for such a strong form of the conjecture, thanks to the counterexamples recently obtained in the number field setting \cite{LL:PolylogNF}. In the new algorithm we may replace this extra-strong form of Kim's Conjecture with its original, safer form. The precise statement appears in \Cref{30305A}. 

Most of the steps of the algorithm are unchanged, and we do not repeat them here in detail. In fact, the only significant change concerns the so-called \emph{geometric step}. The change is that this portion of the algorithm is no longer purely geometric, but includes arithmetic aspects as well, which now rely on relations between motivic multiple zeta values. Thankfully, these arithmetic issues, which intrude on the geometric step, are in fact quite minimal.

\subsection{Notation}%
\label{21123A}
Fix $Z = \Spec \ZZ \setminus S$. Recall from §\ref{sec:mixed-tate-galois-group} and §\ref{sec:hopf-algebra-action} that we denote by $U_Z$ the unipotent radical of the mixed Tate Galois group of $Z$, by $A(Z)$ its ring of functions, by $\fn(Z)$ its Lie algebra, and by $\cH(Z) = A(Z)^{\vee}$ the associated cocommutative Hopf algebra. In the following we write $U_Z$ as $U(Z)$ to be more consistent with \cite{MTMUEII}. Our first goal here is to review our construction of 
an algebra basis of $A(Z)$ consisting of motivic iterated integrals. 
For this, we need to recall some amount of notation.

We fix a positive integer $n$ throughout and introduce the following truncation by half-weight: let $\fn(Z)_{[\geq -n]}$ denote the quotient of $\fn(Z)$ by all elements of half-weight $ < -n$; let $U(Z)_{[\geq -n]}$ denote the associated quotient of $U(Z)$, let $A(Z)_{[\leq n]}$ denote its coordinate ring and $\cH(Z)_{[\geq -n]}$ the associated cocommutative Hopf algebra. Note that $A(Z)_{[\leq n]} \subseteq A(Z)$ is the $\bQ$-subalgebra generated by elements of half-weight $\leq n$; in particular, the two algebras agree in half-weight $\leq n$. 

In order not to cloud our formulas with the repeating decorations ``$[\le n]$'', ``$[ \ge -n]$'', after fixing $n$, we adopt the following notation and terminology:
\begin{align*}
A(Z) = A(Z)_{[\le n]}, \quad &\text{the \emph{$n$th (arithmetic) coordinate subalgebra};}
\\
U(Z) = U(Z)_{[\ge -n]}, \quad &\text{the \emph{$n$th (arithmetic) unipotent quotient};}
\\
\nN(Z) = \nN(Z)_{[\ge -n]}, \quad &\text{the \emph{$n$th (arithmetic) nilpotent quotient};}
\\
\Hh(Z) = \Hh(Z)_{{[\ge -n]}}, \quad &\text{the \emph{$n$th (arithmetic) complete quotient Hopf algebra}}.
\end{align*}
We let $\langle \cdot, \cdot \rangle$ denote the natural nondegenerate pairing
\[
\Hh(Z) \times A(Z) \to \QQ.
\]
Given $f \in A(Z)$ and  $v \in U(Z)(\QQ) \subseteq \cH(Z)$ grouplike, we have
\[
\qq v, f \pp = f(v).
\]

For $1 \leq i \leq n$, we let
\[
E_{i}(Z) = \Ext^1_Z \big(\QQ(0), \QQ(i)\big)
\subseteq A(Z)_{i}
\]
denote the space of extensions, and we let
\[
D_{i}(Z) = \m{Im} \Big(
\bigoplus_{\substack{j,k \ge 1,\\ j+k = i}}
A(Z)_j \otimes A(Z)_k \to A(Z)_{i}
\Big)
\]
denote the space of decomposables.

On the geometric side, we let $U(X)$ denote the level-$n$ quotient (equivalently: the half-weight $\geq -n$ quotient) of the unipotent de Rham fundamental group of $X$, i.e., the quotient of $\pi_1^{\dR}(X)$ by the $(n+1)$st step in the descending central series. Let $\nN(X)$ denote the Lie algebra of~$U(X)$, let $A(X)$ denote the ring of functions, and $\cH(X) = A(X)^{\vee}$ the cocommutative complete Hopf algebra associated to~$U(X)$. As in the case of the unipotent mixed Tate Galois group, we let $\qq \cdot, \cdot \pp$ denote the natural nondegenerate pairing 
\begin{equation*}
\Hh(X) \times A(X) \to \QQ.
\end{equation*}

\subsection{Combinatorially unramified motivic iterated integrals}
\label{21123Ba}
Our search for an algebra basis for $A(Z)$ goes through a collection of motivic iterated integrals that are \textit{obviously} unramified, as we now recall. Let $\om$ be a word of length $\leq n$ in differential forms $\frac{\rd t}{t-c_1}, \dots, \frac{\rd t}{t-c_r}$ with $c_i \in \QQ$, and let $a$, $b$ be basepoints on $\AA^1 \setminus \{c_1, \dots, c_r\}$ (i.e.\ $\QQ$-points, or tangential basepoints). For either type of basepoint $a$, let $a_0$ denote its \textit{location}: $a= a_0$ if $a$ is a rational point; otherwise $a$ is a tangent vector to $\PP^1$ \textit{at} $a_0 \in \PP^1(\QQ)$. We say that the datum $(a,\om,b)$ is \emph{combinatorially unramified} if it behaves nicely over $Z$ in an obvious sense, which requires several cases to state precisely. In all cases, we assume the divisor associated to $\om$, $a_0$, $b_0$, $\infty$ is \'etale over $Z$ (i.e., a union of disjoint sections). If $a$ is a tangent vector and $a_0 \neq b_0$, we assume the support of $0+a+\infty$ on $\PP^1 \supseteq \AA^1 \cong T_{a_0} \bP^1$ is \'etale over $Z$ and similarly for $b$. 
If $a$ and $b$ are tangent vectors at the same point $a_0 = b_0$, we assume the support of $0+a+b+\infty$ is \'etale over $Z$. These conditions assure that the unipotent motivic path torsor ${_b \Pi_a(\AA^1 \setminus{\{c_1, \dots, c_r\}})}$ belongs to the category of mixed Tate motives over $Z$. Generalizing the definition of motivic multiple polylogarithms and multiple zeta values (§\ref{sec:motivic-iterated-integrals}), which are motivic iterated integrals on $\bP^1 \smallsetminus \{0,1,\infty\}$, we can define the motivic iterated integral $I^{\fu}(a;\omega;b) \in A(Z)$ as the function $U(Z) \to \bA^1$,
\[ I^{\fu}(a;\omega;b) \colon u \mapsto \int_{u(\gamma_{a,b}^{\dR})} \omega, \]
where $\gamma_{a,b}^{\dR} \in {_b \Pi_a^{\dR}(\AA^1 \setminus{\{c_1, \dots, c_r\}})}(\bQ)$ is the \emph{canonical} (i.e., $\bG_m$-invariant) de Rham path from $a$ to~$b$ on $\AA^1 \setminus{\{c_1, \dots, c_r\}}$. The assumptions on $a,b, \omega$ are sufficient (but far from necessary) to ensure that $I^\uU(a; \om; b)$ belongs to $A(Z)$; we say that $I^\uU(a; \om; b)$ is \emph{combinatorially unramified over $Z$}, which is a strictly stronger condition than \emph{unramified}. We recall that the half-weight of $I^\uU(a; \om; b)$ is equal to the number of letters in $\om$.

\subsection{Exact basis algorithm}
\label{21123B}
The case $K = \QQ$ of \cite[\S3.8]{MTMUEII} provides an algorithm which, applied to the data $(Z, n)$, produces an open subscheme $Z^o \subseteq Z$, and for each $i \in [2,n]$, a linearly independent set $\Pp_{i} \subseteq A(Z^o)_{i}$ of combinatorially unramified motivic iterated integrals of half-weight $i$ such that $\Pp_{i}$ spans a vector subspace $P_{i} \subseteq A(Z^o)_{i}$ complementary to the (direct) sum $E_{i}(Z^o)+D_{i}(Z^o)$. We refer to this algorithm as the \emph{exact basis algorithm}. The exact basis algorithm only halts if such subsets $\Pp_2, \dots, \Pp_n \subseteq A(Z^o)$ exist for \textit{some} $Z^o \subseteq Z$. This is Goncharov's conjecture, refined with respect to ramification, as formulated in Conjecture 2.2.7 of \cite{MTMUEII}.\footnote{$p$-Adic approximation does intervene in the construction of this algorithm. Nevertheless, the output $\Pp_{i}$ is proved to obey the desiderata described above.} For $i = 1$ we set $\Pp_1 \coloneqq \emptyset$.

In the case $Z = \Spec \bZ$, Brown's theorem \cite{BrownMTZ} guarantees that the exact basis algorithm succeeds with $Z^o = \Spec(\bZ)$ and the output $\Pp_i(\Spec \bZ) \subseteq A(\Spec \bZ)_i$ consists entirely of multiple zeta values $\zeta^{\fu}(n_1,\ldots,n_r)$ of half-weight~$i$. For general $Z \subseteq \Spec \bZ$, the sets $\Pp_i(Z^o)$ satisfy 
\[ \Pp_i(\Spec \bZ) \subseteq \Pp_i(Z^o) \]
by the functoriality discussed in \cite[§5.6]{DCJ:M05}.

\subsection{Construction of associated bases}
\label{21123C}
The sets $\Pp_i$ may be used to construct an algebra basis for the $n$th coordinate subalgebra $A(Z^o)$, which can then be used to construct free generators (modulo the half-weight truncation) for the $n$th nilpotent quotients $\nN(Z^o)$ and $\nN(Z)$, as we now explain. In the case of the number field $\QQ$, the extension spaces $E_i$ for $i \in [1,n]$ have natural bases available: the motivic logarithms $\log^\uU(v)$ for $v \notin Z^o$ form a basis for $E_1(Z^o)$, and the motivic zeta values $\ze^\uU(i)$ for $i \in [3,n]$ odd form bases for the spaces $E_i$ (see Lemma \ref{motivic-zetas-nonzero}). For $i \in [2,n]$ even, the space $E_i$ is trivial and we can take the empty set as basis. Denote these bases $\Ee_i(Z^o)$ and set
\[
\Aa_i(Z^o) := \Ee_i(Z^o) \cup \Pp_i
\quad
\text{and}
\quad
\Aa(Z^o) = \bigcup_{i=1}^n \Aa_i(Z^o).
\]
Then $\Aa(Z^o)$ forms an algebra basis of the $n$th coordinate subalgebra $A(Z^o)$.

We define $\si_i \in \Hh(Z^o)_{-i}$ for $i \in [3,n]$ odd by requiring that $\si_i$ pair trivially with $P_i(Z^o) \oplus D_i(Z^o)$ and that 
\[
\langle \si_i, \ze^\uU(i) \rangle = 1.
\]
We also define $\tau_q \in \Hh(Z^o)_{-1}$ for $q \notin Z^o$ to be the element dual to $\log^\uU(q)$ with respect to the basis $\{\log^\uU(v)\}_{v \notin Z^o}$ of $A(Z^o)_1$. We then find that the elements $\tau_q$, $\si_i$ ($q \notin Z^o$, $i \in [3,n]$ odd) together form a set
\[
\ep(Z^o) = \bigcup_i \ep_i(Z^o)
\]
of free generators (modulo the half-weight truncation) of $\nN(Z^o)$. Moreover, the surjection
\[
\nN(Z^o)
\to
\nN(Z)
\]
sends $\tau_v$ for $v \in Z \setminus Z^o$ to $0$, and sends the remaining generators $\tau_q$, $\si_i$ to a set of free generators of $\nN(Z)$. The further projection $\nN(Z) \to \nN(\Spec \ZZ)$ provides a set
\[
\ep(\Spec \ZZ) = \{\si_3, \si_5, \dots \}
\]
of free generators of $\nN(\Spec \ZZ)$ (modulo the half-weight truncation). For the remainder of this section, we fix these generators ($\ep(\Spec \ZZ)$, $\ep(Z)$, $\ep(Z^o)$) by applying our algorithms.\footnote{In this way, the \textit{Exact basis algorithm}, and the ensuing construction of associated sets of generators, will play a dual role for us. One role is that of a recipe for computation. Another is that of a \textit{definition}. That is, we \textit{define} the set of generators $\ep(Z)$ to be the output of the algorithm. This sort of definition differs from the usual sorts of definitions in mathematics in a few ways. One way is significant, but not so unusual: it's conditional on conjectures (if the algorithm doesn't halt, it produces no output). The other ways are quite trivial: it's unusually long, and it involves many arbitrary choices which we have not made explicitly.}

\subsection{Approximate change of basis algorithm}
\label{30301A}
The algebra basis
$
\Aa(Z^o) 
$
for the $n$th coordinate subalgebra $A(Z^o)$ constructed in paragraph \ref{21123C} using our \textit{exact basis algorithm} gives rise to a vector space basis, namely the set $\Aa^\odot(Z^o)$ of monomials in $\Aa(Z^o)$. Meanwhile, for all $i \leq n$, the set $\ep(Z^o)$ of free generators for the $n$th nilpotent quotient $\nN(Z^o)$ constructed in the same paragraph gives rise to an associated \emph{shuffle basis} $\{a_{\omega}\}_{\omega}$ of $A(Z^o)_{i}$ indexed by words $\omega$ of half-weight~$-i$ in the set $\ep(Z^o)$ of generators. Section 3.9 of \cite{MTMUEII} provides an algorithm for computing the change of basis matrix up to given $p$-adic precision.

\subsection{Exact decomposition algorithm}
\label{30302A}
Brown's theorem \cite{BrownMTZ} combined with the sum total of all known relations between multiple zeta values modulo the ideal generated by $\ze(2)$ provides us, for each $i \in [1,n]$, with a $0$-sequence of finite-dimensional $\QQ$-\textit{vector spaces}
\[
\QQ^R \xrightarrow{M} \QQ^{\vec{\NN}} 
\xrightarrow{\zeta} A(\Spec \ZZ)_i
\]
with $\ze$ surjective. Here $\vec{\NN}$ may be viewed as an index set parametrizing all multiple zeta values of half-weight~$i$ and $R$ may be viewed as an index set parametrizing relations in half-weight~$i$. The basis algorithm (\ref{21123B}, \ref{21123C}) for $Z = \Spec \bZ$ provides us with an algebra basis $\Aa(\Spec \bZ)$ of $A(\Spec \bZ)$ consisting of multiple zeta values. A vector space basis of $A(\Spec \bZ)_i$ is then given by all monomials in these algebra generators, which can be expressed as linear combinations of MZVs using the shuffle relations~\eqref{eq:mzv-shuffle-relation}. We obtain a set $\Bb_i \subseteq \QQ^{\vec{\NN}}$ which maps to a basis of $A(\Spec \ZZ)_i$ for each $i$. The \textit{Exact decomposition algorithm} takes as further input a motivic multiple zeta value viewed as an element $\vec{n} \in \vec{\NN}$ and attempts to use the matrix $M$ to expand $\ze(\vec{n})$ in the basis $\ze(\Bb_i)$ using elementary linear algebra. 

There are many conjectures stating that one or another subset of $R$ spans all relations \cite{BachmannPoster}; any one of them would imply that $\ze$ is the cokernel of $M$, and hence that this attempt succeeds. 

\subsection{Action of mixed Tate Galois group}

We record a formula for the action of the mixed Tate Galois group on the unipotent fundamental group of $X$.

\label{21123F}
\begin{lemma}
\label{galois-action-lemma}
Let $i \in [3,n-1]$ be odd and let $\rho = \si_i$ be the element of half-weight $\opnm{hwt}(\rho) = -i$ of the set of free generators $\ep(\Spec \ZZ) = \{\si_3, \si_5 , \ldots \}$ for the $n$th nilpotent quotient $\nN(\Spec \ZZ)$. Recall that $e_0, e_1$ denote the elements (of half-weight $\opnm{hwt}(e_0) = \opnm{hwt}(e_1) = -1$) of $\nN(X)$ associated to the 1-forms $\frac{\rd z}{z}$ and $\frac{\rd z}{1-z}$, respectively. If $w$ is a word in the alphabet $e_0, e_1$ and $e \in \{e_0, e_1\}$ is a letter, we define the left-division $e\backslash w$ to be~$0$ if the leftmost letter in $w$ is not $e$, or the result of deleting the leftmost letter if it is. We define the right-division $w / e$ similarly. We artificially define $\zeu(0) = 0$. Then
\[
\rho e_1 = \sum_{\opnm{hwt}(w) = \opnm{hwt}(\rho) - 1}
\Big\qq
\rho, \zeu(e_1 \backslash w) - \zeu(w/e_1) 
\Big\pp w.
\]

\end{lemma}

\begin{proof}
For $k \leq n$, the words $w$ of length $k$  in the alphabet $\{e_0,e_1\}$ form a basis of $\cH(X)_{-k}$, and we use the induced inner product on the direct sum of these spaces.
By Theorem \ref{galois-action-lie-algebras}\ref{item:action on e1} and \Cref{lem:kummer-cocycle-via-mzv}, we have 
\begin{align*}
\qq w , \rho e_1 \pp 
&= \big\qq w, [e_1, \ka_1 (\rho)] \big\pp
\\
&= \left\qq
w, \sum_v \qq \rho ,\zeu(v) \pp (e_1v - ve_1) 
\right\pp
\\
&=
\left\qq
w
,
\sum_u 
\big\qq \rho, \zeu(e_1\backslash u) - \zeu(u/e_1) \big\pp
u 
\right\pp
\\
&=
\big\qq \rho, \zeu(e_1\backslash w) - \zeu(w/e_1) \big\pp. \qedhere
\end{align*}
\end{proof}

\subsection{Selmer matrix entries and Selmer coordinates}
\label{21123H}
Let
\[
Z^1_\Gm = \rZ^1\big(U(Z), U(X)\big)^\Gm
\]
be the motivic Selmer scheme associated to the level-$n$ quotient $U(X)$. Note that we can use the half-weight $\geq -n$ quotient $U(Z)$ here instead of the full motivic unipotent Galois group without changing the space of $\bG_m$-equivariant cocycles. If $R$ is a $\QQ$-algebra, then by an \emph{$R$-family} of $\Gm$-equivariant cocycles, we mean a $\Gm$-equivariant cocycle
\[
U(Z)_R \to U(X)_R;
\]
by a \emph{cocycle family}, we mean an $R$-family for some arbitrary $R$. We allow ourselves to drop $R$ from the notation when the generalization to families is evident.

As explained in \cite[§8]{Luedtke:motivic_selmer_scheme} and recalled in §\ref{sec:hopf-algebra-action}, every $\Gm$-equivariant group cocycle~$c$ extends uniquely to a $\Gm$-equivariant Hopf algebra cocycle
\[
\Hh(Z) \to \Hh(X)
\]
on the level of $n$th complete quotient Hopf algebras; the extension may be constructed as the linear dual of the map of $n$th coordinate subalgebras
\[
c^\sharp: A(Z) \from A(X).
\]
The extended cocycle may then be restricted to a $\Gm$-equivariant Lie algebra cocycle
\begin{equation*}
	\label{eq:Lie-cocycle}
\tag{$\nN$}
\nN(Z) \to \nN(X)
\end{equation*}
between $n$th nilpotent quotients. We denote the extension of $c$ to Hopf algebras, as well as its restriction to Lie algebras by the same symbol $c$. Similar comments apply to cocycle families.

Given $v \in \Hh(Z)$ and $f \in A(X)$, there's an associated \emph{Selmer matrix entry}
\[
\Phi^v_f: Z^1_\Gm \to \AA^1_\QQ
\]
given in terms of a cocycle family $c$ by
\[
\Phi^v_f(c ) =  \qq f, c v \pp.
\]
When $f = f_w$ is the linear functional dual to a word $w$ in $\{e_0, e_1\}$ of length $\leq n$, we abbreviate
\[
\Phi^v_{f_w} = \Phi^v_w,
\]
or even
\[
\Phi^v_{e_0} = \Phi^v_0, 
\quad
\Phi^v_{e_1} = \Phi^v_1, 
\quad
\Phi^v_{e_0 e_1} = \Phi^v_{01}, 
\quad
\text{ etc.} 
\]

We will also use a variant defined in terms of the restriction of cocycles to Lie algebras as in \eqref{eq:Lie-cocycle}: given $v \in \nN(Z)$ and $\la \in \nN(X)^\lor$, we have the associated \emph{Selmer Lie matrix entry}
\[
\Psi^v_\la: Z^1_\Gm \to \AA^1_\QQ
\]
given by
\[
\Psi^v_\la(c) = \qq \la, c v \pp.
\]
Recall that one may associate to the Lyndon words $w$ of length $i$ in $\{e_0, e_1\}$ certain Lie monomials $\la(w)$ of length $i$ which form a basis of $\nN(X)_{-i}$ (§\ref{sec:lyndon-basis}). We denote by $\la(w)^\lor$ the dual vector with respect to this basis and set
\[
\tag{SLME}
\Psi^v_w := \Psi^v_{\la(w)^\lor}.
\]
For example, the coordinates on the level-4 Selmer scheme $\rH^1(G_Z, \Pi_4^{\dR})$ from \Cref{ex:coordinates-depth-4} can be written as
\[
x_{\ell} = \Psi^{\tau_{\ell}}_0, \quad
y_{\ell} = \Psi^{\tau_{\ell}}_1, \quad
z_3 = \Psi^{\si_3}_{001}, \quad 
\text{and} \quad
z_{2,1} = \Psi^{\si_3}_{011}
\]
in this notation, abbreviating $e_0e_0e_1$ as $001$ etc.

\subsection{Decomposition of Selmer matrix entries}
\label{21202A}
According to \cite[Theorem~6.4]{Luedtke:motivic_selmer_scheme},
the Selmer Lie matrix entries
\[
\tag{SCF}
\Psi^\rho_w \quad
\text{with } \rho \in \ep(Z) \text{ a letter, and } w 
\text{ a Lyndon word in } \ep(X) = \{e_0, e_1\}
\]
with $\hwt(\rho) = \hwt(w) \geq -n$ form a $\QQ$-algebra basis of the coordinate ring $\Oo(Z^1_{\Gm})$. We refer to these as \emph{Selmer coordinate functions}. 
Thus, each Selmer matrix entry
\[
\tag{SMC}
\Phi^\om_w \quad
\text{with } \om 
\text{ a word in } \ep(Z), 
\text{ and } w 
\text{ a word in } \ep(X)
\]
with $\hwt(\om) = \hwt(w) \geq -n$, may be written as a polynomial in the Selmer coordinate functions (SCF), a process which we refer to as \emph{decomposition of Selmer matrix entries}. We refer to the negative of the common half-weight of $\om$ and $w$ also as the \emph{half-weight of the Selmer matrix entry $\Phi^\om_w$.}
We now construct an algorithm which carries out this process.\footnote{It's important for us that this process produces exact results and not merely $p$-adic approximations, and it is for this reason that we must rely on conjectures regarding relations between motivic multiple zeta values for halting.} Examples were worked out above in the proof of \Cref{univcoord}.

\subsubsection*{Cocycle reduction formula}
\label{crf}
For $\rho \in \ep(Z)$ a letter, $\om$ a word in $\ep(Z)$, and
\[
c: U(Z) \to U(X)
\]
a $\Gm$-equivariant cocycle for the $n$th unipotent quotients, we have
\begin{equation*}
	\label{eq: CRF}
\tag{CRF}
c(\rho \om) = \rho c(\om) + c(\rho)c(\om).
\end{equation*}
This is a special case of Eq.~\eqref{eq:cocycle-on-lie-element}.

\subsubsection*{Action reduction formula}
\label{arf}
For $\rho \in \ep(Z)$ a letter, and $v,w$ words in $\ep(X)$, we have
\begin{equation*}
	\label{eq: ARD}
\tag{ARD}
\rho(vw) = v(\rho w) + (\rho v)w.
\end{equation*}
This is a special case of Eq.~\eqref{eq:hopf-action} and simply expresses the fact that $\rho$, since it is of Lie type, acts as a derivation. 

\label{21202B}
We now turn to the construction of the algorithm. In terms of cocycle families, our first task may be rephrased as follows: fixing $w$ a word in $\ep(X) = \{e_0, e_1\}$, $\om$ a word of length $\operatorname{len}(\om) \ge 2$ in $\ep(Z)$, and $c$ a cocycle family, we wish to write
\[
\Phi^\om_w(c) = \qq w, c\om \pp
= \qq f_w, c\om \pp
 \in R
\]
as a polynomial in elements $\Phi^{\om'}_{w'}(c) \in R$ with 
\[
\operatorname{len}(\om') < \operatorname{len}(\om).
\] 

To do so, we write $\om = \rho \om'$ where $\rho$ is a letter. By applying the cocycle reduction formula \eqref{eq: CRF}, we reduce to computing
\[
\big\qq w, \rho c(\om') \big\pp
\quad
\text {and}
\quad
\big\qq w, c(\rho)c(\om') \big\pp.
\]
The latter is given by
\[
\big\qq w, c(\rho)c(\om') \big\pp
= 
\big\qq v, c(\rho) \big\pp
\big\qq v', c(\om') \big\pp
=
\Phi^{\rho}_{v}(c) \Phi^{\om'}_{v'}(c)
\]
where $w = v v'$ is the decomposition of~$w$ into subwords of half-weights $\hwt(v) = \hwt(\rho)$ and $\hwt(v') = \hwt(\omega')$.
For the former, we write 
\[
\rho c(\om') = \sum_{w'}
\Phi^{\om'}_{w'}(c) \rho w'
\]
by which we're reduced to computing the motivic Galois action $\rho w'$ on a finite collection of words $w'$ in $\ep(X)$ with $\hwt(w') = \hwt(\omega')$. For this, we apply the action reduction formula \eqref{eq: ARD} to decompose $w'$ into its constituent letters. We then combine the formula of \Cref{galois-action-lemma} with the exact decomposition algorithm \ref{30302A} to compute all needed coefficients of $\rho e_1$. 

We're thus reduced to the case of a Selmer matrix entry $\Phi^\rho_w$ with $\rho \in \ep(Z)$ a letter. The coefficient-extraction functional $f_w \in A(X)= \cH(X)^\lor$ restricted to $\nN(X) \subseteq \cH(X)$ may be expanded
in dual Lyndon basis elements $\la(w_i)^\lor \in \nN(X)^\lor$ as
\[
f_w|_{\nN(X)} = \sum_i f_w(\lambda(w_i)) \lambda(w_i)^{\vee},
\]
with $w_i$ running through all Lyndon words of length $|w|$,
and so we have
\[
\Phi^\rho_w = \sum_i f_w(\lambda(w_i)) \Psi^\rho_{w_i}.
\]
This completes the construction. 

\subsection{Geometric step}
\label{30302Z}
We now use the algorithms discussed above to construct an algorithm which we refer to as the \emph{geometric step}. Let $\Si=\Si_{s,n}$ be the union of the sets of symbols $\Si_1 =\{ \tau_1, \dots, \tau_s \}$, $\Si_2 = \emptyset$, $\Si_3 = \{\si_3\}, \ldots, \Si_n = \emptyset$ if $n$ is even and $\Si_n =\{\si_n\}$ if $n$ is odd. We regard $\Si$ as a graded set (i.e.\ a disjoint union indexed by the integers) by placing $\Si_i$ in graded degree $-i$. We refer to the graded degree also as the \emph{half-weight}. The shuffle algebra on $\Si$ is the free vector space on symbols $a_{\om}$ with $\om$ running through the words in the symbols $\Sigma$, with multiplication given by the shuffle product. This is graded by defining the half-weight of $a_{\omega}$ as $-\hwt(\om)$. We let $A(\Si)$ be the $\bQ$-subalgebra of this shuffle algebra generated by elements of half-weight $\leq n$, and let $K(\Si)$ denote the fraction field of $A(\Si)$. We denote by $U(\Sigma) = \Spec A(\Sigma)$ the unipotent group associated to $A(\Sigma)$.

We let $\nN(e_0, e_1) = \nN^{[n]}(e_0,e_1)$ denote the quotient of the free pro-nilpotent Lie algebra on $\{e_0, e_1\}$ by the $(n+1)$st step in the descending central series (the ``$n$th nilpotent quotient''), we let $U(e_0, e_1)$ be the associated unipotent group (the ``$n$th unipotent quotient''), and $A(e_0, e_1)$ its coordinate ring (the $n$th coordinate subalgebra). 

The \emph{geometric step} takes as input natural numbers $s$ and $n$ and outputs a list of elements
\[
F_1, \dots, F_m \in K(\Si_{s,n}) \otimes A(e_0, e_1).
\]

Before constructing the algorithm, we explain the meaning of its output. Let $Z$ be any dense open subscheme of $\Spec \ZZ$ with complement $S = \{q_1, \dots, q_s\}$. Let $A(Z) = A(Z)_{[\le n]}$ be the $n$th arithmetic coordinate subalgebra (\S\ref{21123A}), and let $K(Z)$ be the fraction field of $A(Z)$. We fix free generators (modulo the half-weight truncation) $\ep(Z) \subseteq \nN(Z)$ for the $n$th arithmetic nilpotent quotient using the construction of \S\ref{21123C}.\footnote{Let us emphasize that we here employ the \emph{Exact basis algorithm} (\ref{21123B}) applied to $Z$ in its role as a \emph{definition} only; in the construction of the \emph{Geometric step} we apply the Exact basis algorithm to $\Spec \ZZ$ only. 
}
Using these, we identify
\begin{equation*}
\tag{$\ast$} \label{eq:identification_AZ}
A(Z) \simeq A(\Si_{s,n}). 
\end{equation*}
We also have a canonical isomorphism of $n$th coordinate subalgebras $A(X) \simeq A(e_0,e_1)$. Let 
\[
\ev: U(Z) \times \rZ^1\big(U(Z), U(X)\big)^\Gm
\to U(Z) \times U(X)
\]
be the universal cocycle evaluation map, and let
\[
\ev_{K(Z)}: \rZ^1\big(U(Z), U(X)\big)^\Gm_{K(Z)}
\to U(X)_{K(Z)}
\]
be its base-change to $K(Z)$. Recall that the \emph{motivic Chabauty--Kim ideal}
\[
\cI^\m{mot}_Z= \cI^\m{mot}_{Z,n} \subseteq K(Z) \otimes_\QQ A(X)
\]
is defined as the kernel of the associated map of coordinate rings $\ev_{K(Z)}^\sharp$ (\Cref{def:motivic-kim-function}). Then in terms of the identification~\eqref{eq:identification_AZ}, the functions $F_1, \dots, F_m$ produced by the algorithm generate $\cI^\m{mot}_Z$.

We now turn to the construction of the algorithm.\footnote{We present the steps of the algorithm in their natural setting, and so accomplish two tasks at once: to define the algorithm and to show that the output upon halting is as stated.}  We begin by applying the basis algorithm (\ref{21123B}) to $\Spec \ZZ$.
The output allows us to fix an action of the $n$th unipotent quotient $U(\Spec \bZ) = U(\si_3, \si_5, \dots)$ (with generators $\sigma_i$ for $i \in [3,n]$ odd) on the $n$th unipotent quotient $U(e_0, e_1)$, and we let
\begin{equation*}
\tag{$\ast\ast$} \label{eq:abstract_action}
U(\Si_{s,n}) \quad \text{act on} \quad U(e_0, e_1)
\end{equation*}
via the quotient map
\[
U(\Si_{s,n}) \surj U(\si_3, \si_5, \dots).
\] 
This case is due to Brown \cite{BrownDecomp}. Using the orderings
\begin{equation*}
	\tag{${\ast}{\ast}{\ast}$} \label{eq:ordering}
\tau_1 < \cdots < \tau_s < \si_3 < \cdots
\hand
e_0 < e_1
\end{equation*}
we fix algebra generators $a_\la$ for $A(\Si)$ and $\Li_\mu$ for $A(e_0,e_1)$ (where $\la, \mu$ run over Lyndon words of half-weight $\geq -n$ with respect to the chosen orderings). The action~\eqref{eq:abstract_action} and the orderings~\eqref{eq:ordering} give rise to a map of affine spaces over $\QQ$
\[
\ev_\m{abs}:
U(\Si) \times \rZ^1\big(U(\Si), U(e_0, e_1)\big)^\Gm
\to
U(\Si) \times U(e_0, e_1)
\]
commuting with the projections to $U(\Sigma)$.\footnote{This \emph{abstract} universal cocycle evaluation map is not purely abstract, since it does depend on the action~\eqref{eq:abstract_action} in which lurk relations between motivic multiple zeta values.} The associated homomorphism $\ev_\m{abs}^\sharp$ of polynomial $A(\Sigma)$-algebras satisfies
\[
\ev_\m{abs}^\sharp(\Li_w) = 
\sum_\om  a_\om \Phi^\om_w
\]
for any word $w$ in $\{e_0, e_1\}$ of length $\leq n$, with $\omega$ running through the words in $\Sigma$ with $\hwt(\omega) = \hwt(w)$.
Using the \emph{Decomposition of Selmer matrix entries} (\ref{21202B}) we expand the images $\ev_\m{abs}^\sharp(\Li_\mu)$ for $\mu$ a Lyndon word of length $\leq n$ as polynomials in the Selmer coordinate functions $\Psi^\rho_w$ with $A(\Si)$-coefficients.
Working over the purely transcendental field $K(\Si)$, we may now apply standard methods from elimination theory to construct a set of generators $F_1, \dots, F_m$ of the kernel.\footnote{Experience shows that practical computations on present-day computers typically break down here, and more elaborate constructions, or more sophisticated numerical methods are needed.} This completes the construction.

\subsection{Conditional definition of exact analytic Kim functions}
\label{30303A}
Let us denote by $A_{\myCol}(X)$ the ring of Coleman functions on $X(\bZ_p)$. Just as the \textit{Exact basis algorithm} (\ref{21123B}) was used in \S\ref{21123C} to \textit{define} the set of generators $\ep(Z) \subseteq \nN(Z)$, we may use the \textit{Exact basis algorithm} and the \textit{Geometric step} (\ref{30302Z}) to define analytic Kim functions $F_1^\m{an}, \dots, F_m^\m{an} \in A_{\myCol}(X)$. This definition, like the previous one, is \textit{conditional} in the sense that it depends on the halting of the algorithms to which it refers. 
However, when these algorithms do halt, the resulting Kim functions are \textit{exact} and not merely an approximation.\footnote{Thus, our definition of the locally analytic functions $F_1^\m{an}, \dots, F_m^\m{an}$ involves algorithms along the way, but they are not in themselves the output of an algorithm. Indeed, because our \emph{Approximate change of basis algorithm} (\ref{30301A}) is merely approximate, they are not directly accessible to us computationally. This exact specification of particular analytic Kim functions alongside the actual computational steps of the algorithm which eventually outputs numerical approximations of them is an important point: we do not wish to approximate a map of polynomial algebras \textit{first}, and only \textit{then} look for generators of its kernel.} 

The definition goes as follows. Let $\ep(Z) \subseteq \nN(Z)$ be the set of generators for the $n$th nilpotent quotient obtained in \S\ref{21123C}. We thus have fixed isomorphisms
\[
K(\Si_{s,n}) = K(Z)
\hand
A(e_0, e_1) = A(X).
\]
Using these, we transport the functions $F_1, \dots, F_m$ produced in the \textit{Geometric step} (\ref{30302Z}) to $K(Z) \otimes A(X)$. After clearing denominators we can assume that the $F_i$ are contained in $A(Z) \otimes A(X)$. The $p$-adic period map $\per_p\colon A(Z) \to \bQ_p$ and the $p$-adic de Rham Kummer map $\kappa_p^{\dR}\colon X(\bZ_p) \to U(X)_{\bQ_p}$ together provide a homomorphism of $\QQ$-algebras 
\[
A(Z) \otimes_\QQ A(X) \to A_{\myCol}(X);
\]
we apply it to $F_1, \dots, F_m$ to get the promised analytic Kim functions $F_1^\m{an}, \dots, F_m^\m{an}$.

\subsection{Equation solving algorithm}
\label{30302B}
The remainder of our \emph{Equation solving algorithm} is the same as in \cite[\S7]{MTMUEII}; let us summarize the construction. We're given $Z \subseteq \Spec \ZZ$ as input. We select arbitrarily a prime $p \in Z$. Our algorithm searches through the set of all triples $(h, N, \ep)$, $h, N \in \NN$, $\ep$ in a countable subset of $\RR_{> 0}$ with accumulation point $0$. To each such triple, our algorithm assigns a set $X(Z)_h$ of points of $X(Z)$ and a boolean. If the boolean output is {\it True}, then we output $X(Z)_h$. If the boolean output is {\it False}, then we move on to the next triple $(h,N,\ep)$. To produce the sets $X(Z)_h$, we simply search naively for points; this can easily be arranged so as to guarantee that for \textit{some} $h$ we have $X(Z)_h = X(Z)$.

We turn to the boolean output. Partition $X(\Zp)$ into $\ep$-balls, decreasing $\ep$ as needed to ensure that each ball contains at most one element of the set $X(Z)_h$. Apply the \textit{Exact basis algorithm} to $(Z,N)$ to obtain bases and generators to be used in the subsequent algorithms.
Apply the \textit{Geometric step} with $s$ equal to the number of primes inverted in $Z$ and $n=N$ to obtain motivic Kim functions $F_1, \dots, F_m$ (\ref{30302Z}). Apply the \emph{Approximate change of basis algorithm} followed by numerical approximation of $p$-adic iterated integrals \cite[\S6]{MTMUEII} to obtain $\ep$-approximations $F_1^\ep, \dots, F_m^\ep$ of the associated analytic Kim functions $F_i^\m{an}$ from (\ref{30303A}). Use Newton polygons to check if with our $\ep$-approximations we can guarantee that in each $\ep$-ball there exists a function $F_i^\m{an}$ whose number of roots (counted with multiplicity) is no greater than the number (0 or 1) of points of $X(Z)_h$.\footnote{This requires an estimate on the coefficient valuations in the power series expansions of multiple polylogarithms, generalizing the estimates from \cite[Proposition~6.7]{Lip} for single polylogarithms. In fact, the same estimates hold for a \emph{multiple} polylogarithm of weight~$n$, except that the zeroth coefficient has valuation bounded from below by $(1-\frac1{p-1})n$ instead of~$n$, which follows from \cite[Theorem 5.3]{Chat17}.} Output \textit{True} or \textit{False}.

\subsection{Equation solving theorem}
\label{30303Aii}

In the present setting, \textit{Kim's Conjecture} states that for $N$ sufficiently large, the set $X(\Zp)_{Z,N}$ cut out by the ideal of analytic Kim functions associated to the quotient of the unipotent fundamental group by the $(N+1)$st step along the descending central series is equal to the set $X(Z)$ of $Z$-integral points with its reduced locally analytic subspace structure (\Cref{conj:kim}). \textit{Goncharov's conjecture refined with respect to ramification} \cite[Conjecture 2.2.7]{MTMUEII} states that for any open integer scheme $Z \subseteq \Spec \ZZ$ and any $n\in \NN$, there exists an open subscheme $Z^o \subseteq Z$ such that for every $1 \leq i \leq n$, the homogeneous component $A(Z^o)_i$ is spanned by combinatorially unramified motivic iterated integrals. The \textit{$p$-adic period conjecture} implies that the $p$-adic period map (restricted to the $n$th coordinate subalgebra)
\[
A(Z) \to \Qp
\]
is injective. In fact, in Theorem \ref{30305A}, we only need a small portion of this injectivity. 

As discussed in \S\ref{30302A}, there are many conjectures stating that some given set of relations between multiple zeta values generates all. Let us refer to the hope that at the very least, the union of all known relations generates all relations between motivic multiple zeta values as the \emph{Syncretist's motivic multiple zeta value conjecture}. When working, as we do, with unipotent motivic multiple zeta values à la Goncharov, we must add in the relations generated by `$\ze(2)=0$'. With the formulation of this conjecture, we've completed the task of making our statement of Theorem \ref{30305A} precise; its proof has been worked into our exposition of the \textit{Equation solving algorithm} (§\ref{30302B}).

\printbibliography

\vfill

\Small\textsc{
DC: Department of Mathematics, Ben-Gurion University.
}
\texttt{corwind@bgu.ac.il}

\medskip

\Small\textsc{
ID: Department of Mathematics, Ben-Gurion University.
}
\texttt{ishaida@bgu.ac.il}

\medskip

\Small\textsc{
ML: Carl von Ossietzky University Oldenburg and Department of Mathematics, Ben-Gurion University.
}
\texttt{martin.luedtke@uol.de}

\end{document}